\documentclass[a4paper]{article}

\usepackage[french,english]{babel}
\usepackage{subfiles}

\usepackage[toc]{appendix}

\usepackage[x11names]{xcolor}

\usepackage[margin=1.0in]{geometry}
\usepackage{parskip}
\usepackage[colorlinks=true,linkcolor=Blue4,citecolor=Orange1]{hyperref}
\usepackage{microtype}

\usepackage{amsmath,amssymb,amsfonts}
\usepackage{mathrsfs}
\usepackage{bm}

\usepackage{tikz-cd}
\tikzset{
    dsiml/.style={rotate=90, anchor=south},
    dsimr/.style={rotate=-90, anchor=south}
}

\usepackage{\subfix{thm}}

\usepackage[style=trad-alpha, backend=biber, doi=false, isbn=false, url=false, sorting=nyt]{biblatex}

\usepackage{\subfix{macros}}

\begin{document}

\selectlanguage{english}
\title{%
	A comparison of categories of Nori motivic sheaves
	\\
	{\small Comparaison entre catégories de faisceaux motiviques de Nori}%
}
\author{Emil Jacobsen, Luca Terenzi}
\date{}
\maketitle

\begin{abstract}
	We show that two different possible theories of Nori motivic sheaves, introduced by Ivorra--Morel and by Ayoub, respectively, are canonically equivalent.
	The proof of this result, which exploits the six functor formalism systematically, is based on the Tannakian theory of motivic local systems.
	As a consequence, we obtain a system of realization functors of Voevodsky motivic sheaves into Nori motivic sheaves compatible with the six operations, previously constructed by Tubach using different methods.
\end{abstract}

\selectlanguage{french}
\begin{abstract}
	Nous montrons que deux différentes théories possibles de faisceaux motiviques de Nori, respectivement introduites par Ivorra--Morel et par Ayoub, sont canoniquement équivalentes.
	La preuve de ce résultat, qui exploite systématiquement le formalisme des six opérations, repose sur la théorie Tannakienne des systèmes locaux motiviques.
	Comme conséquence, nous obtenons un système de foncteurs de réalisation des faisceaux motiviques de Voevodsky vers ceux de Nori compatible aux six opérations, précédemment construit par Tubach en utilisant des méthodes différentes.
\end{abstract}

\selectlanguage{english}

\tableofcontents

\section*{Introduction}
\addcontentsline{toc}{section}{Introduction}

\subsection*{Motivation and goal of the paper}

Let $k$ be a field of characteristic $0$ endowed with a complex embedding $\sigma \colon k \hookrightarrow \C$.
In the 1990s, Nori constructed an unconditional candidate $\MNori(k)$ for the conjectural $\Q$-linear abelian category of mixed motives over $k$:
by design, it is the finest possible abelian category of coefficients computing singular cohomology groups of $k$-varieties.
As shown by Nori, the abelian category $\MNori(k)$ carries a canonical tensor product making it neutral Tannakian over $\Q$;
its Tannaka dual is Nori's motivic Galois group.

The conjectural theory of mixed motives over $k$ should fit into a theory of mixed motivic sheaves over $k$-varieties giving rise to a universal six functor formalism.
Recently, two distinct approaches to promote Nori's theory of motives to a theory of motivic sheaves have been developed.
The first construction, initiated by Ivorra--Morel in~\cite{IM19} and completed by the second-named author in~\cite{TerenziNori}, generalizes the universal property of $\MNori(k)$ as an abelian category.
The second construction, described by Ayoub in~\cite{AyoAnab}, generalizes the presentation of Nori motives as representations of the motivic Galois group.
Both constructions build on the pre-existing theory of Voevodsky motivic sheaves over $k$-varieties:
in fact, they could be both regarded as a way to ``force'' the conjectural motivic perverse $t$-structure on Voevodsky's categories to exist.

Since the two approaches differ both at a conceptual and at a technical level, at first sight it is not clear at all whether they should define equivalent theories of Nori motivic sheaves. 
The main goal of the present paper is to prove that this is indeed the case.
The proof, which sheds new light on both constructions, is based on the Tannakian study of motivic local systems due to the first-named author in~\cite{Jacobsen}.

This comparison result allows us to combine the advantages of both approaches to Nori motivic sheaves.
An important application is the construction of realization functors of Voevodsky motivic sheaves into Ivorra--Morel's categories:
this result was originally obtained by Tubach in~\cite{Tub23};
using the features of Ayoub's construction, we offer an alternative proof which highlights the role of the motivic Galois group. 

\subsection*{History and previous work}

Before stating our results in a more precise form, we need to review the two proposed constructions of Nori motivic sheaves.
As mentioned above, both of them build on the theory of Voevodsky motivic sheaves and are motivated by conjectural properties of the latter.
In order to clarify this point, it is convenient to start by recalling Voevodsky's and Nori's approaches to mixed motives in parallel, stressing their interrelation.

For the purposes of the present paper, we are only interested in $\Q$-linear categories of motives.
The conjectural theory of \textit{Mixed Motives} over $k$ is expected to be characterized by two distinctive properties of different nature:
\begin{enumerate}
	\item[(i)] It should arise from the geometry of algebraic cycles on $k$-varieties.
	\item[(ii)] It should take the shape of an abelian category of coefficients defining the universal cohomology theory on $k$-varieties.
\end{enumerate}
More precisely, categorical invariants such as morphisms and extensions between mixed motives should be computable in terms of $K$-theory groups.
The abelian category of mixed motives is expected to carry a monoidal structure making it neutral Tannakian over $\Q$, and also a theory of weights refining the classical Hodge-theoretic notion.
From a conceptual viewpoint, it is reasonable to require the two properties (i) and (ii) together:
the only known procedure to construct universal cohomology classes uniformly within different cohomology theories is by taking the image of algebraic cycles under the various cycle class maps.
From a practical viewpoint, however, it is not clear how to impose (i) and (ii) simultaneously.
In fact, the expectation that algebraic cycles should give rise to an abelian category of coefficients led Grothendieck to formulate the so-called Standard Conjectures, which remain widely open to date.

Anyway, there exist two unconditional approaches to construct categories of mixed motives over $k$.
Roughly speaking, each of them is only known to satisfy one of the two key properties stated above:
\begin{enumerate}
	\item[(i)] Voevodsky's approach (pursued independently also by Hanamura and by Levine) yields a triangulated category $\DAct(k)$ having the correct relation to algebraic cycles, but not known to arise from an abelian category.
	\item[(ii)] Nori's approach yields an abelian category $\MNori(k)$ enjoying a suitable universal property, but not known to be computable in terms of algebraic cycles.
\end{enumerate}
In lack of a complete theory of mixed motives, the two candidate theories represent its closest approximations from the geometric side of algebraic cycles and from the linear-algebraic side of cohomology, respectively.
It is worth stressing that Voevodsky's and Nori's constructions look quite different:
while the first is purely cycle-theoretic and does not depend on the existence of any classical cohomology theory, the second requires the choice of one classical cohomology theory as a model for the category of motives.
However, they are believed to be essentially equivalent:
Nori's category $\MNori(k)$ should be the correct abelian category of mixed motives, while Voevodsky's category $\DAct(k)$ should be its bounded derived category.
In particular, the triangulated category $\DAct(k)$ should carry a \textit{motivic $t$-structure} with heart equivalent to $\MNori(k)$.

The existence of the motivic $t$-structure is an extremely deep conjecture: 
it would imply a large part of the Standard Conjectures, as well as the Beilinson--Soulé Vanishing Conjecture in Algebraic $K$-theory.
In any case, by work of Nori, there exists a canonical realization functor
\begin{equation}\label{formula:Nri_k^*-intro}
	\Nri_k^* \colon \DAct(k) \rightarrow \Db(\MNori(k))
\end{equation}
compatible with the monoidal structures.
This leads to a strong comparison result on the Tannakian level:
as proved by Choudhury--Gallauer in~\cite{CG17}, Nori's motivic Galois group $\GmotNo(k)$ is canonically isomorphic to Ayoub's motivic Galois group $\GmotAy(k)$, constructed directly out of Voevodsky motives.
So far, this result is the stronger piece of evidence towards the conjectural equivalence.

Classical cohomology theories arise from suitable theories of sheaves over $k$-varieties.
Typically, these give rise to a system of triangulated categories of coefficients related by Grothendieck's six functor formalism.
By analogy, the conjectural theory of Mixed Motives should extend to a theory of \textit{Mixed Motivic Sheaves} defining the universal six functor formalism on $k$-varieties.
The precise formulation of this expectation can be found in Beilinson's article~\cite{Bei87H}, where the abelian categories of mixed motivic sheaves are modelled on perverse sheaves rather than on ordinary sheaves. 
The insight that perverse sheaves should provide the natural model for refined abelian categories of coefficients, justified by the formalism of cohomological weights, lies at the heart of Saito's theory of mixed Hodge modules.
A crucial foundational result is Beilinson's equivalence~\cite[Thm.~1.3]{BeilinsonEquivalence}, which allows one to regard the usual constructible derived category as the derived category of perverse sheaves. 

From a technical viewpoint, the construction of the six operations in the motivic world turns out to be easiest in the setting of Voevodsky motives.
Based on seminal ideas of Voevodsky, this subject has been studied thoroughly by Ayoub in~\cite{AyoThesis} and by Cisinski--Déglise in~\cite{CD19}:
the outcome is a system of triangulated categories $X \mapsto \DAct(X)$ related by the six operations.
As shown by Drew--Gallauer in~\cite{DG22}, this really defines the universal six functor formalism on $k$-varieties, provided such a property is stated in the language of $\infty$-categories.
By analogy with the case of mixed motives, it is natural to conjecture the existence of a \textit{motivic perverse $t$-structure} on $\DAct(X)$.
In order to construct a candidate for the conjectural perverse heart unconditionally, one can try to extend Nori's approach to mixed motives over general bases.
In fact, this can be done in two different ways, according to the feature of Nori's category that one wants to generalize:
either its universal property as an abelian category or its presentation as representations of the motivic Galois group.
In both cases, the construction requires the choice of a complex embedding $\sigma \colon k \hookrightarrow \C$, as for Nori's original theory.
Using the associated analytification of $k$-varieties, one defines a system of Betti realization functors
\begin{equation*}
	\Bti_X^* \colon \DAct(X) \to \Dbct(X)
\end{equation*}
into the usual constructible derived categories.
By stabilizing the image of $\Bti_X^*$ inside $\Dbct(X)$ appropriately, one obtains a well-behaved triangulated category  $\Dbgeo(X)$ of \textit{constructible complexes of geometric origin}.
As in~\cite{BBDG}, one can define a perverse $t$-structure on the latter:
its heart $\Pervgeo(X)$ is the category of \textit{perverse sheaves of geometric origin}, and the abelian category of motivic perverse sheaves should be a motivic enhancement of the latter.

Ivorra--Morel's approach to Nori motivic sheaves puts the accent on the universal property:
their category $\MNori(X)$ is defined as the universal abelian category factoring the homological functor
\begin{equation*}
	\DAct(X) \xrightarrow{\Bti_X^*} \Dbgeo(X) \xrightarrow{\pH^0} \Pervgeo(X).
\end{equation*}
Using Nori's realization functor \eqref{formula:Nri_k^*-intro}, one sees that the definition recovers $\MNori(k)$ when $X = \Spec(k)$.
By~\cite{IM19} and~\cite{TerenziNori}, as $X$ varies among quasi-projective $k$-varieties, the derived categories $\Db(\MNori(X))$ are endowed with the six operations, compatibly with those on the underlying constructible derived categories.
The construction of this six functor formalism is long and technical, since one has to reconstruct the six operations starting from $t$-exact functors and natural transformations thereof.
Ayoub's approach to Nori motivic sheaves puts the accent on the motivic Galois group:
the group $\GmotAy(k)$ acts naturally on the triangulated categories $\Dbgeo(X)$, and one is led to consider the associated categories of homotopy-fixed points
\begin{equation*}
	\Dbgeo(X)^{\GmotAy(k)}.
\end{equation*}
Using Choudhury--Gallauer's isomorphism, one sees that the definition recovers $\Db(\MNori(k))$ when $X = \Spec(k)$.
Over a general $k$-variety $X$, one can interpret $\Dbgeo(X)^{\GmotAy(k)}$ informally as a category of $\MNori(k)$-valued constructible complexes.
Since the $\GmotAy(k)$-action is compatible with the six operations, these pass automatically to the homotopy-fixed points, and they even exist $\infty$-categorically.
In conclusion, there are two natural ways to extend Nori's theory of motives to a six functor formalism, and they enjoy distinct advantages:
the universal property is only visible in Ivorra--Morel's setting, whereas the functoriality is much cleaner in Ayoub's setting.

\subsection*{Main results}

Our main result is the following comparison theorem:
\begin{thm*}[Thm.~\ref{thm:comp}, Prop.~\ref{prop:o_S-6ff}]
	For every $k$-variety $X$, there exists a canonical equivalence of triangulated categories
	\begin{equation*}
		\Db(\MNori(X)) \xrightarrow{\sim} \Dbgeo(X)^{\GmotAy(k)}.
	\end{equation*}
	As $X$ varies, these are compatible with the six operations.
\end{thm*}

Let us give an overview of the proof strategy.
This is an intentionally simplified account:
it does not touch certain technical questions about perverse sheaves of geometric origin that we need to solve along the way.

If the above comparison result is to hold, there should be a way to carve out the abelian category $\MNori(X)$ inside the triangulated category $\Dbgeo(X)^{\GmotAy(k)}$ as the heart of a suitable perverse $t$-structure, and the corresponding realization functor should identify $\Dbgeo(X)^{\GmotAy(k)}$ with the derived category of its perverse heart.
We check that this is indeed the case:
\begin{prop*}[Prop.~\ref{prop:Perv^Gmot}, Thm.~\ref{thm:Bei-equiv}]
	For every $k$-variety $X$, the stable $\infty$-category $\Dbgeo(X)^{\GmotAy(k)}$ carries a $t$-structure compatible with the perverse $t$-structure on $\Dbgeo(X)$.
	Its heart is equivalent to $\Pervgeo(X)^{\GmotAy(k)}$, and the associated realization functor
	\begin{equation*}
		\Db(\Pervgeo(X)^{\GmotAy(k)}) \rightarrow \Dbgeo(X)^{\GmotAy(k)}
	\end{equation*}
	is an equivalence. 
\end{prop*}
The existence of the perverse $t$-structure on $\Dbgeo(X)^{\GmotAy(k)}$ essentially follows from the $t$-exactness of the $\GmotAy(k)$-action on $\Dbgeo(X)$.
The final equivalence is a $\GmotAy(k)$-equivariant version of Beilinson's classical result, and the proof is based on Beilinson's argument~\cite[\S~2]{BeilinsonEquivalence}.

Once this preliminary task is completed, we construct a canonical system of exact functors
\begin{equation*}
	o_X \colon \MNori(X) \rightarrow \Pervgeo(X)^{\GmotAy(k)}
\end{equation*}
by exploiting the universal property of Ivorra--Morel's abelian categories.
Passing to the derived categories, we obtain the triangulated comparison functors
\begin{equation*}
	\Db(\MNori(X)) \xrightarrow{o_X} \Db(\Pervgeo(X)^{\GmotAy(k)}) \xrightarrow{\sim} \Dbgeo(X)^{\GmotAy(k)}.
\end{equation*}
Their compatibility with the six operations is a formal consequence of the way these operations are constructed in Ivorra--Morel's setting.
In particular, one does not need to know in advance that the comparison functors are equivalences in order to know their compatibility with the six operations;
in fact, it is the latter result which plays a key role in the proof of the former.

Arguing by Noetherian induction via localization triangles, the comparison statement for Nori motivic sheaves reduces its generic variant:
\begin{prop*}[Prop.~\ref{prop:comp_generic}]
	For every $k$-variety $X$, the functor
	\begin{equation*}
		\twocolim_{U \in \Open_X^{op}} \Db(\MNori(U)) \to \twocolim_{U \in \Open_X^{op}} \Dbgeo(U)^{\GmotAy(k)}
	\end{equation*}
	induced in the colimit over the dense open subsets $U \subset X$ is an equivalence.
\end{prop*}
The main advantage of this generic statement is that it is amenable to Tannakian methods.
Indeed, if $X$ is a smooth $k$-variety, ordinary local systems over $X$ embed into perverse sheaves after shifting appropriately inside $\Dbct(X)$;
we have a similar inclusion $\Locgeo(X) \subset \Pervgeo(X)$ inside $\Dbgeo(X)$.
We are led to consider the full abelian subcategory $\NMLoc(X) \subset \MNori(X)$ spanned by the object whose underlying perverse sheaf belongs to $\Locgeo(X)$.
Our comparison functor restricts to an exact monoidal functor
\begin{equation*}
	o_X \colon \NMLoc(X) \to \Locgeo(X)^{\GmotAy(k)}.
\end{equation*}
Since every perverse sheaf over $X$ restricts to a shifted local system over some dense open subset, it suffices to show that these functors induce an equivalence in the colimit.
If $X$ is geometrically connected over $k$, the two categories involved are neutral Tannakian over $\Q$, and the question can be translated in terms of their Tannaka dual groups. 
In the case when $X = \Spec(k)$, everything comes down to Choudhury--Gallauer's isomorphism between $\GmotNo(k)$ and $\GmotAy(k)$.
In general, one needs to measure the difference between the motivic Galois group of $X$ and that of $k$.
We do this as follows:
\begin{thm*}[Thm.~\ref{thm:fund_ses_general}, Thm.~\ref{thm:ayoub_fund_seq}]
	Let $X$ be a smooth, geometrically connected $k$-variety.
	Then, for every closed point $x \in X(\bar{k})$, the following statements hold:
	\begin{enumerate}
		\item The Tannaka dual $\GmotNo(X,x)$ of $\NMLoc(X)$ fits into the fundamental exact sequence
		\begin{equation*}
			1 \to \pi_1^{\geo}(X,x) \to \GmotNo(X,x) \to \GmotNo(k) \to 1.
		\end{equation*}
		\item The Tannaka dual $\GmotAy(X,x)$ of $\Locgeo(X)^{\GmotAy(k)}$ fits into the fundamental exact sequence
		\begin{equation*}
			1 \to \pi_1^{\geo}(X,x) \to \GmotAy(X,x) \to \GmotAy(k) \to 1.
		\end{equation*}
	\end{enumerate}
\end{thm*}
Here, the group $\pi_1^{\geo}(X,x)$ is Tannaka dual to $\Locgeo(X)$:
it should be regarded as the closest motivic approximation to the usual topological fundamental group.
The proof of exactness of the fundamental sequences consists of two steps:
first we treat the case when $x$ is a $k$-rational point (in which case the two sequences are even split), then we deduce the general case by a Galois-descent argument.
In the $k$-rational case, the two exact sequences are established in different ways:
the first one was obtained by the first-named author in~\cite{Jacobsen} using Hodge-theoretic inputs, while the second one is obtained in the present paper using abstract properties of equivariant Tannakian categories.

At this point, the reader might observe that the exactness of the two fundamental sequences would already imply that the comparison functor $o_X \colon \MNori(X) \to \Locgeo(X)^{\GmotAy(k)}$ is an equivalence, at least when $X$ is geometrically connected over $k$.
The issue is that, in the main body of the paper, we have to consider two distinct notions of local system of geometric origin:
while Ayoub's category $\Locgeo(X)$ is tailored to the second fundamental sequence, the natural category of local systems to use for the first one is a priori smaller.
In any case, we can show that the possible difference between these two categories of local systems disappears in the colimit, which allows us to prove the generic equivalence.
Once the full comparison theorem is known, we deduce that the two notions of local system of geometric origin must in fact coincide.
We obtain a similar result for perverse sheaves of geometric origin.

Part of the technical difficulty of the paper comes from the necessity to understand how the various categories involved behave under finite extensions of the base field $k$:
this is crucial, for example, to apply Galois-descent ideas in the study of the fundamental sequences.
In particular, we have to prove some invariance properties of Ayoub's categories of motivic local systems with respect to the chosen base field $k$, which are already known to hold for Ivorra--Morel's categories.
To this end, we establish some auxiliary short exact sequences of Tannaka groups involving Artin motivic sheaves, generalizing results by Nori and Ayoub.
The reader may safely skip these additional passages in a first reading, by assuming $k$ algebraically closed throughout the entire paper. 

As an application of our comparison theorem, we recover one of the main results of Tubach's paper~\cite{Tub23}, in which Nori's realization functor~\eqref{formula:Nri_k^*-intro} is extended to motivic sheaves:  
\begin{thm*}[Thm.~\ref{thm:infty}, Thm.~\ref{thm:real}, Prop.~\ref{prop:Nri-weights}]
	The six functor formalism $X \mapsto \Db(\MNori(X))$ admits a canonical $\infty$-categorical enhancement.
	Consequently, for every $k$-variety $X$ there exists a canonical realization functor
	\begin{equation}\label{formula:Nri_X^*-intro}
		\Nri_X^* \colon \DAct(X) \rightarrow \Db(\MNori(X))
	\end{equation}
	such that the Betti realization over $X$ factors as
	\begin{equation*}
		\Bti_X^* \colon \DAct(X) \xrightarrow{\Nri_X^*} \Db(\MNori(X)) \xrightarrow{\iota_X} \Db(\Pervgeo(X)) \xrightarrow{\sim} \Dbgeo(X).
	\end{equation*}
	As $X$ varies, the functors~\eqref{formula:Nri_X^*-intro} are compatible with the six operations and respect weights.
\end{thm*}

The existence and essential uniqueness of the realization functors~\eqref{formula:Nri_X^*-intro}, as well as their compatibility with the six operations, follow at once from the $\infty$-categorical universal property of the six functor formalism $X \mapsto \DAct(X)$.
The same universal property implies that all self-equivalences of Voevodsky motivic sheaves as a six functor formalism are trivial.
Exploiting the realization functors~\eqref{formula:Nri_X^*-intro}, we obtain a partial analogue of this property for Nori motivic sheaves:
\begin{thm*}[Thm.~\ref{thm:exact-Auteq}]
	All self-equivalences of the six functor formalism $X \mapsto \Db(\MNori(X))$ respecting the perverse $t$-structures are trivial.
\end{thm*}
It is natural to expect that every self-equivalence of Nori motivic sheaves be $t$-exact.
Unfortunately, this looks like a deeper question, related to the existence of the motivic perverse $t$-structure on $\DAct(X)$.
In fact, as proved by Tubach in~\cite[Thm.~4.11]{Tub23}, the existence of the motivic perverse $t$-structures over all $k$-varieties would imply that all realization functors~\eqref{formula:Nri_X^*-intro} are equivalences.
This would allow one to rephrase our comparison theorem in terms of Voevodsky motivic sheaves.
It would be interesting to have some unconditional analogue of our result purely in the setting of Voevodsky motives.

One last by-product of our comparison theorem is the following independence property of Ayoub's construction of Nori motivic sheaves:
\begin{thm*}[Thm.~\ref{thm:Dbgeo^Gmot_sigma1=Dbgeo^Gmot_sigma2}]
	Fix two distinct complex embeddings $\sigma_1, \sigma_2 \colon k \hookrightarrow \C$.
	Then, for every $k$-variety $X$, the $\infty$-categories of homotopy-fixed points $\Dbgeo(X)^{\GmotAy(k)}$ constructed via $\sigma_1$ and $\sigma_2$ are canonically equivalent, compatibly with the six operations.
\end{thm*}
This result is non-trivial because the six functor formalism $X \mapsto \Dbgeo(X)$ does depend on the chosen complex embedding.
The key point is that Ivorra--Morel's categories are already known to be independent of all auxiliary choices, as it should be the case for categories of motivic sheaves.

\subsection*{Related and future work}

The $\infty$-categorical enhancement of Ivorra--Morel's theory of Nori motivic sheaves, and the consequent realization from Voevodsky motivic sheaves, have been obtained previously by Tubach in~\cite{Tub23} using different methods.
Tubach's methods also apply to other six functor formalisms endowed with ordinary $t$-structures, such as Saito's mixed Hodge modules.
In fact, the Hodge-theoretic version of Tubach's result plays a role in the construction of the fundamental short exact sequence by the first-named author.

Compared to Tubach's approach to Nori motivic sheaves, ours has two main advantages:
it does not resort to the ordinary $t$-structure, and it yields a more explicit description of the realization functors from Voevodsky motivic sheaves.
In a forthcoming project, we plan to study the Hodge-theoretic variant of Ayoub's construction of Nori motivic sheaves as homotopy-fixed points:
this will yield a clean Tannakian description of the six functor formalism on mixed Hodge modules of geometric origin.
We will then recover the Hodge realization functor of Nori motivic sheaves constructed in~\cite{Tub23} as a natural restriction functor on homotopy-fixed points.

\subsection*{Structure of the paper}

In Section~\ref{sect:Dbgeo}, after reviewing the formalism of Voevodsky motivic sheaves, we explain how they relate to sheaves of geometric origin through the Betti realization.
Along the way, we study the compatibility of the theory with finite extensions of the base field;
since these technical results play a purely auxiliary role in the paper, the proofs can be safely skipped in a first reading.

In Section~\ref{sect:IM_cats}, we review Ivorra--Morel's theory of Nori motivic sheaves in some detail.
After explaining the definition of these categories in terms of Voevodsky motivic sheaves, we summarize the construction of the six operations.
In the final part, we focus on motivic local systems, and we discuss their relation to local systems of geometric origin through the fundamental exact sequence.

The goal of Section~\ref{sect:fund_ses_Voe} is to derive the same fundamental sequence in the setting of Voevodsky motivic sheaves.
The main constructions rely on the formalism of equivariant objects under pro-algebraic groups, discussed in Appendix~\ref{sect:App}.

In Section~\ref{sect:ayoub_cat}, following Ayoub, we study categories of equivariant complexes of geometric origin under the motivic Galois group.
The main result is the construction of a perverse $t$-structure satisfying the equivariant version of Beilinson's equivalence.

Section~\ref{sect:comp-thm} is devoted to our main result, in which we compare Ivorra--Morel's categories of Nori motives with Ayoub's equivariant categories.
The proof combines the main results of the previous three sections.  

In the final Section~\ref{sect:appli_compl}, we apply our comparison theorem to the construction of realization functors from Voevodsky motivic sheaves to Nori motivic sheaves.
Other interesting consequences of our main result are discussed as well.

\subsection*{Acknowledgments}

We are deeply grateful to Joseph Ayoub for his interest on this project and for his illuminating answers to our technical questions.
In particular, we thank him for explaining to us that the natural way to prove Lemma~\ref{lem::AMLoc_X'/X} is through the formalism of induced representations.
We would also like to thank Marco D'Addezio, Sophie Morel, Dan Petersen, and Swann Tubach for useful conversations.

The second-named author was funded by the LabEx MILYON at ENS de Lyon.
The first-named author was funded by
Dan Petersen's Wallenberg Scholar fellowship.

\section*{Notation and conventions}
\addcontentsline{toc}{section}{Notation and conventions}

\subsection*{Algebraic Geometry}

\begin{itemize}
	\item We work over a base field $k$ of characteristic $0$ endowed with a complex embedding $\sigma\colon k \hookrightarrow \C$.
	\item By a \textit{$k$-variety} we mean a reduced, separated $k$-scheme of finite type.
	\item By a \textit{morphism of $k$-varieties} we mean a morphism of $k$-schemes.
	\item Given a $k$-variety $X$ and an open subvariety $U \subset X$, by the \textit{closed complement} of $U$ in $X$ we mean the closed subspace $X \setminus U$ endowed with the reduced scheme structure.
	\item Given a $k$-variety $X$, we let $\Open_X$ denote the poset of Zariski dense open subsets of $X$, ordered by inclusion.
\end{itemize}

\subsection*{Analytic sheaves}

Let $X$ be a $k$-variety.
\begin{itemize}
	\item We let $X^{\sigma}$ denote the complex-analytic space underlying the complex variety $X \times_{k,\sigma} \C$, and we call it the \textit{$\sigma$-analytification} of $X$.
	\item We let $D(X)$ denote the derived category of sheaves of $\Q$-vector spaces for the complex-analytic topology on $X^{\sigma}$.
	\item We let $\Dbct(X) \subset D(X)$ denote the $k$-algebraic constructible derived category (see Definition~\ref{defn:Dbgeo}).
	\item We let $\Perv(X) \subset \Dbct(X)$ denote the heart of the perverse $t$-structure.
	\item If $X$ is smooth over $k$, we let $\Loc(X) \subset \Perv(X)$ denote the abelian subcategory of shifted local systems (see Notation~\ref{nota:Loc(X)}).
\end{itemize} 

\subsection*{Motivic sheaves and motivic Galois groups}

\begin{itemize}
	\item We let $\GmotNo(k)$ denote Nori's motivic Galois group;
	we write it as $\GmotNo(k,\sigma)$ if we want to stress its dependence on $\sigma$.
	\item We let $\Gmot(k)$ denote Ayoub's spectral motivic Galois group;
	we write it as $\Gmot(k,\sigma)$ if we want to stress its dependence on $\sigma$.
	\item We let $\GmotAy(k)$ denote the classical pro-algebraic group underlying $\Gmot(k)$;
	we write it as $\GmotAy(k,\sigma)$ if we want to stress its dependence on $\sigma$.
	\item We let $\DA(X)$ denote the stable $\infty$-category of $\Q$-linear étale Voevodsky motivic sheaves on $X$.
	\item We let $\DAct(X) \subset \DA(X)$ denote the sub-$\infty$-category of constructible motivic sheaves.
	\item We let $\MNori(X)$ denote the abelian category of perverse Nori motives over $X$;
	we write it as $\MNori_{\sigma}(X)$ if we want to stress its dependence on $\sigma$ (and, implicitly, on the base field $k$).
\end{itemize}

\section{Voevodsky motives and sheaves of geometric origin}
\label{sect:Dbgeo}

This first section focuses on the theory of Voevodsky motivic sheaves, which lies at the background of the recent approaches to Nori motivic sheaves.
After recalling well-known facts about Voevodsky motives and the Betti realization, we introduce the associated Betti algebras and summarize their properties.
Then we review the theory of constructible sheaves of geometric origin, following Ayoub's treatment in~\cite[\S~1.6]{AyoAnab}.
We are particularly interested in the abelian categories of perverse sheaves of geometric origin and in their subcategories of shifted local systems.

Throughout, we work over a base field $k$ of characteristic $0$.
For future reference, in the course of the section we study the behaviour of the theory with respect to finite extensions of $k$;
the details of this technical discussion can be safely skipped in a first reading.

\subsection{Voevodsky motives and the Betti realization}\label{subsect:Voe-Bti}

For every $k$-variety $X$
there exists a $\Q$-linear triangulated category of
\textit{Voevodsky motivic sheaves} on $X$,
which underlies a symmetric monoidal stable $\infty$-category.
Several constructions of these motivic categories are available in the literature, notably:
Voevodsky's original construction
of Nisnevich motives with transfers,
étale motives without transfers by Ayoub (see~\cite{AyoThesis}),
and Beilinson motives by Cisinski--Déglise (see~\cite{CD19}).
All these approaches are known to yield equivalent categories (see~\cite[Thm.~B.1]{AyoEtaleReal} and~\cite[Thm.~16.2.18]{CD19}).
Throughout this paper, we work with Ayoub's version of the theory.

We let $\DA(X)$ denote the category of
$\Q$-linear étale Voevodsky motivic sheaves over $X$:
it is constructed in~\cite[\S~3]{AyoEtaleReal}
starting from the category $\Sm/X$ of smooth $X$-schemes,
by taking the $\infty$-category of $\Q$-linear presheaves on it
and imposing
étale descent,
$\mathbb{A}^1$-invariance,
and $\mathbb{P}^1$-stability.
By construction, there is a canonical functor
\begin{equation*}
	M_X \colon \Sm/X \to \DA(X)
\end{equation*} 
associating to every smooth $X$-scheme $W$ its relative motive.
The object $\Q_X := M_X(X) \in \DA(X)$,
called the \textit{unit motive} over $X$,
is the unit for the monoidal structure on $\DA(X)$.
By~\cite[Thm.~4.5.67]{AyoThesis}, the triangulated category $\DA(X)$ is compactly generated by the Tate twists of the motives $M_X(W)$ with $W \in \Sm/X$.

As $X$ varies among all $k$-varieties, the categories $\DA(X)$ are endowed with a canonical six functor formalism:
Ayoub's work~\cite{AyoThesis} describes the construction of the six operations over quasi-projective bases, and Cisinski--Déglise's work~\cite{CD19} extends it beyond the quasi-projective setting.
The inverse image functor
\begin{equation*}
f^* \colon \DA(Y) \to \DA(X)
\end{equation*}
associated to a morphism $f \colon X \to Y$ is defined by extension of the base-change functor
\begin{equation*}
	\Sm/Y \to \Sm/X, \quad W \mapsto W \times_Y X.
\end{equation*}
Its right adjoint is the direct image functor
\begin{equation*}
	f_* \colon \DA(X) \to \DA(Y).
\end{equation*}
The tensor product functor
\begin{equation*}
	- \otimes - \colon \DA(X) \times \DA(X) \to \DA(X)
\end{equation*}
is defined by extension of the fibre product
\begin{equation*}
	\Sm/X \times \Sm/X \to \Sm/X, \quad (W_1, W_2) \mapsto W_1 \times_X W_2.
\end{equation*}
These rules allow one to compare certain combinations of the six operations by looking at their effect on the generating sites.

\begin{rem}\label{rem:f_finét}
	In general, the direct image functor $f_* \colon \DA(X) \to \DA(Y)$ does not extend a functor $\Sm/X \to \Sm/Y$.
	There is, however, one important exception:
	given a finite étale morphism $q \colon X' \to X$, the functor $q_*$ coincides with the left adjoint to $q^*$ (see \cite[Sch.~1.4.2]{AyoThesis}), which extends the forgetful functor $\Sm/{X'} \to \Sm/X$.
	Moreover, the natural transformation between functors $\DA(X) \to \DA(X)$
	\begin{equation*}
		\id_{\DA(X)} \xrightarrow{\eta} q_* q^* \xrightarrow{\epsilon} \id_{\DA(X)}
	\end{equation*}
	and the natural transformation between functors $\DA(X') \to \DA(X')$
	\begin{equation*}
		\id_{\DA(X')} \xrightarrow{\eta} q^* q_* \xrightarrow{\epsilon} \id_{\DA(X')}
	\end{equation*}
	are both multiplication by the degree of $q$ (see~\cite[Prop.~13.7.6]{CD19}).
	It follows that every object $M \in \DA(X)$ is a direct summand of $q_* q^* M$, while every object $M' \in \DA(X')$ is a direct summand of $q^* q_* M'$.
\end{rem}

An object of $\DA(X)$ is \emph{constructible} if it belongs to the smallest stable sub-$\infty$-category containing the motives $M_X(W)$ with $W \in \Sm/X$ and closed under retracts.
Since $\DA(X)$ is compactly generated by such motives,
an object is constructible
if and only if
it is compact~\cite[Prop.~15.1.4]{CD19}.
By~\cite[Thm.~6.3.26]{CD16},
this notion of constructibility matches with the topological intuition:
an object $A \in \DA(X)$ is constructible
if and only if
there is a stratification of $X$
into locally closed subvarieties $S$
such that
the inverse image of $A$ on each stratum $S$
is dualizable in the symmetric monoidal category $\DA(S)$.
We let $\DAct(X)$ denote
the full subcategory of
constructible objects in $\DA(X)$.
By~\cite[Thm.~8.10, Thm.~8.12]{AyoEtaleReal},
the subcategories $\DAct(X)$ are stable under the six operations.
In conclusion,
$\DAct(X)$ is the smallest stable sub-$\infty$-category of $\DA(X)$ containing the unit motive $\Q_X$ and closed under any functor $\DA(X) \to \DA(X)$ obtained out of the six operations.

Let us now assume that we are given a complex embedding
$\sigma \colon k \hookrightarrow \C$.
This allows us to compare the theory of Voevodsky motives with the classical theory of sheaves for the analytic topology.
 
\begin{nota}
	Let $X$ be a $k$-variety.
	\begin{itemize}

		\item We let $X^{\sigma}$ denote the complex-analytic space attached to $X$ via $\sigma$:
		it is the complex-analytic space associated to
		the $\C$-variety $X \times_{k,\sigma} \C$.

		\item We let $D(X)$ denote the derived category of sheaves of $\Q$-vector spaces over the complex-analytic space $X^{\sigma}$;
		we write it as $D_{\sigma}(X)$ if we want to stress its dependence on the chosen embedding $\sigma$ (and, implicitly, on the base field $k$).
	\end{itemize}
\end{nota}

In~\cite{AyoBettiReal},
Ayoub constructs adjoint functors
\begin{equation*}
	\Bti_X^* \colon \DA(X) \leftrightarrows D(X) \colon \Bti_{X,*}.
\end{equation*}
The functor $\Bti_X^*$,
called the \textit{Betti realization functor},
extends the change-of-site functor
\begin{equation*}
	\Sm/X \to \Sm/{X^{\sigma}}, \quad W \mapsto W^{\sigma}.
\end{equation*}
The functors $\Bti_X^*$ are compatible with inverse images and with the tensor product.
Consequently, the right adjoint functors $\Bti_{X,*}$ are compatible with direct images.
The combination of~\cite[Thm.~3.4, Thm.~3.7, Thm.~3.8]{AyoBettiReal} shows that the functors
\begin{equation*}
	\Bti_X^* \colon \DAct(X) \to D(X)
\end{equation*}
obtained by restriction to constructible objects
commute with all the six operations.

The Betti realization is controlled by the Betti algebra:

\begin{lem}[{\cite[Defn.~1.4.1]{AyoAnab}}]
\label{lem:BtiAlg_defn}\ 
\begin{enumerate}

\item For every $k$-variety $X$, the object
\begin{equation*}
\mathcal{B}_X := \Bti_{X,*} \Q_X
\in \DA(X)
\end{equation*}
is naturally a commutative algebra object, called the \textit{Betti algebra} of $X$.

\item For every morphism of $k$-varieties $f \colon X \to Y$, there is a canonical commutative algebra isomorphism in $\DA(X)$
\begin{equation*}
f^* \mathcal{B}_Y \xrightarrow{\sim} \mathcal{B}_X.
\end{equation*}
In particular, for every $k$-variety $X$, there is a canonical isomorphism in $\DA(X)$
\begin{equation*}
a_X^* \mathcal{B}_k \xrightarrow{\sim} \mathcal{B}_X,
\end{equation*}
where $a_X \colon X \to \Spec(k)$ denotes the structural morphism.
\end{enumerate}
\end{lem}

\begin{proof}
	The first statement is a formal consequence of
	the fact that $\Bti_X^*$ is symmetric monoidal.
	Let us just mention that the unit of $\mathcal{B}_X$ is the unit morphism
	\begin{equation*}
		\Q_X \xrightarrow{\eta} \Bti_{X,*} \Bti_X^* \Q_X = \clg{B}_X,
	\end{equation*}
	while the multiplication is defined as the composite
	\begin{equation*}
		\begin{aligned}
			\Bti_{X,*} \Q_X \otimes \Bti_{X,*} \Q_X \xrightarrow{\eta} &\Bti_{X,*} \Bti_X^* (\Bti_{X,*} \Q_X \otimes \Bti_{X,*} \Q_X) \\
			= &\Bti_{X,*}(\Bti_X^* \Bti_{X,*} \Q_X \otimes \Bti_X^* \Bti_{X,*} \Q_X) \\
			\xrightarrow{\epsilon} &\Bti_{X,*}(\Q_X \otimes \Q_X) = \Bti_{X,*} \Q_X.
		\end{aligned}
	\end{equation*}
	See~\cite[Cor.~1.15]{AyoHopf1} for more details.
	
	The morphism in the second statement is defined as the composite
	\begin{equation*}
		f^* \Bti_{Y,*} \Q_Y \xrightarrow{\eta} f^* \Bti_{Y,*} f_* f^* \Q_Y = f^* f_* \Bti_{X,*} f^* \Q_Y \xrightarrow{\epsilon} \Bti_{X,*} f^* \Q_Y = \Bti_{X,*} \Q_X.
	\end{equation*}
	Compatibility with multiplication and units follows formally from
	the lax monoidality of the direct image functors $f_*$ and of the functors $\Bti_{X,*}$ and $\Bti_{Y,*}$.
	In order to show that it is an isomorphism,
	it suffices to check that the induced map
	\begin{equation*}
		\Hom_{\DA(X)}(A,f^* \mathcal{B}_Y) \to \Hom_{\DA(X)}(A,\mathcal{B}_X)
	\end{equation*}
	is bijective when $A$ runs over a well-chosen family
	of compact generators of $\DA(X)$.
	This can be done as in the proof of~\cite[Prop.~1.6.6]{AyoAnab}
	(see also~\cite[Rem.~1.6.7]{AyoAnab}).
\end{proof}

Let us insist that
$\mathcal{B}_X$ is a commutative algebra object
in $\DA(X)$,
in the $\infty$-categorical sense.
This gives access to a well-behaved theory of modules over the Betti algebra,
discussed in \S~\ref{subsect:Dbgeo}. 

In the final part of this section,
we spell out the behaviour of Betti algebras
under finite extensions of the base field $k$.
This is the prototype for
several auxiliary results collected in the present section.
To avoid possible confusion, we need to use more explicit notation:

\begin{nota}
For every $k$-variety $X$, we write
$\Bti_{\sigma,X}^*$,
$\Bti_{\sigma,X,*}$,
and $\mathcal{B}_{\sigma,X}$ in place of
$\Bti_X^*$,
$\Bti_{X,*}$,
and $\mathcal{B}_X$,
respectively,
when we want to stress their dependence on $\sigma$
(and, implicitly, on the base field $k$).
\end{nota}

Let $k'/k$ be a finite extension.
Given a $k$-variety $X$, set $X' := X \times_k k'$, and let $e_X \colon X' \to X$ denote the corresponding finite étale morphism.
Choose a complex embedding $\sigma' \colon k' \hookrightarrow \C$ extending $\sigma$.

\begin{prop}\label{prop:e_*BtiAlg'}
	The object $e_{X,*} \mathcal{B}_{\sigma',X'} \in \DA(X)$ is naturally a commutative algebra object, and there is a canonical commutative algebra isomorphism in $\DA(X)$
	\begin{equation*}
		\mathcal{B}_{\sigma,X} = e_{X,*} \mathcal{B}_{\sigma',X'}.
	\end{equation*}
\end{prop}

The starting point is the identification of complex-analytic spaces
\begin{equation}\label{formula:X^sigma=X'^sigma'}
	X^{\sigma} = (X')^{\sigma'},
\end{equation}
induced by the identification of $\C$-varieties
\begin{equation}\label{eq:X^sigma=X'^sigma'_varieties}
	X \times_{k,\sigma} \C = (X \times_k k') \times_{k',\sigma'} \C.
\end{equation}
This determines a canonical equivalence of monoidal stable $\infty$-categories
\begin{equation*}
	D_{\sigma}(X) \xrightarrow{\sim} D_{\sigma'}(X'),
\end{equation*}
which is compatible with the Betti realization functors in the expected way:

\begin{lem}\label{lem:Bti_k=Bti_k-e^*}
	The diagram of monoidal functors
	\begin{equation*}
		\begin{tikzcd}
			\DA(X) \arrow{rr}{e_X^*} \arrow{d}{\Bti_{\sigma,X}^*} && \DA(X') \arrow{d}{\Bti_{\sigma',X'}^*} \\
			D_{\sigma}(X) \arrow{rr}{\sim} && D_{\sigma'}(X')
		\end{tikzcd}
	\end{equation*}
	commutes up to monoidal natural isomorphism.
\end{lem}

\begin{proof}
	Indeed, the identification~\eqref{formula:X^sigma=X'^sigma'}
	extends to a natural isomorphism between functors on $\Sm/X$
	\begin{equation*}
		W^{\sigma} = (W \times_X X')^{\sigma'},
	\end{equation*}
	compatibly with the Cartesian monoidal structure.
\end{proof}

\begin{proof}[Proof of Proposition~\ref{prop:e_*BtiAlg'}]
	The commutative algebra structure of $e_{X,*} \mathcal{B}_{\sigma',X'}$ is induced by that of $\mathcal{B}_{\sigma',X'}$:
	the unit is defined as
	\begin{equation*}
		\Q_X \xrightarrow{\eta} e_{X,*} e_X^* \Q_X = e_{X,*} \Q_{X'} \rightarrow e_{X,*} \mathcal{B}_{\sigma',X'},
	\end{equation*}
	while the multiplication is defined as
	\begin{equation*}
		\begin{aligned}
			e_{X,*} \mathcal{B}_{\sigma',X'}
				\otimes e_{X,*} \mathcal{B}_{\sigma',X'}
			& \xrightarrow{\eta} e_{X,*} e_X^* (e_{X,*} \mathcal{B}_{\sigma',X'}
				\otimes e_{X,*} \mathcal{B}_{\sigma',X'}) \\
			& = e_{X,*} (e_X^* e_{X,*} \mathcal{B}_{\sigma',X'}
				\otimes e_X^* e_{X,*} \mathcal{B}_{\sigma',X'}) \\
			& \xrightarrow{\epsilon} e_{X,*} (\mathcal{B}_{\sigma',X'}
				\otimes \mathcal{B}_{\sigma',X'}) \\
			& \rightarrow e_{X,*} \mathcal{B}_{\sigma',X'}.
		\end{aligned}
	\end{equation*}
	The isomorphism in the statement
	comes from the commutative diagram of
	Lemma~\ref{lem:Bti_k=Bti_k-e^*}.
	Its compatibility with the units and with the multiplication maps comes down to the compatibility of units and counits of adjunctions with respect to composition of functors;
	we leave the details to the interested reader.
\end{proof}

Since the scheme $X'$ can be regarded either as a $k'$-variety or as a $k$-variety, one can attach to it two distinct Betti algebras, namely $\mathcal{B}_{\sigma',X'}$ and $\mathcal{B}_{\sigma,X'}$.
In order to understand their relation, it is convenient to consider all possible choices of $\sigma'$ at once.
So let $\Hom_{\sigma}(k',\C)$ denote the set of complex embeddings of $k'$ extending $\sigma$, and let us temporarily forget about the previous choice of an element $\sigma' \in \Hom_{\sigma}(k',\C)$.

\begin{prop}\label{prop:BtiAlg_directsum_embeddings}
	There is a canonical commutative algebra isomorphism in $\DA(X')$
	\begin{equation*}
		\clg{B}_{\sigma,X'} = \bigoplus_{\sigma' \in \Hom_{\sigma}(k',\C)} \clg{B}_{\sigma',X'}.
	\end{equation*}
\end{prop}

To begin with, note that we have a decomposition of complex-analytic spaces
\begin{equation}\label{eq:X'^sigma-decomp}
	(X')^{\sigma} = \coprod_{\sigma' \in \Hom_{\sigma}(k',\C)} (X')^{\sigma'}.
\end{equation}
It is induced by the decomposition of $\C$-varieties
\begin{equation*}
	X' \times_{k,\sigma} \C = (X \times_k k') \times_{k,\sigma} \C = X \times_k (k' \times_{k,\sigma} \C) = \coprod_{\sigma' \in \Hom_{\sigma}(k',\C)} X \times_{k,\sigma} \C = \coprod_{\sigma' \in \Hom_{\sigma}(k',\C)} X' \times_{k',\sigma'} \C,
\end{equation*}
where the last identification is the one
in~\eqref{eq:X^sigma=X'^sigma'_varieties},
as $\sigma'$ varies.
We deduce a canonical equivalence of monoidal stable $\infty$-categories
\begin{equation}\label{eq:D_sigma(X')-decomp}
	D_{\sigma}(X') = \prod_{\sigma' \in \Hom_{\sigma}(k',\C)} D_{\sigma'}(X').
\end{equation}
This allows us to identify each factor $D_{\sigma'}(X')$ with the full sub-$\infty$-category of $D_{\sigma}(X')$ spanned by the objects with support on $(X')^{\sigma'}$.

\begin{rem}\label{rem:unit-decomp_X'}
	Under the equivalence~\eqref{eq:D_sigma(X')-decomp}, the unit object $\Q_{\sigma,X'} \in D_{\sigma}(X')$ corresponds to the unit tuple $(\Q_{\sigma',X'})_{\sigma'}$.
	The sheaves $\Q_{\sigma',X'}$, extended by zero to $(X')^{\sigma}$, form a family of pairwise orthogonal idempotents with respect to the monoidal structure of $D_{\sigma}(X')$, and the subcategory $D_{\sigma'}(X') \subset D_{\sigma}(X')$ can be identified with the essential image of the endofunctor $\Q_{\sigma',X'} \otimes -$.
\end{rem}

The decomposition~\eqref{eq:D_sigma(X')-decomp} is compatible with the Betti realization functors in the expected way:

\begin{lem}\label{lem:Bti^*-decomp}
	The diagram of monoidal functors
	\begin{equation*}
		\begin{tikzcd}
			& \DA(X') \arrow{dl}{\Bti_{\sigma,X'}^*} \arrow{dr}{(\Bti_{\sigma',X'}^*)_{\sigma'}} \\
			D_{\sigma}(X') \arrow[equal]{rr} && \prod_{\sigma' \in \Hom_{\sigma}(k',\C)} D_{\sigma'}(X')
		\end{tikzcd}
	\end{equation*}
	commutes up to monoidal natural isomorphism.
\end{lem}
\begin{proof}
	Indeed, the decomposition \eqref{eq:X'^sigma-decomp} extends to a functorial decomposition on $\Sm/{X'}$
	\begin{equation*}
		W^{\sigma} = \coprod_{\sigma' \in \Hom_{\sigma}(k',\C)} W^{\sigma'}
	\end{equation*}
	compatible with the Cartesian monoidal structure.
	This is induced by the functorial decomposition of $\C$-varieties
	\begin{equation*}
		W \times_k \C = W \times_{X'} (X' \times_k \C) = \coprod_{\sigma' \in \Hom_{\sigma}(k',\C)} W \times_{X'} (X' \times_{k',\sigma'} \C) = \coprod_{\sigma' \in \Hom_{\sigma}(k',\C)} W \times_{k',\sigma'} \C,
	\end{equation*}
	where the second identification 
	comes from~\eqref{eq:X'^sigma-decomp}.
\end{proof}

\begin{proof}[Proof of Proposition \ref{prop:BtiAlg_directsum_embeddings}]
	Note that the right adjoint to the functor
	\begin{equation*}
		\DA(X') \to \prod_{\sigma' \in \Hom_{\sigma}(k',\C)} D_{\sigma'}(X'), \quad A \mapsto (\Bti_{\sigma',X'}^* A)_{\sigma'} 
	\end{equation*}
	is given by
	\begin{equation*}
		\prod_{\sigma' \in \Hom_{\sigma}(k',\C)} D_{\sigma'}(X') \to \DA(X'), \quad (K_{\sigma'})_{\sigma'} \mapsto \bigoplus_{\sigma' \in \Hom_{\sigma}(k',\C)} \Bti_{\sigma',X',*} K_{\sigma'}.
	\end{equation*}
	We deduce the commutative algebra isomorphism in $\DA(X')$
	\begin{align*}
		\clg{B}_{\sigma,X'} := \Bti_{\sigma,X',*} \Q_{\sigma,X'}  
		&= \Bti_{\sigma,X',*} \Bti_{\sigma,X'}^* \Q_{X'} \\
		&= \bigoplus_{\sigma' \in \Hom_{\sigma}(k',\C)} \Bti_{\sigma',X',*} \Bti_{\sigma',X'}^* \Q_{X'} && \textup{(by Lemma \ref{lem:Bti^*-decomp})} \\
		&= \bigoplus_{\sigma' \in \Hom_{\sigma}(k',\C)} \Bti_{\sigma',X',*} \Q_{\sigma',X'} \\
		&=: \bigoplus_{\sigma' \in \Hom_{\sigma}(k',\C)} \clg{B}_{\sigma',X'},
	\end{align*}
	as wanted.
\end{proof}

\subsection{Constructible complexes of geometric origin}\label{subsect:Dbgeo}

We need to study
the Betti realization $\Bti_X^* \colon \DAct(X) \to D(X)$
more closely.
This leads us to review the theory of constructible sheaves.

\begin{nota}\label{nota:Strat_X/k}
	Let $X$ be a $k$-variety.
	\begin{itemize}
		\item By a \textit{$k$-stratification} $\Sigma$ of $X$ we mean a finite collection of locally-closed $k$-subvarieties $S$ of $X$, called the \textit{strata} of $\Sigma$, with the following properties:
		\begin{enumerate}
			\item[(i)] Each stratum $S \in \Sigma$ is smooth over $k$ and connected.
			\item[(ii)] The strata of $\Sigma$ define a partition of the underlying topological space of $X$.
			\item[(iii)] For each stratum $S \in \Sigma$, the Zariski closure $\bar{S}$ of $S$ inside $X$ is a union of strata.
		\end{enumerate}
		\item Given two $k$-stratifications $\Sigma$ and $\Sigma'$ of $X$, we say that $\Sigma'$ is a \textit{refinement} of $\Sigma$ if each stratum of $\Sigma$ is a union of strata of $\Sigma'$.
		\item We let $\Strat_{X/k}$ denote the poset of $k$-stratifications of $X$, with the order defined by saying that $\Sigma \leq \Sigma'$ if and only if $\Sigma'$ is a refinement of $\Sigma$.
	\end{itemize}
\end{nota}

Given two $k$-stratifications $\Sigma_1, \Sigma_2 \in \Strat_{X/k}$, the collection of all intersections $S_1 \cap S_2$ with $S_1 \in \Sigma_1$ and $S_2 \in \Sigma_2$ satisfies conditions (ii) and (iii) of a $k$-stratification, but not necessarily condition (i).
Nevertheless, any collection of locally closed $k$-subvarieties of $X$ satisfying (ii) and (iii) admits a refinement which also satisfies (i).
This implies that the poset $\Strat_{X/k}$ is filtered.

\begin{defn}
	Let $X$ be a $k$-variety, and let $\Sigma \in \Strat_{X/k}$.
	\begin{itemize}
		\item An object $K \in D(X)$ is \textit{$\Sigma$-constructible} if, for every stratum $S \in \Sigma$, with inclusion $s \colon S \hookrightarrow X$, the object $s^* K \in D(S)$ is dualizable.

		\item We let $\Dbct(X,\Sigma)$ denote the full sub-$\infty$-category of $D(X)$ spanned by the $\Sigma$-constructible objects. 

		\item We define the stable $\infty$-category of \textit{constructible complexes} as the filtered union
		\begin{equation*}
			\Dbct(X) := \bigcup_{\Sigma \in \Strat_{X/k}} \Dbct(X,\Sigma) \subset D(X).
		\end{equation*}
	\end{itemize}
\end{defn}

As the notation suggests, every complex $K \in \Dbct(X)$ is automatically bounded (as this is true for dualizable complexes on the strata).
Let us note the following result:
\begin{prop}[{\cite[Prop.~1.2.13]{AyoAnab}}]\label{prop:Dbct-6ff}
	As $X$ varies among all $k$-varieties, the subcategories $\Dbct(X) \subset D(X)$ are stable under the six operations.
\end{prop}

\begin{cor}
	For every $k$-variety $X$, the Betti realization $\Bti_X^* \colon \DA(X) \to D(X)$ restricts to
	\begin{equation*}
		\Bti_X^* \colon \DAct(X) \to \Dbct(X).
	\end{equation*}
\end{cor}
\begin{proof}
	By~\cite[Thm.~6.3.26]{CD16}, for every object $A \in \DAct(X)$ there is a $k$-algebraic stratification of $X$ such that the restriction of $A$ to each stratum is dualizable;
	using localization sequences, one can reconstruct $A$ from such restrictions in a finite number of steps.
	Hence, in view of Proposition~\ref{prop:Dbct-6ff}, it suffices to observe that the Betti realization functors take dualizable constructible motives to dualizable sheaves (being symmetric monoidal).
\end{proof}

Note that, however,
the right adjoint
$\Bti_{X,*} \colon D(X) \to \DA(X)$
does not take $\Dbct(X)$ to $\DAct(X)$:
for instance,
the Betti algebra $\mathcal{B}_X \in \DA(X)$ is not constructible.%
\footnote{%
It is, however, ind-constructible.
}

An interesting part of $\Dbct(X)$ is obtained from
the relative cohomology of $k$-varieties:

\begin{defn}[{\cite[Def.~1.6.1, Rmk.~1.6.2]{AyoAnab}}]\label{defn:Dbgeo}
	For every $k$-variety $X$, we define the stable $\infty$-category $\Dbgeo(X)$ of \textit{constructible complexes of geometric origin} as the smallest sub-$\infty$-category of $\Dbct(X)$ containing all objects of the form $p_* \Q_{W^{\sigma}} = \Bti_X^*(p_* \Q_W)$, with $p \colon W \to X$ a proper morphism of $k$-varieties, and stable under finite limits, finite colimits, and retracts:
	in other words, the full sub-$\infty$-category spanned by the thick triangulated subcategory generated by the objects $p_* \Q_{W^{\sigma}}$ above;
	we write it as $\Dbgeoemb{\sigma}(X)$ if we want to stress its dependence on $\sigma$.
\end{defn}

Setting $\Dbgeo(X,\Sigma) := \Dbgeo(X) \cap \Dbct(X,\Sigma)$ for every $k$-stratification $\Sigma$ of $X$, we get the filtered union
\begin{equation*}
	\Dbgeo(X) = \bigcup_{\Sigma \in \Strat_{X/k}} \Dbgeo(X,\Sigma).
\end{equation*}

In the rest of the present subsection, we discuss the close relation between sheaves of geometric origin and Voevodsky motives.

\begin{prop}[{\cite[Rem.~1.6.4]{AyoAnab}}]\label{prop:Dbgeo-6ff}
	As $X$ varies among all $k$-varietes, the subcategories $\Dbgeo(X) \subset \Dbct(X)$ are stable under the six operations, as well as under Beilinson's gluing functors.
\end{prop}

\begin{proof}
	Stability under inverse images follows from proper base-change, and stability under the tensor product follows from the projection formula.
	The method of~\cite[\S\S~2.2.2, 2.3.10]{AyoThesis} allows one to deduce stability under the six operations from this.
	Stability under the gluing functors follows from their description in terms of the six operations (see~\cite[\S\S~3-7]{Mor18} or~\cite[\S~3]{IM19}).
\end{proof}

\begin{prop}[{\cite[Rem.~1.6.4]{AyoAnab}}]\label{prop:Bti-through-Dbgeo}
	For every $k$-variety $X$, the Betti realization $\Bti_X^* \colon \DAct(X) \to \Dbct(X)$ refines to
	\begin{equation*}
		\Bti_X^* \colon \DAct(X) \rightarrow \Dbgeo(X).
	\end{equation*}
\end{prop}

\begin{proof}
	By~\cite[Lem.~2.2.23]{AyoThesis}, the stable $\infty$-category $\DA(X)$ is generated under colimits, negative Tate twists, and negative shifts by the objects of the form $p_* \Q_W$  with $p \colon W \to X$ a proper morphism.
	Hence, the result follows from Proposition \ref{prop:Dbgeo-6ff}.
\end{proof}

This leads to a more conceptual description of sheaves of geometric origin as a six functor formalism.

\begin{nota}[{\cite[Rmk.~1.1.21, Constr.~1.6.5]{AyoAnab}}]
	Let $X$ be a $k$-variety.
	\begin{itemize}
		\item We let $\DA(X;\mathcal{B}_X)$ denote the stable $\infty$-category of $\mathcal{B}_X$-modules in $\DA(X)$, and we write
		\begin{equation*}
			\mathcal{B}_X \otimes - \colon \DA(X) \to \DA(X;\mathcal{B}_X)
		\end{equation*}
		for the canonical functor.
		\item We let $\DAct(X;\mathcal{B}_X)$ denote the smallest sub-$\infty$-category of $\DA(X;\mathcal{B}_X)$ containing the image of $\DAct(X)$ and stable under finite limits, finite colimits, and retracts.
	\end{itemize} 
\end{nota}

As explained in \cite[Rmk.~1.1.21]{AyoAnab}, the stable $\infty$-categories $\DA(X;\mathcal{B}_X)$ naturally assemble into a six functor formalism, and the canonical functors $\mathcal{B}_X \otimes -$ are compatible with the six operations.

\begin{prop}[{\cite[Prop.~1.6.6]{AyoAnab}}, {\cite[Ex.~17.1.7]{CD19}}]
\label{prop:Dbgeo=Mod_DA}
	For every $k$-variety $X$, the Betti realization $\Bti_X^* \colon \DA(X) \to D(X)$ admits a canonical factorization of the form
	\begin{equation*}
		\Bti_X^* \colon \DA(X) \xrightarrow{\mathcal{B}_X \otimes -} \DA(X;\mathcal{B}_X) \to D(X),
	\end{equation*}
	compatibly with the six operations.
	Moreover, the functor
	\begin{equation*}
		\DAct(X;\mathcal{B}_X) \to D(X)
	\end{equation*} 
	obtained by restriction is fully faithful.
\end{prop}

\begin{cor}\label{cor:Dbgeo=DAct(Btialg)}
	For every $k$-variety $X$, there is a canonical equivalence of stable $\infty$-categories
	\begin{equation*}
		\DAct(X;\mathcal{B}_X) \xrightarrow{\sim} \Dbgeo(X),
	\end{equation*}
	compatibly with the six operations.
\end{cor}

\begin{proof}
	Since the functor $\DAct(X;\mathcal{B}_X) \to D(X)$ is fully faithful (by Proposition~\ref{prop:Dbgeo=Mod_DA}), it remains to show that its essential image coincides with $\Dbgeo(X)$.
	This follows from the fact that all generators of $\Dbgeo(X)$ lie in the essential image of the Betti realization $\Bti_X^* \colon \DAct(X) \to \Dbct(X)$.
\end{proof}

\begin{rem}\label{rem:Dgeo}
	Although the functors $\DA(X;\mathcal{B}_X) \to D(X)$ are not fully faithful, it is still possible to construct by hand a six functor formalism $X \mapsto \Dgeo(X)$ of unbounded complexes of geometric origin such that the previous functors factor through equivalences $\DA(X;\mathcal{B}_X) \xrightarrow{\sim} \Dgeo(X)$ (see~\cite[Defn.~1.6.1]{AyoAnab}).
	Despite the suggestive notation, $\Dgeo(X)$ is not defined as a full subcategory of $D(X)$. 
\end{rem}

Later in the paper, we need to know how complexes of geometric origin change under finite extensions of the base field $k$.
For this reason, we now come back to the situation considered at the end of Subsection~\ref{subsect:Voe-Bti}, and we relate the results obtained there to sheaves of geometric origin.
So let $k'/k$ be a finite extension.
Given a $k$-variety $X$, set $X' := X \times_k k'$, and write $e_X \colon X' \to X$ for the corresponding finite étale morphism.
Choose a complex embedding $\sigma' \colon k' \hookrightarrow \C$
extending $\sigma$.
We have a symmetric monoidal functor
\begin{equation*}
	\DA(X;\mathcal{B}_{\sigma,X}) = \DA(X;e_{X,*} \mathcal{B}_{\sigma',X'}) \xrightarrow{e_X^*} \DA(X';e_X^* e_{X,*} \mathcal{B}_{\sigma',X'}) \xrightarrow{\epsilon} \DA(X';\mathcal{B}_{\sigma',X'}),
\end{equation*}
where the first equivalence is induced by Proposition~\ref{prop:e_*BtiAlg'}.
By restriction to constructible objects,
this induces a symmetric monoidal functor
\begin{equation}\label{eq:DAct(X,B)_to_DAct(X',B')}
	\DAct(X;\mathcal{B}_{\sigma,X}) \to \DAct(X';\mathcal{B}_{\sigma',X'}).
\end{equation}
In view of Corollary~\ref{cor:Dbgeo=DAct(Btialg)},
we look at the corresponding functor on constructible complexes. 

\begin{lem}\label{lem:Dbgeo_X'/X}
The functor
\begin{equation*}
\Dbgeoemb{\sigma}(X) \rightarrow \Dbgeoemb{\sigma'}(X')
\end{equation*}
corresponding to~\eqref{eq:DAct(X,B)_to_DAct(X',B')} under the identifications of Corollary~\ref{cor:Dbgeo=DAct(Btialg)} is an equivalence.
\end{lem}

\begin{proof}
	The canonical equivalence $D_{\sigma}(X) \xrightarrow{\sim} D_{\sigma'}(X')$ allows us to view both categories in the statement as full subcategories of $D_{\sigma'}(X')$, and we claim that their essential images coincide.
	To this end, it suffices to check that each category contains the generators of the other one.
	
	In one direction, given a proper morphism of $k$-varieties $p \colon Y \rightarrow X$, let $p' \colon Y' \rightarrow X'$ denote the proper morphism of $k'$-varieties obtained by base-change to $k'$:
	it fits into the Cartesian square
	\begin{equation*}
		\begin{tikzcd}
			Y' \arrow{r}{e_Y} \arrow{d}{p'} & Y \arrow{d}{p} \\
			X' \arrow{r}{e_X} & X.
		\end{tikzcd}
	\end{equation*}
	We have a chain of natural isomorphisms
	\begin{align*}
		\Bti_{\sigma,X}^* (p_* \Q_Y)
		&= \Bti_{\sigma',X'}^* (e_X^* p_* \Q_Y) && \textup{(by Lemma~\ref{lem:Bti_k=Bti_k-e^*})}\\
		&= \Bti_{\sigma',X'}^* (p'_* e_Y^* \Q_Y) && \textup{(by proper base-change)} \\
		&= \Bti_{\sigma',X'}^* (p'_* \Q_{Y'}) && \textup{(as $e_Y^* \Q_Y = \Q_{Y'}$)}.
	\end{align*}
	Thus, each generator of $\Dbgeoemb{\sigma}(X)$ belongs to $\Dbgeoemb{\sigma'}(X')$.
	Note that the inclusion $\Dbgeoemb{\sigma}(X) \subset \Dbgeoemb{\sigma'}(X')$ just obtained is compatible with the functor \eqref{eq:DAct(X,B)_to_DAct(X',B')}:
	essentially, this comes down to Lemma~\ref{lem:Bti_k=Bti_k-e^*}.
	
	In the other direction, given a proper morphism of $k'$-varieties $q \colon Z \rightarrow X'$, consider the composite morphism of $k$-varieties $r := e_X \circ q \colon Z \rightarrow X$.
	Since $e_X$ is finite étale, the object $q_* \Q_Z \in \DA(X')$ is a direct summand of $e_X^* e_{X,*} q_* \Q_Z$ (by Remark~\ref{rem:f_finét}).
	Since $\Dbgeoemb{\sigma}(X)$ is closed under retracts inside $D_{\sigma}(X) = D_{\sigma'}(X')$ (by definition), it suffices to check that $\Dbgeoemb{\sigma}(X)$ contains the complex $\Bti_{\sigma',X'}^*(e_X^* e_{X,*} q_* \Q_Z)$ in order to conclude that it contains $\Bti_{\sigma',X'}^*(q_* \Q_Z)$.
	But we have
	\begin{align*}
		\Bti_{\sigma',X'}^* (e_X^* e_{X,*} q_* \Q_Z) = & \Bti_{\sigma,X}^*(e_{X,*} q_* \Q_Z) && \textup{(by Lemma~\ref{lem:Bti_k=Bti_k-e^*})} \\
		= & \Bti_{\sigma,X}(r_* \Q_Z) && \textup{(as $e_{X} \circ q = r$)}.
	\end{align*}
	Thus, each generator of $\Dbgeoemb{\sigma'}(X')$ belongs to $\Dbgeoemb{\sigma}(X)$.
	This concludes the proof of the claim.
\end{proof}

Recall that the categories of analytic sheaves on $X'$, regarded either as a $k'$-variety or as a $k$-variety, are related by the decomposition~\eqref{eq:D_sigma(X')-decomp}, indexed by the set $\Hom_{\sigma}(k',\C)$ of complex embeddings of $k'$ extending $\sigma$.
In what follows,
we forget again about the previous choice of
an element $\sigma' \in \Hom_{\sigma}(k',\C)$
and rather consider all such extensions at once.
We have a canonical equivalence
\begin{equation*}
	\DA(X';\Btialg_{\sigma,X'}) = \DA(X';\oplus_{\sigma' \in \Hom_{\sigma}(k',\C)} \Btialg_{\sigma',X'}) = \prod_{\sigma' \in \Hom_{\sigma}(k',\C)} \DA(X';\Btialg_{\sigma',X'}),
\end{equation*}
where the first equivalence is induced by Proposition~\ref{prop:BtiAlg_directsum_embeddings}.
By restriction to constructible objects, it induces an equivalence
\begin{equation*}
	\DAct(X';\Btialg_{\sigma,X'}) = \prod_{\sigma' \in \Hom_{\sigma}(k',\C)} \DAct(X';\Btialg_{\sigma',X'}).
\end{equation*}
In view of Corollary~\ref{cor:Dbgeo=DAct(Btialg)}, we look for the same relation on constructible complexes.

\begin{lem}\label{lem:Dbgeo_sigma(X')-decomp}
	The equivalence~\eqref{eq:D_sigma(X')-decomp} restricts
	to a monoidal equivalence
	\begin{equation*}
		\Dbgeoemb{\sigma}(X') = \prod_{\sigma' \in \Hom_{\sigma}(k',\C)} \Dbgeoemb{\sigma'}(X')
	\end{equation*}
	making the diagram
	\begin{equation*}
		\begin{tikzcd}
			\DAct(X';\Btialg_{\sigma,X'}) \arrow[equal]{rr} \arrow{d}{\sim} && \prod_{\sigma' \in \Hom_{\sigma}(k',\C)} \DAct(X';\Btialg_{\sigma',X'}) \arrow{d}{\sim} \\
			\Dbgeoemb{\sigma}(X') \arrow[equal]{rr} && \prod_{\sigma' \in \Hom_{\sigma}(k',\C)} \Dbgeoemb{\sigma'}(X')
		\end{tikzcd}
	\end{equation*}
	commute up to monoidal natural isomorphism.
\end{lem}

\begin{proof}
	It suffices to show that the diagram
	\begin{equation*}
		\begin{tikzcd}
			\DAct(X';\Btialg_{\sigma,X'}) \arrow[equal]{rr} \arrow{d} && \prod_{\sigma' \in \Hom_{\sigma}(k',\C)} \DAct(X';\Btialg_{\sigma',X'}) \arrow{d} \\
			\Dbctemb{\sigma}(X') \arrow[equal]{rr} && \prod_{\sigma' \in \Hom_{\sigma}(k',\C)} \Dbctemb{\sigma'}(X')
		\end{tikzcd}
	\end{equation*}
	commutes up to monoidal natural isomorphism.
	Since the algebra structure of Betti algebras comes from the Betti realization, this follows formally from Lemma~\ref{lem:Bti^*-decomp}.
\end{proof}

\subsection{Perverse sheaves and local systems of geometric origin}\label{subsect:Pervgeo}

Let $X$ be a $k$-variety.
Beilinson, Bernstein, Deligne and Gabber~\cite{BBDG}
define the perverse $t$-structure on $\Dbct(X)$.
Its heart $\Perv(X)$ is the abelian category of
($k$-algebraically constructible)
perverse sheaves on $X^{\sigma}$,
which we simply call \textit{perverse sheaves} on $X$ here.

We want to specialize the theory of perverse sheaves to
the setting of sheaves of geometric origin.

\begin{nota}
	Let $\Pervgeo^{\rm{A}}(X)$ denote
	the intersection $\Dbgeo(X) \cap \Perv(X)$.
	We write it as $\Pervgeoemb{\sigma}^{\rm{A}}(X)$
	if we want to stress its dependence on the embedding $\sigma$.
\end{nota}

This is but one possible definition of perverse sheaves of geometric origin:
the superscript is meant to distinguish it from another definition,
introduced in \S~\ref{subsect:IM_cat}.

In order to make sure that $\Pervgeo^{\rm{A}}(X)$ is a reasonable category,
we rely on the following result:

\begin{prop}[{\cite[Rmk.~1.6.23]{AyoAnab}}]\label{prop:Dbgeo_perv_t-str}
	The perverse $t$-structure on $\Dbct(X)$ restricts to a $t$-structure on $\Dbgeo(X)$, and its heart $\Pervgeo^{\rm{A}}(X)$ is stable under subquotients and extensions inside $\Perv(X)$.
\end{prop}
\begin{proof}
	We want to apply the criterion of~\cite[Lem.~1.6.22]{AyoAnab}.
	To this end, let $\mathcal{S}$ denote the collection of all subquotients in $\Perv(X)$ of the perverse sheaves $\pH^n(p_* \Q_{W^{\sigma}})$ with $p \colon W \to X$ proper and $n \geq 0$.
	Define $\Dbct(X)_{\mathcal{S}}$ as the smallest sub-$\infty$-category  $\Dbct(X)$ containing $\mathcal{S}$ and stable under finite colimits and negative shifts:
	in other words, the full sub-$\infty$-category spanned by the triangulated subcategory generated by $\mathcal{S}$.
	By~\cite[Lem.~1.6.22]{AyoAnab}, the perverse $t$-structure on $\Dbct(X)$ restricts to a $t$-structure on $\Dbct(X)_{\mathcal{S}}$, and the heart $\Perv(X)_{\mathcal{S}}$ of the latter coincides with the smallest abelian subcategory of $\Perv(X)$ containing $\mathcal{S}$ and stable under subquotients and extensions.
	In order to conclude, we claim that $\Dbct(X)_{\mathcal{S}}$ coincides in fact with $\Dbgeo(X)$.
	
	To this end, we start by checking that all objects of $\mathcal{S}$ belong to $\Dbgeo(X)$.
	This is a consequence of the
	Decomposition Theorem~\cite[Thm.~6.2.5]{BBDG}
	(see also~\cite[Thm.~2.1.1(b), Thm.~2.3.3]{cataldo-migliorini},
	for example):
	for every proper morphism $p \colon W \to X$, the complex $p_* \Q_{W^{\sigma}} \in \Dbct(X)$ decomposes as
	\begin{equation*}
		p_* \Q_{W^{\sigma}} = \bigoplus_{n \geq 0} \pH^n(p_* \Q_{W^{\sigma}})[-n],
	\end{equation*}
	and moreover each direct summand $\pH^n(p_* \Q_{W^{\sigma}})$ is semisimple in $\Perv(X)$.
	Since $\Dbgeo(X)$ is stable under retracts (by definition) and positive shifts
	(being stable under finite colimits),
	it contains indeed the whole of $\mathcal{S}$,
	and therefore it contains $\Dbct(X)_{\mathcal{S}}$.
	
	It remains to show that $\Dbct(X)_{\mathcal{S}}$ is stable under retracts inside $\Dbct(X)$, which follows formally from the boundedness of the perverse $t$-structure.
	In detail, fix a non-zero object $K \in \Dbct(X)_{\mathcal{S}}$ which is isomorphic in $\Dbct(X)$ to a direct sum $K' \oplus K''$, and let us show that $K'$ lies in $\Dbct(X)_{\mathcal{S}}$ as well.
	To this end, let $a(K)$ (resp. $b(K)$) denotes the minimal (resp. maximal) index $n$ with $\pH^n(K) \neq 0$, and let $l(K) := b(K) - a(K)$ denote the perverse cohomological length of $K$.
	We prove the claim by induction on $l(K) \geq 0$.
	The base step is when $l(K) = 0$, in which case $K \in \Perv(X)_{\mathcal{S}}[a(K)]$ and the conclusion is already known (as $\Perv(X)_{\mathcal{S}}$ is stable under retracts inside $\Perv(X)$).
	For the inductive step, assume that $l(K) \geq 1$ and that the conclusion is known to hold for all complexes with perverse cohomological length smaller than $l(K)$.
	Form the truncation sequence in $\Dbct(X)$
	\begin{equation*}
		\tau_{< b} K \to K \to \tau_{\geq b} K
	\end{equation*}
	and observe that the direct sum decomposition of $K$
	induces a decomposition of the whole fibre sequence.
	Since $l(\tau_{< b} K) < l(K)$ and $l(\tau_{\geq b} K) = 0$ (by construction), the direct summands $\tau_{< b} K'$ and $\tau_{\geq b} K'$ lie in $\Dbct(X)_{\mathcal{S}}$ by inductive hypothesis.
	Since $\Dbct(X)_{\mathcal{S}}$ is a triangulated subcategory of $\Dbct(X)$, this implies that $K'$ lies in $\Dbct(X)_{\mathcal{S}}$ as well.
	This concludes the proof of the claim.
\end{proof}

\begin{rem}
	It follows that $\Pervgeo^{\rm{A}}(X)$ coincides with the classical abelian category of perverse sheaves of geometric origin, as defined in~\cite[\S~6.2.4]{BBDG}.
\end{rem}

Perverse sheaves of geometric origin have a useful characterization in terms of Voevodsky motives:
\begin{cor}\label{cor:PervgeoAyoub_pH^0Bti}
	The category $\Pervgeo^{\rm{A}}(X)$ coincides with the smallest abelian subcategory of $\Perv(X)$ containing the image of the homological functor
	\begin{equation*}
		\DAct(X) \xrightarrow{\Bti_X^*} \Dbct(X) \xrightarrow{\pH^0} \Perv(X)
	\end{equation*}
	and stable under subquotients and extensions.
\end{cor}
\begin{proof}
	As shown in the proof of Proposition~\ref{prop:Dbgeo_perv_t-str}, $\Pervgeo^{\rm{A}}(X)$ coincides with the smallest abelian subcategory of $\Perv(X)$ containing the perverse sheaves $\pH^n(p_* \Q_{W^{\sigma}})$ with $p: W \to X$ proper and $n \geq 0$ and stable under subquotients and extensions.
	Therefore, in order to conclude that the abelian category in the statement equals $\Pervgeo^{\rm{A}}(X)$, it suffices to show that each of them contains the generators of the other one.
	
	In one direction, it suffices to observe that the above perverse sheaves $\pH^0(p_* \Q_{W^{\sigma}}) = \pH^0 \Bti_X^*(p_* \Q_W)$ obviously lie in the essential image of $\pH^0 \circ \Bti_X^*$.
	
	In the other direction, it suffices to check that the composite functor $\pH^0 \circ \Bti_X^*$ factors through $\Pervgeo^{\rm{A}}(X)$.
	This follows from the fact that $\Bti_X^*$ factors thorough $\Dbgeo(X)$ (by Corollary~\ref{prop:Bti-through-Dbgeo}).
\end{proof}

For future reference, we spell out the behaviour of perverse sheaves of geometric origin with respect to finite extensions of $k$.
As usual, let $k'/k$ be a finite extension, set $X' := X \times_k k'$, and write $e_X \colon X' \to X$ for the corresponding finite étale morphism.
Choose a complex embedding $\sigma' \colon k' \hookrightarrow \C$ extending $\sigma$.

\begin{lem}\label{lem:Perv_X'/X}
	The following hold:
	\begin{enumerate}
		\item There is a canonical equivalence
		\begin{equation}\label{formula:Perv(X)=Perv(X')}
			\Pervgeoemb{\sigma}^{\rm{A}}(X) \xrightarrow{\sim} \Pervgeoemb{\sigma'}^{\rm{A}}(X').
		\end{equation}
		\item If $k'/k$ is Galois, the inverse image functor $e_X^* \colon \Pervgeoemb{\sigma}^{\rm{A}}(X) \rightarrow \Pervgeoemb{\sigma}^{\rm{A}}(X')$ induces a canonical equivalence
		\begin{equation*}
			\Pervgeoemb{\sigma}^{\rm{A}}(X) \xrightarrow{\sim} \Pervgeoemb{\sigma}^{\rm{A}}(X')^{\Gal(k'/k)},
		\end{equation*}
		where the action of $\Gal(k'/k)$ on $\Pervgeoemb{\sigma}^{\rm{A}}(X')$ is induced by its action on $X'$ via $k$-automorphisms.
	\end{enumerate}
\end{lem}
\begin{proof}
	The first statement follows from the fact that the equivalence of Lemma~\ref{lem:Dbgeo_X'/X} clearly respects the perverse $t$-structures. 
	
	The second statement follows from the fact that the assignment $X \mapsto \Pervgeoemb{\sigma}^{\rm{A}}(X)$, regarded as a fibered category over étale morphisms, is a stack.
	
	In detail, since $e_X$ is finite étale, the functor $e_X^*$ identifies objects of $\Pervgeoemb{\sigma}^{\rm{A}}(X)$ with objects of $\Pervgeoemb{\sigma}^{\rm{A}}(X')$ endowed with an isomorphism between their inverse images under the two projections $X' \times_X X' \rightarrow X'$, subject to the usual cocycle condition over the triple fibered product $X' \times_X X' \times_X X'$.
	Similarly, it identifies morphisms in $\Pervgeoemb{\sigma}^{\rm{A}}(X)$ with those morphisms in $\Pervgeoemb{\sigma}^{\rm{A}}(X')$ which are compatible with the additional isomorphisms.
	This holds regardless of whether $k'/k$ is Galois or not.
	The hypothesis that $k'/k$ is Galois implies that the scheme $X' \times_X X' = X' \times_{k'} (k' \times_k k')$ is a disjoint union of $[k'\colon k]$-many copies of $X'$ permuted transitively by $\Gal(k'/k)$.
	Under this identification, equipping an object of $\Pervgeoemb{\sigma}^{\rm{A}}(X')$ with an isomorphism between its two inverse images, subject to the cocycle condition, amounts precisely to turning it into an equivariant object under the $\Gal(k'/k)$-action.
\end{proof}

From now until the end of the present section, let $X$ be a smooth $k$-variety.
Recall from~\cite[\S~4]{BBDG} that, if $X$ is connected of dimension $d \geq 0$, the shift functor $[d] \colon \Dbct(X) \to \Dbct(X)$ identifies the abelian category of local systems on $X$ with a full abelian subcategory of $\Perv(X)$.
One naturally generalizes this to arbitrary smooth $k$-varieties,
by considering each connected component separately.
Throughout this paper, we always view local systems as perverse sheaves:
\begin{nota}\label{nota:Loc(X)}
	We let $\Loc(X)$ denote the full abelian subcategory of shifted local systems inside $\Perv(X)$.
\end{nota}

Let us record the following important property:
\begin{lem}\label{lem:Locp_stable_inside_Perv}
	The abelian subcategory $\Loc(X) \subset \Perv(X)$ is stable under subquotients and extensions.
\end{lem}
\begin{proof}
	Without loss of generality, we may assume that $X$ is connected, say of dimension $d$.
	
	In this case, stability under extensions follows from the facts that extensions in the abelian category $\Perv(X)$ coincide with extensions in the triangulated category $\Dbct(X)$ (see~\cite[Thm.~1.3.6]{BBDG}), that shifting by $-d$ induces a bijection on extensions in $\Dbct(X)$ (being an equivalence), and that ordinary local systems are stable under extensions inside ordinary constructible sheaves.
	
Let us now treat stability under subquotients.
Since $\Loc(X)$ is an abelian subcategory of $\Perv(X)$, it suffices to show its stability under subobjects.
To this end, fix a non-zero object $L \in \Loc(X)$, and let $K$ be a subobject of $L$ in $\Perv(X)$.
To prove that $K$ lies in $\Loc(X)$, we argue by induction on the composition length $n \geq 1$ of $L$ in $\Loc(X)$.
The base step is when $n = 1$:
in this case $L$ is irreducible in $\Loc(X)$, and therefore also in $\Perv(X)$ (by~\cite[Lem.~4.3.3]{BBDG}), so the conclusion is clear.
For the induction step, assume that $n \geq 2$ and that the conclusion is known whenever $L$ is replaced by a local system with composition length $n-1$. 
Fix a filtration with irreducible graded pieces in $\Loc(X)$
\begin{equation*}
	0 = L_0 \subset \dots \subset L_n = L.
\end{equation*}
The intersection $K \cap L_{n-1}$, being a subobject of the local system $L_{n-1}$, belongs to $\Loc(X)$ (by inductive hypothesis).
And the quotient $K/(K \cap L_{n-1})$, being a subobject of the irreducible local system $L/L_{n-1}$, belongs to $\Loc(X)$ (by the base step $n = 1$).
Since $\Loc(X)$ is stable under extensions inside $\Perv(X)$ (as explained above), we conclude that $K$ lies in $\Loc(X)$ as well.
\end{proof}

Note that, with our conventions, the abelian category $\Loc(X)$ carries a shifted tensor product:
it is defined by the formula
\begin{equation*}
	L_1 \otimes^{\dagger} L_2 := (L_1[-d] \otimes L_2[-d])[d]
\end{equation*}
when $X$ is connected of dimension $d$, and using the same formula separately on each connected component of $X$ in the general case.

We are mostly interested in the case when $X$ is geometrically connected over $k$, so that the complex-analytic space $X^{\sigma}$ is connected.
In this case, the abelian category $\Loc(X)$, equipped with the shifted tensor product, is a neutral Tannakian category over $\Q$:
the fibre functor at a point $x \in X^{\sigma}$
is the shifted inverse image functor
\begin{equation*}
	x^*[-d] \colon \Loc(X) \to \vect_{\Q},
\end{equation*}
where $d$ denotes again the dimension of $X$.

\begin{rem}\label{rem:constantLoc}
	Assume $X$ geometrically connected of dimension $d$, and let $a_X \colon X \to \Spec(k)$ denote the structural morphism. 
	The shifted inverse image functor
	\begin{equation*}
		a_X^*[d] \colon \vect_{\Q} = \Perv(\Spec(k)) \to \Perv(X)
	\end{equation*}
	takes values in $\Loc(X)$ and admits a right adjoint, given by the composite
	\begin{equation*}
		\pH^0 \circ a_{X,*}[-d] \colon \Perv(X) \to \Perv(\Spec(k)) = \vect_{\Q}.
	\end{equation*}
	It is a well-known fact from ordinary sheaf theory
	that $a_X^*[d]$ is fully faithful,
	with essential image closed under subquotients.
	This can be expressed purely in terms of the six operations:
	for every $V \in \Perv(\Spec(k))$, the unit morphism
	\begin{equation*}
		\eta \colon V \to \pH^0 a_{X,*}[-d](a_X^*[d] V)
	\end{equation*}
	is invertible and, for every subobject $L \subset a_X^*[d] V$ in $\Perv(X)$, the counit morphism
	\begin{equation*}
		\epsilon \colon a_X^*[d](\pH^0 a_{X,*}[-d] L) \to L
	\end{equation*}
	is invertible.
	The result generalizes naturally to arbitrary smooth $k$-varieties.
\end{rem}

The theory of perverse sheaves of geometric origin gives access to a good notion of local system of geometric origin:  

\begin{nota}
	We let $\Locgeo^{\rm{A}}(X)$ denote the intersection $\Pervgeo^{\rm{A}}(X) \cap \Loc(X)$;
	we write it as $\Locgeoemb{\sigma}^{\rm{A}}(X)$ if we want to stress its dependence on the embedding $\sigma$.
\end{nota}

Note that $\Locgeo^{\rm{A}}$ is an abelian subcategory of $\Pervgeo^{\rm{A}}(X)$, stable under subquotients and extensions (by Lemma \ref{lem:Locp_stable_inside_Perv}).
Moreover, since $\Pervgeo^{\rm{A}}(X)$ is stable under subquotients and extensions inside $\Perv(X)$ (by Proposition~\ref{prop:Dbgeo_perv_t-str}), $\Locgeo^{\rm{A}}(X)$ enjoys the same stability properties inside $\Loc(X)$.
In fact, since constructible complexes of geometric origin are stable under the tensor product (by Proposition~\ref{prop:Dbgeo-6ff}), in the geometrically connected case $\Locgeo^{\rm{A}}(X)$ is even a Tannakian subcategory of $\Loc(X)$.
We defer the study of $\Locgeo^{\rm{A}}(X)$ from a Tannakian viewpoint to Section~\ref{sect:fund_ses_Voe}. 

\begin{ex}\label{ex:Locgeo(k)=vect}
	When $X = \Spec(k)$, there is no difference between perverse sheaves, ordinary constructible sheaves, or local systems.
	In fact, we have canonical monoidal equivalences
	\begin{equation*}
		\Locgeo(\Spec(k)) = \Pervgeo(\Spec(k)) = \vect_{\Q},
	\end{equation*}
	induced by the global sections functor.
	\qed
\end{ex}

The behaviour of local systems of geometric origin with respect to finite extensions of the base field $k$ is analogous to that of perverse sheaves: 

\begin{lem}\label{lem:Locgeo_X'/X}
	With the same notation as in Lemma~\ref{lem:Perv_X'/X}, the following statements hold:
	\begin{enumerate}
		\item
		There is a canonical equivalence
		\begin{equation*}
			\Locgeoemb{\sigma}^{\rm{A}}(X) \xrightarrow{\sim} \Locgeoemb{\sigma'}^{\rm{A}}(X).
		\end{equation*}

		\item
		If $k'/k$ is Galois, the inverse image functor $e_X^* \colon \Locgeoemb{\sigma}^{\rm{A}}(X) \rightarrow \Locgeoemb{\sigma}^{\rm{A}}(X')$ induces a canonical equivalence
		\begin{equation*}
			\Locgeoemb{\sigma}^{\rm{A}}(X) \xrightarrow{\sim} \Locgeoemb{\sigma}^{\rm{A}}(X')^{\Gal(k'/k)},
		\end{equation*}
		where the action of $\Gal(k'/k)$ on $\Locgeoemb{\sigma}^{\rm{A}}(X')$ is induced by its action on $X'$ via $k$-automorphisms.
	\end{enumerate}
\end{lem}

\begin{proof}
	The first statement follows from the fact that the equivalence of Lemma~\ref{lem:Dbgeo_X'/X} respects shifted local systems:
	indeed, it respects dualizable complexes (being symmetric monoidal) as well as the ordinary $t$-structures.
	
	The second statement is a particular case of Lemma~\ref{lem:Perv_X'/X}(2), because a perverse sheaf on $X$ is a shifted local system if and only if this holds for its restriction to any given étale cover of $X$.
\end{proof}

\section{Recollections on perverse Nori motives}\label{sect:IM_cats}

In this section we summarize some highlights of the theory of perverse Nori motives, developed in~\cite{IM19} and~\cite{TerenziNori} building on Nori's theory of mixed motives (see~\cite{HMS17}):
it is a theory of mixed motivic sheaves satisfying the formal properties listed by Beilinson in~\cite[\S~5.10]{Bei87H}.

Throughout, we work over a field $k$ of characteristic $0$
endowed with a complex embedding
$\sigma \colon k \hookrightarrow \C$.
The category of perverse Nori motives over a $k$-variety $X$
is modelled on the abelian category $\Perv(X)$.
We start by reviewing the definition of perverse Nori motives
and their general functoriality.
We then focus on the Tannakian aspects of the theory:
we discuss Nori's motivic Galois group over a field,
and finally, following~\cite{Jacobsen}, we study the motivic Galois group of a general $k$-variety. 

\subsection{Universal property and functoriality}\label{subsect:IM_cat}

For every $k$-variety $X$, the $\Q$-linear abelian category $\MNori(X)$ of \textit{perverse Nori motives} over $X$ is defined in~\cite[\S~2.1]{IM19} as the universal abelian factorization of the homological functor
\begin{equation*}
	\DAct(X) \xrightarrow{\Bti_X^*} \Dbct(X) \xrightarrow{\pH^0} \Perv(X).
\end{equation*}
By construction, it comes equipped with a homological functor
\begin{equation*}
	h_X \colon \DAct(X) \rightarrow \MNori(X)
\end{equation*}
and with a faithful exact functor
\begin{equation*}
	\iota_X \colon \MNori(X) \rightarrow \Perv(X)
\end{equation*}
such that the homological functor $\pH^0 \circ \Bti_X^*$ factors as the composite
\begin{equation*}
	\pH^0 \circ \Bti_X^*\colon \DAct(X)
	\xrightarrow{h_X} \MNori(X)
	\xrightarrow{\iota_X} \Perv(X).
\end{equation*}
By design, $\MNori(X)$ is initial among all abelian categories
through which such a factorization is possible: 
given another abelian category $\Abelian(X)$ endowed with a homological functor $h'_X \colon \DAct(X) \rightarrow \Abelian(X)$,
a faithful exact functor
$\iota'_X \colon \Abelian(X) \rightarrow \Perv(X)$,
and a natural isomorphism between functors $\DAct(X) \rightarrow \Perv(X)$
\begin{equation*}
	\kappa_X \colon \pH^0 \circ \Bti_X^* = \iota_X \circ h_X \xrightarrow{\sim} \iota'_X \circ h'_X,
\end{equation*}
there exist a unique faithful exact functor
\begin{equation*}
	o_X \colon \MNori(X) \rightarrow \Abelian(X)
\end{equation*}
providing a factorization of $h'_X$ of the form
\begin{equation*}
	h'_X \colon \DAct(X) \xrightarrow{h_X} \MNori(X) \xrightarrow{o_X} \Abelian(X),
\end{equation*}
and a unique natural isomorphism between functors $\MNori(X) \rightarrow \Perv(X)$
\begin{equation*}
	\tilde{\kappa}_X \colon \iota_X \xrightarrow{\sim} \iota'_X \circ o_X
\end{equation*}
making the diagram of functors $\DAct(X) \to \Perv(X)$
\begin{equation*}
	\begin{tikzcd}
		\iota_X \circ h_X \arrow{rr}{\tilde{\kappa}_X} \arrow{drr}{\kappa_X} && \iota'_X \circ o_X \circ h_X \arrow[equal]{d} \\
		&& \iota'_X \circ h'_X
	\end{tikzcd}
\end{equation*}
commute.
In fact, in this situation, $\MNori(X)$ can be identified with 
the universal abelian factorization of the homological functor
$h'_X$
(see~\cite[Prop.~1.10, Cor.~2.9]{TerenziFunctoriality}).

For instance, since the Betti realization takes values in complexes of geometric origin (by Corollary~\ref{prop:Bti-through-Dbgeo}), one could take $\Abelian(X) := \Pervgeo^{\rm{A}}(X)$ - that is, one could replace $\Perv(X)$ by $\Pervgeo^{\rm{A}}(X)$ in the definition of $\MNori(X)$.
In fact, this observation suggests considering a possible alternative notion of perverse sheaf of geometric origin:
\begin{nota}\label{nota:PervgeoNori}
	For every $k$-variety $X$, we let $\Pervgeo^{\rm{N}}(X)$ denote the smallest abelian subcategory of $\Perv(X)$ containing the essential image of $\iota_X \colon \MNori(X) \to \Perv(X)$ and stable under subquotients;
	we write it as $\Pervgeoemb{\sigma}^{\rm{N}}(X)$ if we want to stress its dependence on the embedding $\sigma$.
\end{nota}

Since $\Pervgeo^{\rm{A}}(X)$ is stable under subquotients inside $\Perv(X)$ (by Proposition~\ref{prop:Dbgeo_perv_t-str}), we have the inclusion $\Pervgeo^{\rm{N}}(X) \subset \Pervgeo^{\rm{A}}(X)$.
At this stage, it is not clear how far this inclusion is from an equivalence:
the point is that, while we already know that $\Pervgeo^{\rm{A}}(X)$ is stable under extensions inside $\Perv(X)$ (again by Proposition~\ref{prop:Dbgeo_perv_t-str}), no such stability is required in the definition of $\Pervgeo^{\rm{N}}(X)$. 

\begin{rem}\label{rem:DA(X)^heart=MNori(X)}
	The triangulated category $\DAct(X)$ is expected to carry a $t$-structure compatible with the perverse $t$-structure on $\Dbct(X)$ via the Betti realization.
	If this \textit{motivic perverse $t$-structure} on $\DAct(X)$ exists, then its heart $\DAct(X)^{\heartsuit}$ enjoys the same universal property as $\MNori(X)$, and so the exact functor
	\begin{equation*}
		\DAct(X)^{\heartsuit} \subset \DAct(X) \xrightarrow{h_X} \MNori(X)
	\end{equation*}
	is necessarily an equivalence.
	In the lack of a motivic perverse $t$-structure on $\DAct(X)$,
	the homological functor $h_X$ should be regarded as a formal way to define the degree $0$ perverse cohomology object of a motivic complex.
	By the very construction of universal abelian factorizations,
	the abelian category $\MNori(X)$
	is generated by the image of $h_X$ under subquotients
	(see~\cite[\S~1]{IM19}).
\end{rem}

The faithful exact functor $\iota_X \colon \MNori(X) \to \Perv(X)$
extends to a conservative triangulated functor
\begin{equation*}
\iota_X \colon \Db(\MNori(X)) \rightarrow \Db(\Perv(X)) = \Dbct(X),
\end{equation*}
where the last passage witnesses Beilinson's equivalence~\cite[Thm.~1.3]{BeilinsonEquivalence}.
These functors can be used to lift the rich functoriality of the constructible derived categories to the setting of Nori motives.

\begin{thm}[{\cite[Thm.~5.1]{IM19}}, {\cite[Thm.~2.1, Thm.~5.1]{TerenziNori}}]\label{thm:six_ops_MNori}
As $X$ varies among quasi-projective $k$-varieties,
the triangulated categories $\Db(\MNori(X))$
are endowed with a canonical six functor formalism,
and the triangulated functors $\iota_X$
commute with the six operations.
\end{thm}

The proof of this result is based on the method of~\cite{AyoThesis} and~\cite{AyoBettiReal}, which is tailored to the quasi-projective setting. 
The four operations of type $f^*$, $f_*$, $f_!$ and $f^!$ are constructed in~\cite{IM19}, while the closed monoidal structure is constructed in~\cite{TerenziNori}.
In both works, the basic principle is to exploit the universal property of perverse Nori motives in order to lift $t$-exact functors on the constructible derived categories and natural transformations thereof to the motivic level.
For future reference, let us illustrate this principle in a simple case:
\begin{ex}[{\cite[\S~2.5]{IM19}}]\label{ex:j^*}
	Let $j \colon U \hookrightarrow X$ be an open immersion of (quasi-projective) $k$-varieties.
	By~\cite[Prop.~4.2.5]{BBDG}, the inverse image functor
	\begin{equation}\label{formula:j^*-D^b_c}
		j^* \colon \Dbct(X) \rightarrow \Dbct(U)
	\end{equation}
	is $t$-exact for the perverse $t$-structures, and so induces an exact functor
	\begin{equation*}
		j^* \colon \Perv(X) \rightarrow \Perv(U).
	\end{equation*}
	In fact, as a consequence of~\cite[Thm.~1]{Vologodsky}, the trivial derived functor
	\begin{equation*}
		j^* \colon \Db(\Perv(X)) \rightarrow \Db(\Perv(U))
	\end{equation*}
	recovers the original triangulated functor \eqref{formula:j^*-D^b_c} modulo Beilinson's equivalences.
	Now consider the analogous inverse image functor of Voevodsky motives
	\begin{equation}\label{formula:j^*-DA}
		j^* \colon \DAct(X) \rightarrow \DAct(U).
	\end{equation}
	By~\cite[Lem.~4.3]{AyoBettiReal}, the upper half of the diagram
	\begin{equation}\label{dia:j^*-Bti}
		\begin{tikzcd}
			\DAct(X) \arrow{rr}{j^*} \arrow{d}{\Bti_X^*} && \DAct(U) \arrow{d}{\Bti_U^*} \\
			\Dbct(X) \arrow{rr}{j^*} \arrow{d}{\pH^0} && \Dbct(U) \arrow{d}{\pH^0} \\
			\Perv(X) \arrow{rr}{j^*} && \Perv(U)
		\end{tikzcd}
	\end{equation}
	commutes up to canonical natural isomorphism, and thus the same holds for the outer rectangle.
	Applying~\cite[Prop.~2.5]{TerenziFunctoriality}, we get a canonical exact functor
	\begin{equation}\label{formula:j^*-MNori}
		j^* \colon \MNori(X) \rightarrow \MNori(U)
	\end{equation}
	making both halves of the diagram
	\begin{equation}\label{dia:j^*-MNori}
		\begin{tikzcd}
			\DAct(X) \arrow{rr}{j^*} \arrow{d}{h_X} && \DAct(U,\Q) \arrow{d}{h_U} \\
			\MNori(X) \arrow{rr}{j^*} \arrow{d}{\iota_X} && \MNori(U) \arrow{d}{\iota_U} \\
			\Perv(X) \arrow{rr}{j^*} && \Perv(U)
		\end{tikzcd}
	\end{equation}
	commute up to natural isomorphism.
	Hence, the induced triangulated functor
	\begin{equation*}
		j^* \colon \Db(\MNori(X)) \rightarrow \Db(\MNori(U))
	\end{equation*}
	makes the diagram
	\begin{equation}\label{dia:j^*-iota}
		\begin{tikzcd}
			\Db(\MNori(X)) \arrow{rr}{j^*} \arrow{d}{\iota_X} && \Db(\MNori(U)) \arrow{d}{\iota_U} \\
			\Db(\Perv(X)) \arrow{rr}{j^*} && \Db(\Perv(U))
		\end{tikzcd}
	\end{equation}
	commute up to the induced natural isomorphism.
	There is a similar lifting principle for natural transformations, stated in~\cite[Prop.~3.4]{TerenziFunctoriality}. 
	In particular, the usual connection isomorphisms on the constructible derived categories with respect to composition of open immersions lift to analogous natural isomorphisms on perverse Nori motives.
	In this way, the assignment $X \mapsto \Db(\MNori(X))$ becomes a triangulated fibered category over the category of open immersions between $k$-varieties.
	\qed
\end{ex}

\begin{rem}\label{rem:j^*-unique}
	Let us stress that~\eqref{formula:j^*-MNori} is the unique exact functor making the upper half of~\eqref{dia:j^*-MNori} commute on the nose;
	the natural isomorphism filling the lower half of~\eqref{dia:j^*-MNori} is then completely determined by its compatibility with the natural isomorphism filling~\eqref{dia:j^*-Bti}.
	In other words, while the existence of the exact functor~\eqref{formula:j^*-MNori} depends on the existence of~\eqref{formula:j^*-D^b_c} (and on the commutativity of~\eqref{dia:j^*-Bti}), its actual expression only depends on the functor~\eqref{formula:j^*-DA}.
\end{rem}

The same discussion of Example~\ref{ex:j^*} applies to any class of $t$-exact functors on the constructible derived categories which is part of the six functor formalism and is stable under composition:
in particular, it applies to shifted inverse images under smooth morphisms and to direct images under closed immersions.
There is a similar lifting result for multilinear functors which are $t$-exact in each variable, such as the external tensor product (see~\cite[\S~4]{TerenziFunctoriality}).
However, lifting the external tensor product to the motivic level
requires working with an alternative presentation of $\MNori(X)$
(see \S~\ref{subsect:MNori(k)} below
for the case $X = \Spec(k)$,
and~\cite[Thm.~1.12]{TerenziNori} for the general statement).

The key tool to extend the functors obtained out of the universal property to a complete six functor formalism is provided by Beilinson's gluing functors, which include unipotent nearby and vanishing cycles (see~\cite[\S~3]{IM19}).
Although these functors are typically not regarded as building blocks of the six functor formalism but rather as an outcome of the latter, they can be lifted to perverse Nori motives by direct application of the universal property:
\begin{prop}[{\cite[\S~3.5]{IM19}}, {\cite[\S~1]{TerenziNori}}]
The abelian categories $\MNori(X)$ are endowed with Beilinson's gluing functors, and they satisfy Beilinson's perverse gluing formalism.
\end{prop}

The gluing functors play a crucial role in at least three steps of~\cite{IM19}:
showing that inverse images under closed immersions exist,
constructing the connection isomorphisms for composition of inverse image functors, and showing that direct images under general morphisms exist.
The perverse gluing formalism for Nori motives, as stated in~\cite[Prop.~1.7]{TerenziNori}, is used to show that the aforementioned alternative presentation of $\MNori(X)$ really defines the same category.

\subsection{Nori's motivic Galois group}\label{subsect:MNori(k)}

By~\cite[Prop.~2.11]{IM19}, when $X = \Spec(k)$ Ivorra--Morel's definition of perverse Nori motives recovers the $\Q$-linear abelian category $\MNori(k)$ of Nori motives.
This category was originally introduced by Nori in~\cite{Nori}:
by design, it is the finest possible $\Q$-linear abelian category computing the singular cohomology groups of all pairs of $k$-varieties.
Adopting Ivorra--Morel's definition, we obtain a factorization of the homological functor $H^0 \circ \Bti_k^* \colon \DAct(k) \to \vect_{\Q}$ as
\begin{equation*}
	H^0 \circ \Bti_k^* \colon \DAct(k) \xrightarrow{h_k} \MNori(k) \xrightarrow{\iota_k} \vect_{\Q}.
\end{equation*}
Here, we identify $\Perv(\Spec(k))$ with $\vect_{\Q}$ as in Example~\ref{ex:Locgeo(k)=vect}.
Note that the forgetful functor
\begin{equation*}
	\iota_k \colon \MNori(k) \to \vect_{\Q}
\end{equation*} 
sends non-zero objects to non-zero objects, being faithful and exact.

\begin{ex}\label{ex:unit-motive}
	The simplest non-zero Nori motive is the cohomology of the point
	\begin{equation*}
		\Q_k := h_k^0(\Q_k) \in \MNori(k),
	\end{equation*}
	called the \textit{unit motive}.
	Note that the $\Q$-vector space $\End_{\MNori(k)}(\Q_k)$ has dimension $1$:
	it cannot be $0$ (since $\Q_k$ is non-zero) nor bigger than $1$ (since the underlying vector space $\iota_k(\Q_k) = \Q$ is $1$-dimensional).
	\qed
\end{ex}

\begin{ex}\label{ex:Artin-motive}
	More interesting Nori motives arise from the cohomology of zero-dimensional $k$-varieties.
	For instance, let $k'/k$ be a finite field extension, and write $e \colon \Spec(k') \to \Spec(k)$ for the corresponding finite étale morphism.
	Consider the object
	\begin{equation*}
		\Q_{k'/k} := e_* \Q_{k'} = e_* e^* \Q_k,
	\end{equation*}
	usually called an \textit{Artin motive}.
	The underlying $\Q$-vector space $\iota_k(\Q_{k'/k}) = \Bti_k^*(e_* \Q_{k'})$ is the $0$-th cohomology group $H^0(\Spec(k')^{\sigma};\Q)$.
	The complex-analytic space $\Spec(k')^{\sigma}$ consists of $[k' \colon k]$ points, indexed by the set $\Hom_{\sigma}(k',\C)$ of complex embeddings of $k'$ extending $\sigma$.
	Hence, its cohomology can be naturally identified with $\Q^{\Hom_{\sigma}(k',\C)}$.
	
	We claim that, if $k'/k$ is Galois, the $\Q$-algebra $\End_{\MNori(k)}(\Q_{k'/k})$ is canonically isomorphic to the group algebra $\Q[\Gal(k'/k)^\rom{op}]$.
	To see this, consider the group homomorphism
	\begin{equation*}
		\Gal(k'/k)^\rom{op} \to \Aut_{\MNori(k)}(\Q_{k'/k})
	\end{equation*}
	sending an element $t \in \Gal(k'/k)$ to the composite isomorphism in $\MNori(k)$
	\begin{equation*}
		\alpha_{\phi} \colon e_* e^* \Q_k \xrightarrow{\eta} e_* f_* f^* e^* \Q_k = (ef)_* (ef)^* \Q_k = e_* e^* \Q_k,
	\end{equation*}
	where $f$ denotes the automorphism of $\Spec(k')$ induced by $t^{-1}$.
	Under the faithful exact functor $\iota_k$, this matches with the right $\Gal(k'/k)$-action on $H^0(\Spec(k')^{\sigma};\Q)$ induced by the left action
	\begin{equation*}
		\Gal(k'/k) \times \Hom_{\sigma}(k',\C) \to \Hom_{\sigma}(k',\C), \quad (r,\sigma') \mapsto \sigma' \circ r^{-1}.
	\end{equation*}
	In order to check that the induced $\Q$-algebra homomorphism
	\begin{equation*}
		\Q[\Gal(k'/k)^\rom{op}] \to \End_{\MNori(k)}(\Q_{k'/k})
	\end{equation*}
	is an isomorphism, the trick is to extend scalars to $\Q_{\ell}$ and use the factorization of the forgetful functor 
	\begin{equation*}
		\iota_{k,\ell} \colon \MNori(k)_{\Q_{\ell}} \to \vect_{\Q_{\ell}}
	\end{equation*} 
	through the category of $\ell$-adic $\Gal(\bar{k}/k)$-representations
	(see~\cite[Lem.~5.2.9]{ExpMot} for a detailed argument).
	
	Although the choice to work with the opposite group $\Gal(k'/k)^\rom{op}$ in place of $\Gal(k'/k)$ seems quite unnatural here, we adopted this convention is order to get a cleaner result in Example~\ref{ex:ArtinMot_k} below.
	\qed
\end{ex}

The main structural result about $\MNori(k)$ is the existence of a canonical tensor product making it a neutral Tannakian category over $\Q$ with fibre functor the forgetful functor $\iota_k \colon \MNori(k) \to \vect_{\Q}$;
the unit for this monoidal structure is precisely the unit motive $\Q_k$ of Example~\ref{ex:unit-motive}.
The monoidal structure on $\MNori(k)$, constructed by Nori, represented one of the most interesting features of his theory (see~\cite[Thm.~9.3.4, Thm.~9.3.10]{HMS17}).
From the perspective of perverse Nori motives, this result looks like a particular case of Theorem~\ref{thm:six_ops_MNori},
but this is quite deceitful: 
indeed, the geometric insights behind Nori's original proof play a crucial role also in the identification of Nori's category $\MNori(k)$ with Ivorra--Morel's category $\MNori(\Spec(k))$.

For future reference, let us briefly comment on the tensor structure of $\MNori(k)$ in terms of Ivorra--Morel's definition.
The homological functor $H^0 \circ \Bti_k^*$ is not monoidal, since $\Bti_k^*$ is monoidal but $H^0$ is not.
This prevents one from defining the tensor product on $\MNori(k)$ by the general lifting principles of universal abelian factorizations.
As a workaround, one is led to consider the full additive subcategory
\begin{equation*}
	\DAct(k)^0 := \left\{A \in \DAct(k) \; | \; \Bti_k^*(A) \in \vect_{\Q}\right\} \subset \DAct(k),
\end{equation*}
which is nothing but the candidate for the heart of the conjectural motivic $t$-structure.
By design, the restriction of $H^0 \circ \Bti_k^*$ to $\DAct(k)^0$ reduces to $\Bti_k^*$ and so is monoidal.
Hence, the associated universal abelian factorization $\MNori(k)^0$ inherits a well-defined tensor product.
Now the problem becomes to show that the natural faithful exact functor
\begin{equation*}
	\MNori(k)^0 \to \MNori(k)
\end{equation*}
is in fact an equivalence.
Nori's geometric insights play a major role in the proof of this fact.

\begin{nota}
	We let $\GmotNo(k)$ denote the Tannaka dual of $\MNori(k)$ with respect to the fibre functor $\iota_k$, and we call it \textit{Nori's motivic Galois group} of $k$.
\end{nota}

In concrete situations, one is typically interested in the Tannakian subcategory $\langle M \rangle^{\otimes}$ of $\MNori(k)$ generated by a single motive $M$:
as usual, by this we mean the smallest full abelian subcategory of $\MNori(k)$ containing $M$ and stable under subquotients, tensor products, and duals.
By~\cite[Prop.~2.21(a)]{deligne-milne}, the inclusion of $\langle M \rangle^{\otimes}$ into $\MNori(k)$ induces a surjection of $\GmotNo(k)$ onto the Tannaka dual of $\langle M \rangle^{\otimes}$.
The situation is particularly simple in the case of Artin motives:
\begin{ex}\label{ex:ArtinMot_k}
	Let $k'/k$ be a finite Galois extension, and consider the Artin motive $\Q_{k'/k}$ from Example~\ref{ex:Artin-motive}.
	By~\cite[\S\S~9.4-9.5]{HMS17}, the Tannakian subcategory $\langle \Q_{k'/k} \rangle^{\otimes} \subset \MNori(k)$ is equivalent to $\Rep_{\Q}(\Gal(k'/k))$, with
	the generator $\Q_{k'/k}$ corresponding to the regular representation;
	this equivalence depends on the choice of an element $\sigma' \in \Hom_{\sigma}(k',\C)$.
	This determines a surjection between Tannaka dual groups
	\begin{equation}\label{eq:Gmot(k)_to_Gal(k'/k)}
		\GmotNo(k) \twoheadrightarrow \Gal(k'/k),
	\end{equation}
	where $\Gal(k'/k)$ is regarded as a finite algebraic group over $\Q$.
	
	Following~\cite[Cor.~8.1.17]{HMS17}, the starting point is to compute the centralizer $E$ of $\End_{\MNori(k)}(\Q_{k'/k})$ inside $\End_{\Q}(\iota_k(\Q_{k'/k}))$.
	To this end, consider the canonical $\Q$-algebra isomorphism
	\begin{equation*}
		\Q[\Gal(k'/k)^\rom{op}] \xrightarrow{\sim} \End_{\MNori(k)}(\Q_{k'/k})
	\end{equation*}
	obtained in Example~\ref{ex:Artin-motive}.
	The vector space $\iota_k(\Q_{k'/k}) = H^0(\Spec(k')^{\sigma};\Q)$ is a free $\Q[\Gal(k'/k)^\rom{op}]$-module of rank $1$ (because the $\Gal(k'/k)$-action on $\Hom_{\sigma}(k',\C)$ is simply transitive).
	This implies that $E$ is isomorphic to the opposite algebra $\Q[\Gal(k'/k)^\rom{op}]^\rom{op} = \Q[\Gal(k'/k)]$:
	more precisely, the choice of an element $\sigma'_0 \in \Hom_{\sigma}(k',\C)$ induces a $\Q[\Gal(k'/k)^\rom{op}]$-module isomorphism
	\begin{equation*}
		\iota_k(\Q_{k'/k}) \simeq \Q[\Gal(k'/k)^\rom{op}],
	\end{equation*}
	and the centralizer for the left-multiplication action of $\Q[\Gal(k'/k)^\rom{op}]$ on itself is canonically isomorphic to the opposite algebra.
	From now on, fix such an element $\sigma'_0$.
	
	Consider the linear dual $E^{\lor}$ with its natural bialgebra structure. 
	The isomorphism $E \simeq \Q[\Gal(k'/k)]$ defined by $\sigma'_0$ induces an isomorphism
	\begin{equation*}
		\Spec(E^{\lor}) \simeq \Spec(\Q[\Gal(k'/k)]^{\lor}) = \Gal(k'/k).
	\end{equation*}
	Note that the forgetful functor $\iota_k \colon \langle \Q_{k'/k} \rangle^{\otimes} \to \vect_{\Q}$ factors through a monoidal functor
	\begin{equation*}
		\langle \Q_{k'/k} \rangle^{\otimes} \to \Rep_{\Q}(\Spec(E^{\lor}))
	\end{equation*}
	sending the generator $\Q_{k'/k}$ to the regular representation $\Q^{\Hom_{\sigma}(k',\C)}$.
	We claim that the latter functor is an equivalence.
	Since the abelian category $\Rep_{\Q}(\Gal(k'/k))$ is semisimple, this is the case as soon as the map
	\begin{equation*}
		\Hom_{\MNori(k)}(\Q_{k'/k}^{\otimes n_1},\Q_{k'/k}^{\otimes n_2}) \to \Hom_{\Q}^{\Gal(k'/k)}((\Q^{\Hom_{\sigma}(k,\C)})^{\otimes n_1},(\Q^{\Hom_{\sigma}(k,\C)})^{\otimes n_2})
	\end{equation*}
	is bijective for all $n_1, n_2 \in \Z$.
	In the case when $n_1 = 1 = n_2$, the computation in Example~\ref{ex:Artin-motive} shows that this condition is indeed satisfied.
	In order to conclude that the same holds in general, it suffices to prove the formula
	\begin{equation*}
		\Q_{k'/k}^{\otimes 2} \simeq \Q_{k'/k}^{\oplus [k'\colon k]}.
	\end{equation*}
	To this end, start from the canonical isomorphism
	\begin{align*}
		\Q_{k'/k} \otimes \Q_{k'/k} &:= e_* \Q_{k'} \otimes \Q_{k'/k} \\
		&= e_* (\Q_{k'} \otimes e^* \Q_{k'/k}) && \textup{(by the projection formula)} \\
		&= e_* e^* \Q_{k'/k}.
	\end{align*}
	Now form the Cartesian square
	\begin{equation*}
		\begin{tikzcd}
			\Spec(k' \otimes_k k') \arrow{rr}{e'_1} \arrow{d}{e'_2} && \Spec(k') \arrow{d}{e} \\
			\Spec(k') \arrow{rr}{e} && \Spec(k),
		\end{tikzcd}
	\end{equation*}
	and regard $\Spec(k' \otimes_k k')$ as a $k'$-variety via the second projection $e'_2$.
	Since $k'/k$ is Galois, the scheme $\Spec(k' \otimes_k k')$ decomposes as the disjoint union of $[k'\colon k]$ copies of $\Spec(k')$, each of which maps isomorphically to $\Spec(k')$ via the first projection $e'_1$.
	Hence, we have an isomorphism in $\MNori(\Spec(k'))$
	\begin{equation}\label{eq:iso_Spec(k'_k_k')}
		e_{1,*} \Q_{k' \otimes_k k'} \simeq \Q_{k'}^{\oplus [k'\colon k]}.
	\end{equation}
	We deduce the isomorphism in $\MNori(\Spec(k'))$
	\begin{align*}
		e^* \Q_{k'/k} &= e^* e_* \Q_{k'} \\
		&= e'_{1,*} {e'_2}^* \Q_{k'} && \textup{(by proper base-change)} \\
		&= e'_{1,*} \Q_{k' \otimes_k k'} && \textup{(as ${e'_2}^* \Q_{k'} = \Q_{k' \otimes_k k'}$)} \\
		&\simeq \Q_{k'}^{\oplus [k'\colon k]}, && \textup{(by \eqref{eq:iso_Spec(k'_k_k')})}
	\end{align*}
	which implies the sought-after formula.
	\qed
\end{ex}

\begin{rem}\label{rem:ArtinMot_two-Gal-actions}
	We have defined two distinct $\Gal(k'/k)$-actions on the $\Q$-vector space $\iota_k(\Q_{k'/k})$:
	one determined by the $\Gal(k'/k)$-action on the motive $\Q_{k'/k}$ (in Example~\ref{ex:Artin-motive}), and one defined by Tannakian formalism (in Example~\ref{ex:ArtinMot_k}).
	Recall that the first action is canonical, whereas the second one depends on the choice of an element $\sigma' \in \Hom_{\sigma}(k',\C)$. 
	Under the isomorphism between $\iota_k(\Q_{k'/k}) = \Q^{\Hom_{\sigma}(k',\C)}$ and $\Q^{\Gal(k'/k)}$ determined by such a choice, they are induced by the left-multiplication and the right-multiplication actions of $\Gal(k'/k)$ on $\Q[\Gal(k'/k)]$, respectively.
	In particular, they commute with one another.
\end{rem}

The above discussion leads one to study the difference between the motivic Galois groups of $k$ and $k'$ in terms of classical Galois theory.
This requires some care, since the natural way to define the motivic Galois group of $k'$ is by regarding $k'$ rather than $k$ as the base field of the theory.

\begin{nota}\label{nota:MNori_emb}
	For sake of clarity, we write $\MNori_{\sigma}(X)$ in place of $\MNori(X)$ when we want to stress its dependence on the base field $k$ (and on the chosen complex embedding $\sigma$).
\end{nota}

Until the end of the present subsection, let $k'/k$ be a finite extension, and fix a complex embedding $\sigma' \colon k' \hookrightarrow \C$ extending $\sigma$.

\begin{lem}\label{lem:MNori_X'/X}
	Let $X$ be a $k$-variety; 
	set $X' := X \times_k k'$, and write $e_X \colon X' \rightarrow X$ for the corresponding finite étale morphism.
	Then:
	\begin{enumerate}
		\item The abelian categories $\MNori_{\sigma'}(X')$ and $\MNori_{\sigma}(X')$ (defined by regarding $X'$ as a $k'$-variety or as a $k$-variety, respectively) are canonically equivalent.
		\item If $k'/k$ is Galois, the inverse image functor $e_X^* \colon \MNori_{\sigma}(X) \rightarrow \MNori_{\sigma}(X')$ induces a canonical equivalence
		\begin{equation*}
			\MNori_{\sigma}(X) \xrightarrow{\sim} \MNori_{\sigma}(X')^{\Gal(k'/k)},
		\end{equation*}
		where the action of $\Gal(k'/k)$ on $\MNori_{\sigma}(X')$ is induced by its action on $X'$ via $k$-automorphisms.
	\end{enumerate}
\end{lem}
\begin{proof}
	The first statement is a particular case of~\cite[Prop.~6.11]{IM19} 
	(see also~\cite[Prop.~1.15]{TerenziNori} for a related result on the level of generic points).
	
	The second statement follows formally from the fact that perverse Nori motives form a stack for the étale topology, as noted in~\cite[Prop.~2.7]{IM19}:
	it suffices to repeat the argument for Lemma~\ref{lem:Perv_X'/X}(2) word-by-word.
\end{proof}

For the moment, we are interested in using this result only in the case when $X = \Spec(k)$.
As part of the structure of a six functor formalism, the inverse image functor
\begin{equation*}
	e^* \colon \MNori(k) \to \MNori(k')
\end{equation*}
is canonically monoidal.
As a consequence of Remark~\ref{rem:f_finét}, the direct image functor
\begin{equation*}
	e_* \colon \MNori(k') \to \MNori(k)
\end{equation*}
is both left adjoint and right adjoint to $e^*$, and the unit natural transformations
\begin{equation*}
	\eta \colon M \to e_* e^* M, \qquad \eta \colon M' \to e^* e_* M'
\end{equation*}
are both split monomorphisms (see also~\cite[\S\S~2.4-2.5]{IM19}).
By Lemma~\ref{lem:Bti_k=Bti_k-e^*}, the diagram of monoidal functors
\begin{equation*}
\begin{tikzcd}
\DAct(k) \arrow{rr}{e^*} \ar[dr, "\Bti_{\sigma,k}^*" ']
	&& \DAct(k') \arrow{dl}{\Bti_{\sigma',k'}^*}
	\\			
& \Db(\vect_{\Q})
\end{tikzcd}
\end{equation*}
commutes up to monoidal natural isomorphism.
Composing with $H^0 \colon \Db(\vect_{\Q}) \to \vect_{\Q}$ and passing to the universal abelian factorizations, we deduce that the diagram of monoidal functors
\begin{equation*}
\begin{tikzcd}
\MNori(k) \arrow{rr}{e^*} \arrow{dr}[']{\iota_k}
	&& \MNori(k') \arrow{dl}{\iota_{k'}}
	\\
& \vect_{\Q}
\end{tikzcd}
\end{equation*}
commutes up to monoidal natural isomorphism as well.
Thus, we get a homomorphism of Tannaka dual groups
\begin{equation}\label{eq:Gmot(k')_to_Gmot(k)}
	\GmotNo(k') \rightarrow \GmotNo(k).
\end{equation}
The difference between the two sides is indeed measured by classical Galois theory:

\begin{prop}\label{prop:ses_Gmot-Gmot-Gal}
	If $k'/k$ is a Galois extension, the sequence of pro-algebraic groups
	\begin{equation}\label{eq:ses_Gmot(k'/k)}
		1 \to \GmotNo(k') \to \GmotNo(k) \to \Gal(k'/k) \to 1
	\end{equation}
	defined by the homomorphisms~\eqref{eq:Gmot(k)_to_Gal(k'/k)} and~\eqref{eq:Gmot(k')_to_Gmot(k)} is exact.
\end{prop}

\begin{proof}
	The analogous statement for the algebraic closure $\bar{k}$ in place of $k'$ is~\cite[Thm.~9.1.16]{HMS17};
	the proof of that statement (see~\cite[\S~9.5]{HMS17}) contains all the ingredients to prove the version that we have stated.
	Since that argument there is written down in terms of Nori's definition of Nori motives, let us spell it out here in terms of Ivorra--Morel's definition:
	we have to show that the homomorphism $\GmotNo(k') \rightarrow \GmotNo(k)$ is an observable closed immersion and that its image coincides with the kernel of the projection to $\Gal(k'/k)$.
	
	The first property amounts to saying that every object $M' \in \MNori(k')$ is a subobject of some object in the image of $e^*$ (see~\cite[Cor.~A.10]{DAE22}).
	To check that this condition is satisfied, one can use the unit monomorphism $\eta \colon M' \hookrightarrow e^* e_* M'$.
	The assumption that $k'/k$ is Galois does not play any role here.
	
	The second property amounts to saying that the Tannakian subcategory $\langle \Q_{k'/k} \rangle^{\otimes} \subset \MNori(k)$ contains exactly those objects of $\MNori(k)$ whose underlying $\GmotNo(k')$-representation is trivial (see~\cite[Prop.~A.13]{DAE22}).
	In one direction, it suffices to show that the object $e^* \Q_{k'/k} \in \MNori(k')$ is a trivial $\GmotNo(k')$-representation, which follows from the isomorphism 
	\begin{equation*}
		e^* \Q_{k'/k} \simeq \Q_{k'}^{\oplus [k'\colon k]}
	\end{equation*} 
	established at the end of Example~\ref{ex:ArtinMot_k}.
	In the other direction, if an object $M \in \MNori(k)$ satisfies $e^* M \simeq \Q_{k'}^{\oplus m}$ for some $m \geq 0$, then the object $e_* e^* M$ satisfies
	\begin{equation*}
			e_* e^* M \simeq e_* \Q_{k'}^{\oplus m} = \Q_{k'/k}^{\oplus m}
	\end{equation*}
	and thus belongs to $\langle \Q_{k'/k} \rangle^{\otimes}$.
	Using the unit monomorphism $\eta \colon M \hookrightarrow e_* e^* M$, we deduce that $M$ belongs to $\langle \Q_{k'/k} \rangle^{\otimes}$ as well.
\end{proof}

\subsection{Motivic local systems}\label{subsect:fundses_MNori}

In this subsection, we discuss the Tannakian aspects of the theory of Nori motives over general smooth $k$-varieties.
Just like perverse Nori motives are modelled on perverse sheaves,
motivic local systems are modelled on local systems.
Throughout, we let $X$ be a smooth, quasi-projective $k$-variety.

\begin{nota}\label{nota:NMLoc(X)}
	\ 
	\begin{itemize}
		\item We say that an object $M \in \MNori(X)$ is a \textit{motivic local system} if the perverse sheaf $\iota_X(M) \in \Perv(X)$ belongs to $\Loc(X)$.
		We let $\NMLoc(X)$ denote the full subcategory of $\MNori(X)$ spanned by the motivic local systems.
		\item Let $\Locgeo^{\rm{N}}(X)$ denote the smallest abelian subcategory of $\Loc(X)$ containing the essential image of the forgetful functor $\iota_X \colon \NMLoc(X) \to \Loc(X)$ and stable under subquotients.
	\end{itemize}
\end{nota}

Recall that $\Loc(X)$ is stable under subquotients and extensions inside $\Perv(X)$ (by Lemma~\ref{lem:Locp_stable_inside_Perv}).
This implies that $\NMLoc(X)$ is stable under subquotients and extensions inside $\MNori(X)$.
Moreover, it implies that $\Locgeo^{\rm{N}}(X)$ is stable under subquotients inside $\Pervgeo^{\rm{N}}(X)$ (see Notation~\ref{nota:PervgeoNori}).
In fact, there is more to be said:

\begin{prop}\label{prop:LocgeoNori_stable_ext}
	The abelian category $\Locgeo^{\rm{N}}(X)$ is stable under extensions inside $\Loc(X)$.
\end{prop}

For the proof, we need to understand the behaviour of $\Locgeo^{\rm{N}}(X)$ with respect to finite extensions of the base field $k$, along the lines of Lemma~\ref{lem:Locgeo_X'/X}.
Note that the inclusion $\Pervgeo^{\rm{N}}(X) \subset \Pervgeo^{\rm{A}}(X)$ restricts to an inclusion $\Locgeo^{\rm{N}}(X) \subset \Locgeo^{\rm{A}}(X)$.

\begin{lem}\label{lem:LocgeoNori_X'/X}
	Let $k'/k$ be a finite field extension; 
	set $X' := X \times_k k'$, and write $e_X \colon X' \to X$ for the corresponding finite étale morphism.
	Choose a complex embedding $\sigma' \colon k' \hookrightarrow \C$ extending $\sigma$.
	Then, there is a canonical equivalence
	\begin{equation*}
		\Locgeoemb{\sigma}^{\rm{N}}(X) \xrightarrow{\sim} \Locgeoemb{\sigma'}^{\rm{N}}(X').
	\end{equation*}
\end{lem}
\begin{proof}
	By construction,
	the equivalences
	of Lemma~\ref{lem:Perv_X'/X}(1) and Lemma~\ref{lem:MNori_X'/X}(1) fit into the commutative diagram
	\begin{equation*}
		\begin{tikzcd}
			\MNori_{\sigma}(X) \arrow{rr}{\sim} \arrow{d}{\iota_X} && \MNori_{\sigma'}(X') \arrow{d}{\iota_{X'}} \\
			\Pervgeoemb{\sigma}^{\rm{A}}(X) \arrow{rr}{\sim} && \Pervgeoemb{\sigma'}^{\rm{A}}(X').
		\end{tikzcd}
	\end{equation*}
	Restricting to local systems, we get the commutative diagram
	\begin{equation*}
		\begin{tikzcd}
			\NMLoc_{\sigma}(X') \arrow{rr}{\sim} \arrow{d}{\iota_X} && \NMLoc_{\sigma'}(X') \arrow{d}{\iota_{X'}} \\
			\Locgeoemb{\sigma}^{\rm{A}}(X) \arrow{rr}{\sim} && \Locgeoemb{\sigma'}^{\rm{A}}(X'),
		\end{tikzcd}
	\end{equation*}
	where the horizontal arrows are still equivalences (as noted in Lemma~\ref{lem:Locgeo_X'/X}(1) for the lower one, and in Lemma~\ref{lem:NMLoc_X'_X}(1) below for the upper one).
	Taking the essential images of the vertical arrows and stabilizing under subquotients, we get the conclusion.
\end{proof}

Unfortunately, at this stage we are not able to prove the analogue of Lemma~\ref{lem:Locgeo_X'/X}(2), since it is not clear from the definition that the étale prestack $X \mapsto \Locgeoemb{\sigma}^{\rm{N}}(X)$ is in fact a stack.

\begin{proof}[Proof of Proposition~\ref{prop:LocgeoNori_stable_ext}]
	The fundamental case is when $X$ is connected and admits a $k$-rational point (which implies, in particular, that it is geometrically connected over $k$):
	in this case, the result is established in~\cite[Thm.~7.7(1)]{Jacobsen}.
	In order to conclude, we reduce the case of a general (non-empty) $k$-variety $X$ to the above fundamental case.
	
	Up to considering each connected component of $X$ separately, we may assume $X$ connected.
	First, we reduce to the case when $X$ is geometrically connected over $k$.
	To this end, let $k(X)$ denote the function field of $X$, and let $k'$ be the algebraic closure of $k$ inside $k(X)$:
	it is a finite extension of $k$.
	Note that, since $X$ is normal (being smooth over $k$), all elements of $k'$ are regular functions on $X$.
	By construction, the connected components of the smooth $k'$-variety $X' := X \times_k k'$ are geometrically connected over $k'$.
	Applying Lemma~\ref{lem:LocgeoNori_X'/X}, we are allowed to replace the $k'$-variety $X$ by the $k$-variety $X'$.
    Replacing the latter by its single connected components, we may assume $X$ geometrically connected over $k$.
	
	Next, we further reduce to the case when the geometrically connected $k$-variety $X$ admits a $k$-rational point.
	To this end, fix a closed point $x \in X(\bar{k})$, and let $k'/k$ be a finite extension such that $x \in X(k')$.
	Note that the $k'$-variety $X' := X \times_k k'$ is geometrically connected over $k'$ and has a canonical $k'$-rational point $x' \in X'(k')$ lying over $x$.
	Applying Lemma~\ref{lem:LocgeoNori_X'/X}, we are allowed to replace the $k$-variety $X$ by the $k'$-variety $X'$.
	This concludes the reduction.
\end{proof}

Recall that the extension-stability of $\Pervgeo^{\rm{N}}(X)$ inside $\Pervgeo^{\rm{A}}(X)$ is still unclear at this point, and represents in fact the only possible obstruction for the inclusion $\Pervgeo^{\rm{N}}(X) \subset \Pervgeo^{\rm{A}}(X)$ to be an equivalence.
Unfortunately, knowing the extension-stability of $\Locgeo^{\rm{N}}(X)$ inside $\Locgeo^{\rm{A}}(X)$ does not imply formally that they coincide:
the issue is that $\Locgeo^{\rm{N}}(X)$ is not defined as the intersection $\Pervgeo^{\rm{N}}(X) \cap \Loc(X)$.

Until the end of the present subsection, we assume the smooth $k$-variety $X$ to be geometrically connected over $k$, say of dimension $d$, and
we let $a_X \colon X \to \Spec(k)$ denote the structural morphism.
By~\cite[Thm.~6.2]{TerenziNori}, the shifted tensor product
\begin{equation*}
	M_1 \otimes^{\dagger} M_2 := (M_1[-d] \otimes M_2[-d])[d]
\end{equation*}
turns $\NMLoc(X)$ into a neutral Tannakian category over $\Q$ with unit object given by the shifted unit motive ${^\rom{p} \Q_X} := \Q_X[d]$.
By construction, the forgetful functor
\begin{equation*}
	\iota_X \colon \NMLoc(X) \to \Loc(X)
\end{equation*}
is canonically monoidal.
Therefore, for every closed point $x \in X(\bar{k})$, the composite
\begin{equation*}
	\NMLoc(X) \xrightarrow{\iota_X} \Loc(X) \xrightarrow{x^*[-d]} \vect_{\Q}
\end{equation*}
defines a fibre functor for $\NMLoc(X)$.

\begin{nota}\label{nota:pi1geo-Gmot(X)}
	Let $x \in X(\bar{k})$ be a closed point.
	\begin{itemize}
		\item We let $\pi_1^{\rm{N}}(X,x)$ denote the Tannaka dual of $\Locgeo^{\rm{N}}(X)$ with respect to the fibre functor at $x$.
		\item We let $\GmotNo(X,x)$ denote the Tannaka dual of $\NMLoc(X)$ with respect to the fibre functor at $x$, and we call it \textit{Nori's motivic Galois group} of $X$ with base-point $x$.
	\end{itemize}
\end{nota}

\begin{rem}\label{rem:a_X^*-MNori-fullyfaithful}
	The shifted inverse image functor
	\begin{equation*}
		a_X^*[d] \colon \MNori(k) \to \MNori(X)
	\end{equation*}
	takes values in $\NMLoc(X)$ and admits a right adjoint
	\begin{equation*}
		\pH^0 \circ a_{X,*}[-d] \colon \MNori(X) \to \MNori(k).
	\end{equation*}
	As for usual local systems, the functor $a_X^*[d]$ is fully faithful, with essential image stable under subquotients.
	By~\cite[Prop.~2.21(a)]{deligne-milne}, this determines a canonical surjection between Tannaka dual groups
	\begin{equation}\label{eq:Gmot(X)-to-Gmot(k)}
		\GmotNo(X,x) \to \GmotNo(k).
	\end{equation}
    To check the above properties of $a_X^*[d]$, it suffices to translate them in terms of unit and counit morphisms, as in Remark~\ref{rem:constantLoc}:
	since the forgetful functors $\iota_k$ and $\iota_X$ are conservative, the invertibility of the relevant natural transformations in the motivic setting follows from the invertibility of the corresponding natural transformations in the topological setting. 
\end{rem}

Recall that $\Locgeo^{\rm{N}}(X)$ is defined as the full abelian subcategory of $\Loc(X)$ generated by the image of the forgetful functor $\iota_X \colon \NMLoc(X) \to \Loc(X)$ under subquotients.
By the general theory of Tannakian categories, $\Locgeo^{\rm{N}}$ is stable under tensor products and duals inside $\Loc(X)$ (see the proof of~\cite[Lem.~2.13]{Jacobsen}), hence a Tannakian subcategory.
By construction, the homomorphism of Tannaka duals 
\begin{equation}\label{eq:pi^1_geo(X)-to-Gmot(X)}
	\pi_1^{\rm{N}}(X,x) \to \GmotNo(X,x)
\end{equation}
induced by the forgetful functor $\iota_X \colon \NMLoc(X) \to \Locgeo^{\rm{N}}(X)$ is a closed immersion (see~\cite[Prop.~2.21(b)]{deligne-milne}).
This suggests that part of the complexity of the Tannakian category $\NMLoc(X)$ comes from the complexity of $X$ as an algebraic variety;
both the topological complexity of $X^{\sigma}$ and the algebraic complexity of $k$ contribute to the latter.
This can be seen clearly in the following generalization of Example~\ref{ex:ArtinMot_k}:

\begin{ex}\label{ex:ArtinMot_X}
	Let $k'/k$ be a finite Galois extension; 
	set $X' := X \times_k k'$, and write $e_X \colon X' \to X$ for the corresponding finite étale morphism.
	Since $X$ is assumed to be geometrically connected, $e_X$ is a Galois covering and the canonical homomorphism $\Gal(k'/k) \to \Gal(X'/X)$ is bijective.
	Consider the motivic local system
	\begin{equation*}
		{^\rom{p} \Q}_{X'/X} := e_{X,*} {^\rom{p} \Q}_{X'} = e_{X,*} e_X^* {^\rom{p} \Q}_X \in \NMLoc(X),
	\end{equation*}
	which we may call an \textit{Artin motivic local system}.  
	We claim that the Tannakian subcategory $\langle {^\rom{p} \Q}_{X'/X} \rangle^{\otimes} \subset \NMLoc(X)$ is canonically equivalent to $\Rep_{\Q}(\Gal(X'/X))$, with the generator ${^\rom{p} \Q}_{X'/X}$ corresponding to the regular representation.
	For any closed point $x \in X(\bar{k})$, this determines a surjection between Tannaka dual groups
	\begin{equation}\label{eq:Gmot(X)-to-Gal(X'/X)}
		\GmotNo(X,x) \twoheadrightarrow \Gal(X'/X).
	\end{equation}
	To prove the claim, consider the canonical $\Q$-algebra homomorphism
	\begin{equation*}
		\Q[\Gal(X'/X)] \to \End_{\MNori(X)}({^\rom{p} \Q}_{X'/X})
	\end{equation*}
	obtained as in Example~\ref{ex:ArtinMot_k}.
	The key point is to show that this is bijective;
	the claim then follows by repeating the argument of Example~\ref{ex:ArtinMot_k} in the relative setting.
	Note that the above $\Q$-algebra homomorphism fits into the commutative diagram
	\begin{equation*}
		\begin{tikzcd}
			\Q[\Gal(k'/k)] \arrow{rr} \arrow[equal]{d} && \End_{\MNori(k)}(\Q_{k'/k}) \arrow{d} \\
			\Q[\Gal(X'/X)] \arrow{rr} && \End_{\MNori(X)}({^\rom{p} \Q}_{X'/X}),
		\end{tikzcd}
	\end{equation*}
	where the right-most vertical arrow is induced by the shifted inverse image functor $a_X^*[d]$.
	Since the upper horizontal arrow and the right-most vertical arrow are already known to be bijective (by Example~\ref{ex:Artin-motive} and Remark~\ref{rem:a_X^*-MNori-fullyfaithful}, respectively), the same must hold for the lower horizontal arrow.
	This proves the claim.
	\qed
\end{ex}

As in Subsection~\ref{subsect:MNori(k)}, this discussion leads us to study the difference between the motivic Galois groups of $X$ and of $X'$ via classical Galois theory.
Note that $X'$ is geometrically connected over $k'$ but not over $k$.
Hence, in order to define the motivic Galois group of $X'$, we need to regard $X'$ as a $k'$-variety rather than as a $k$-variety.

\begin{nota}
	Following Notation~\ref{nota:MNori_emb}, we write $\NMLoc_{\sigma}(X)$ in place of $\NMLoc(X)$ when we want to stress its dependence on the base field $k$ (and on the chosen complex embedding $\sigma$).
\end{nota}

Fix a complex embedding $\sigma' \colon k' \hookrightarrow \C$ extending $\sigma$.
Lemma~\ref{lem:MNori_X'/X} admits the following variant for motivic local systems:
\begin{lem}\label{lem:NMLoc_X'_X}
	Keep the notation and assumptions of Lemma \ref{lem:MNori_X'/X}, and assume $X$ smooth over $k$.
	Then:
	\begin{enumerate}
		\item The abelian categories $\NMLoc_{\sigma'}(X')$ and $\NMLoc_{\sigma}(X')$ (defined by regarding $X'$ as a $k'$-variety or as a $k$-variety, respectively) are canonically equivalent.
		\item If $k'/k$ is Galois, the inverse image functor $e_X^* \colon \NMLoc_{\sigma}(X) \rightarrow \NMLoc_{\sigma}(X')$ induces a canonical equivalence
		\begin{equation*}
			\NMLoc_{\sigma}(X) \xrightarrow{\sim} \NMLoc_{\sigma}(X')^{\Gal(k'/k)},
		\end{equation*}
		where the action of $\Gal(k'/k)$ on $\NMLoc_{\sigma}(X')$ is induced by its action on $X'$ via $k$-automorphisms.
	\end{enumerate}
\end{lem}
\begin{proof}
	The two statements follow from the corresponding statements of Lemma~\ref{lem:MNori_X'/X}.
\end{proof}

By Lemma~\ref{lem:Bti_k=Bti_k-e^*}, the diagram of monoidal functors
\begin{equation*}
	\begin{tikzcd}
		\DAct(X) \arrow{rr}{e_X^*} \arrow{d}{\Bti_{\sigma,X}^*} && \DAct(X') \arrow{d}{\Bti_{\sigma',X'}^*} \\
		\Dbgeoemb{\sigma}(X) \arrow{rr}{\sim} && \Dbgeoemb{\sigma'}(X')
	\end{tikzcd}
\end{equation*}
commutes up to monoidal natural isomorphism.
Composing with the $0$th perverse cohomology,
passing to the universal abelian factorizations,
and restricting to motivic local systems,
we deduce that the diagram of monoidal functors
\begin{equation*}
	\begin{tikzcd}
		\NMLoc_{\sigma}(X) \arrow{rr}{e_X^*} \arrow{d}{\iota_X} && \NMLoc_{\sigma'}(X') \arrow{d}{\iota_{X'}} \\
		\Locgeoemb{\sigma}(X) \arrow{rr}{\sim} && \Locgeoemb{\sigma'}(X')
	\end{tikzcd}
\end{equation*}
commutes up to monoidal natural isomorphism.
This yields a canonical homomorphism between Tannaka dual groups
\begin{equation}\label{eq:Gmot(X')-to-Gmot(X)}
	\GmotNo(X',x') \to \GmotNo(X,x).
\end{equation}
We have the following generalization of Proposition~\ref{prop:ses_Gmot-Gmot-Gal}:
\begin{prop}\label{prop:ses_Gmot(X')-Gmot(X)-Gal(X'/X)}
	Keep the notation and assumptions of Example~\ref{ex:ArtinMot_X}.
	The sequence of pro-algebraic groups
	\begin{equation*}
		1 \to \GmotNo(X',x') \to \GmotNo(X,x) \to \Gal(X'/X) \to 1
		,
	\end{equation*}
	defined by the
	homomorphisms~\eqref{eq:Gmot(X)-to-Gal(X'/X)}
	and~\eqref{eq:Gmot(X')-to-Gmot(X)},
	is exact.
\end{prop}

\begin{proof}
	It suffices to rewrite the proof of Proposition~\ref{prop:ses_Gmot-Gmot-Gal} with $X$ in place of $\Spec(k)$, with Example~\ref{ex:ArtinMot_X} playing here the same role as Example~\ref{ex:ArtinMot_k} there;
	we leave the details to the interested reader.
\end{proof}

We can finally state the main result of this subsection:

\begin{thm}\label{thm:fund_ses_general}
	Let $X$ be a smooth, geometrically connected, quasi-projective $k$-variety.
	Then, for every closed point $x \in X(\bar{k})$, the sequence of pro-algebraic groups
	\begin{equation*}
		1 \rightarrow \pi_1^{\rm{N}}(X,x) \rightarrow \GmotNo(X,x) \rightarrow \GmotNo(k) \rightarrow 1
	\end{equation*}
	defined by the homomorphisms~\eqref{eq:pi^1_geo(X)-to-Gmot(X)} and~\eqref{eq:Gmot(X)-to-Gmot(k)} is exact.
\end{thm}

\begin{proof}
	We divide the argument into two main steps.
	
	Assume first that the base-point $x \in X(\bar{k})$ is in fact a $k$-rational point $x \in X(k)$.
	In this case, by~\cite[Thm.~7.7(1)]{Jacobsen}, the sequence in the statement is even split-exact;
	the section $\GmotNo(k) \rightarrow \GmotNo(X,x)$ is induced by the section $x \colon \Spec(k) \rightarrow X$ to the structural morphism.
	
	In the general case, fix a finite Galois extension $k'/k$ such that $x \in X(k')$;
	as usual, set $X' := X \times_k k'$, and write $e_X \colon X' \rightarrow X$ for the corresponding finite Galois covering.
	The $k'$-variety $X'$ is still smooth and geometrically connected, and it has a canonical point $x' \in X'(k')$ over $x$. 
	Choose a complex embedding $\sigma' \colon k' \hookrightarrow \C$ extending $\sigma$.
	The monoidal equivalence $\Locgeoemb{\sigma}^{\rm{N}}(X) \xrightarrow{\sim} \Locgeoemb{\sigma'}^{\rm{N}}(X')$ (given by Lemma~\ref{lem:LocgeoNori_X'/X}) translates into an isomorphism of Tannaka duals
	\begin{equation*}
		\pi_1^{\rm{N}}(X',x') \xrightarrow{\sim} \pi_1^{\rm{N}}(X,x).
	\end{equation*}
	All in all, we get a commutative diagram of the form
	\begin{equation*}
		\begin{tikzcd}
			& & 1 \arrow{d} & 1 \arrow{d} \\
			1 \arrow{r} & \pi_1^{\rom{N}}(X',x') \arrow{d}{\sim} \arrow{r} & \GmotNo(X',x') \arrow{d} \arrow{r} & \GmotNo(k') \arrow{d} \arrow{r} & 1 \\
			1 \arrow{r} & \pi_1^{\rom{N}}(X,x) \arrow{r} & \GmotNo(X,x) \arrow{d} \arrow{r} & \GmotNo(k) \arrow{d} \arrow{r}& 1 \\
			& & \Gal(X'/X) \arrow{d} \arrow[equal]{r} & \Gal(k'/k) \arrow{d} \\
			& & 1 & 1,
		\end{tikzcd}
	\end{equation*}
	where both columns are already known to be exact (the right-most one by Proposition~\ref{prop:ses_Gmot-Gmot-Gal}, the middle one by Proposition~\ref{prop:ses_Gmot(X')-Gmot(X)-Gal(X'/X)}).
	Since the upper row is also known to be exact (by the previous step of the proof), the exactness of the middle row follows formally by diagram-chasing. 
	This concludes the proof.
\end{proof}

\section{The fundamental sequence in Voevodsky's setting}
\label{sect:fund_ses_Voe}

In this section, we describe an alternative construction of Tannakian categories of motivic local systems, based on Ayoub's theory of sheaves of geometric origin.
Our guiding principle is that, for any reasonable notion of motivic local system over a smooth $k$-variety, the associated Tannaka dual should fit into a fundamental sequence like the one of Theorem~\ref{thm:fund_ses_general}. 
The fundamental sequence constructed in the present section uses Ayoub's motivic Galois group $\GmotAy(k)$ in place of $\GmotNo(k)$.

After reviewing the construction of Ayoub's group and its main properties, we discuss Ayoub's presentation of the motivic Galois group as a group of symmetries of constructible complexes of geometric origin.
We show that the resulting categories of equivariant local systems behave like Ivorra--Morel's motivic local systems, and in particular, satisfy the analogous fundamental sequence:
this is Theorem~\ref{thm:ayoub_fund_seq} below.
The main constructions and results of this section rely heavily on the general formalism of equivariant Tannakian categories, summarized in Appendix~\ref{sect:App}.

As in the previous section, we fix a complex embedding $\sigma \colon k \hookrightarrow \C$.

\subsection{Ayoub's motivic Galois group}\label{subsect:GmotAyoub}

Ayoub's construction of the motivic Galois group,
worked out in~\cite{AyoHopf1,AyoHopf2,AyoHopf3},
exploits the Betti realization
$\Bti_k^* \colon \DA(k) \to D(\Q)$
and, specifically, the associated Betti algebra
$\mathcal{B}_k := \Bti_{k,*}\Q \in \DA(k)$.

\begin{thm}[{\cite[Thm.~1.45, Defn.~2.9, Cor.~2.105]{AyoHopf1}}]
	\label{thm:AyoHopf}
The complex
\begin{equation*}
	\mathcal{H}_k := \Bti_k^* \mathcal{B}_k = \Bti_k^* \Bti_{k,*} \Q \in D(\Q)
\end{equation*} 
is naturally a commutative Hopf algebra object, and its cohomology is concentrated in non-positive degrees. 
\end{thm}

The object $\mathcal{H}_k$ is called the \textit{motivic Hopf algebra}:
it inherits the algebra structure from the Betti algebra, whereas, to define the co-algebra structure, one also needs to use the canonical monoidal section
\begin{equation*}
	s_k \colon D(\Q) \to \DA(k)
\end{equation*}
sending a complex in $D(\Q)$ to the associated constant motive.
While~\cite{AyoHopf1} defines $\mathcal{H}_k$ as
a Hopf algebra object in the triangulated category $D(\Q)$,
the method of~\cite{AyoHopf3} promotes it canonically to a homotopy Hopf algebra, thereby solving the potential technical issues related to the lack of functoriality of triangulated categories.
The latter description of $\mathcal{H}_k$
can be promoted to the $\infty$-category $D(\Q)$,
as done in~\cite[\S~1.3]{AyoAnab}.

The property that the cohomology of $\mathcal{H}_k$ vanishes in positive degrees is a highly non-formal result.
It implies the following:
\begin{cor}
	The $0$-th cohomology object
	\begin{equation*}
		\mathcal{H}_k^0 := H^0(\mathcal{H}_k)
	\end{equation*}
	is naturally a classical commutative Hopf algebra over $\Q$.
\end{cor} 

By~\cite[Thm.~8.3]{AyoHopf3}, the Betti realization over $k$ factors canonically as a composite of triangulated functors of the form
\begin{equation*}
	\Bti_k^* \colon \DA(k) \to \mathsf{coMod}_{\mathcal{H}_k}(D(\Q)) \xrightarrow{\for_{\mathcal{H}_k}} D(\Q),
\end{equation*}
where the triangulated structure of $\mathsf{coMod}_{\mathcal{H}_k}(D(\Q))$ is induced by the higher structure of the homotopy Hopf algebra $\mathcal{H}_k$. 
By~\cite[Prop.~1.55]{AyoHopf1}, the following universal property holds:
given a commutative bialgebra object $\mathcal{K} \in D(\Q)$ and a monoidal functor
\begin{equation*}
	\DA(k) \rightarrow \mathsf{coMod}_{\mathcal{K}}(D(\Q))
\end{equation*} 
making both halves of the diagram
\begin{equation*}
\begin{tikzcd}
D(\Q) \arrow{d}[']{s_k} \arrow{drr}
	\\
\DA(k) \arrow{rr} \arrow{drr}[']{\Bti_k^*}
	&& \mathsf{coMod}_{\mathcal{K}}(D(\Q)) \arrow{d}{\rm{for}_{\mathcal{K}}}
	\\
&& D(\Q)
\end{tikzcd}
\end{equation*} 
commute up to monoidal natural isomorphism, there exists a unique bialgebra morphism $\varphi \colon \mathcal{H}_k \rightarrow \mathcal{K}$ in $D(\Q)$ making both halves of the diagram
\begin{equation*}
\begin{tikzcd}
\DA(k) \arrow{rr} \arrow{d}
	&& \mathsf{coMod}_{\mathcal{K}}(D(\Q)) \arrow{d}{\rm{for}_{\mathcal{K}}}
	\\
\mathsf{coMod}_{\mathcal{H}_k}(D(\Q))
	\arrow{rr}[']{\rm{for}_{\mathcal{H}_k}}
	\arrow{urr}{\varphi_*}
	&& D(\Q)
\end{tikzcd}
\end{equation*}
commute up to the induced monoidal natural isomorphism.

\begin{nota}~
	\begin{itemize}
		\item We let $\Gmot(k)$ denote the spectrum of the derived Hopf algebra $\mathcal{H}_k$ in the sense of spectral algebraic geometry.
		\item We let $\GmotAy(k)$ denote the spectrum of the classical Hopf algebra $\mathcal{H}_k^0$, and we call it \textit{Ayoub's motivic Galois group} of $k$;
		we write it as $\GmotAy(k,\sigma)$ if we want to stress its dependence on $\sigma$.
	\end{itemize}
\end{nota}

The morphism of derived Hopf algebras $\mathcal{H}_k \to \mathcal{H}_k^0$ induces a canonical homomorphism of spectral groups
\begin{equation*}
	\GmotAy(k) \to \Gmot(k)
\end{equation*}
with the property that the homomorphism on $\Lambda$-valued points
\begin{equation*}
	\GmotAy(k)(\Lambda) \to \Gmot(k)(\Lambda)
\end{equation*}
is bijective whenever $\Lambda$ is a classical $\Q$-algebra.

\begin{rem}\label{rem:Gmot(k)_expected_classical}
In fact, the spectral group $\Gmot(k)$ is expected to be already classical,
hence equal to $\GmotAy(k)$:
in other words,
the complex $\mathcal{H}_k \in D(\Q)$
is expected to be already concentrated in degree $0$.
This would follow from
the existence of the motivic $t$-structure on $\DAct(k)$
(see~\cite[Prop.~3.2.9]{AyoConj}).
\end{rem}

Even if the spectral group is the most natural object
to work with in the setting of Voevodsky motives,
in this paper we only use its classical counterpart.
We need to discuss some basic properties of $\GmotAy(k)$,
some of which are only stated for $\Gmot(k)$ in the literature:
in particular, we need to study the behaviour of $\GmotAy(k)$ under finite extensions of the base field $k$, as done for Nori's motivic Galois group in Section~\ref{subsect:MNori(k)}.

So let $k'/k$ be a finite extension, and write $e \colon \Spec(k') \to \Spec(k)$ for the corresponding finite étale morphism.
Choose a complex embedding $\sigma' \colon k' \hookrightarrow \C$ extending $\sigma$.
Using Lemma~\ref{lem:Bti_k=Bti_k-e^*}, we obtain a canonical morphism of motivic Hopf algebras
\begin{equation*}
	\mathcal{H}_{\sigma,k} := \Bti_{\sigma,k}^* \Bti_{\sigma,k,*} \Q = \Bti_{\sigma',k'}^* e^* e_* \Bti_{\sigma',k',*} \Q \xrightarrow{\epsilon} \Bti_{\sigma',k'}^* \Bti_{\sigma',k',*} \Q =: \mathcal{H}_{\sigma',k'}.
\end{equation*}
Taking the $0$-th cohomology objects and passing to the associated spectra, this defines a homomorphism
\begin{equation}\label{eq:Gmot'-to-Gmot_Ayoub}
	\GmotAy(k') \to \GmotAy(k).
\end{equation}
Moreover, we have a canonical commutative algebra morphism in $\DA(k)$
\begin{equation*}
	\Q_{k'/k} := e_* e^* \Q_k
	\xrightarrow{\eta} e_* \Bti_{\sigma',k',*} \Bti_{\sigma',k'}^* e^* \Q_k
	= \Bti_{\sigma,k,*} \Bti_{\sigma,k}^* \Q_k
	= \mathcal{B}_k.
\end{equation*}
Note that the complex $\Bti_{\sigma,k}^* \Q_{k'/k} \in D(\Q)$ is concentrated in degree $0$, where it coincides with the cohomology group $H^0(\Spec(k')^{\sigma};\Q)$.
Recall that the latter can be identified with $\Q^{\Hom_{\sigma}(k',\C)}$, where $\Hom_{\sigma}(k',\C)$ denotes the set of complex embeddings of $k'$ extending $\sigma$.
	
Suppose now that the extension $k'/k$ is Galois. 
The choice of an element in $\Hom_{\sigma}(k',\C)$ defines a bijection between $\Gal(k'/k)$ and $\Hom_{\sigma}(k',\C)$, hence induces an isomorphism between $\Q^{\Hom_{\sigma}(k',\C)}$ and the dual group algebra $\Q[\Gal(k'/k)]^{\lor} = \Q^{\Gal(k'/k)}$.
If one chooses the same $\sigma'$ as above, the induced morphism of commutative algebras in $D(\Q)$ 
\begin{equation*}
	\Q[\Gal(k'/k)]^{\lor} = \Bti_{\sigma,k}^*(\Q_{k'/k}) \to \Bti_{\sigma,k}^*(\mathcal{B}_k) =: \mathcal{H}_k
\end{equation*}
is in fact a morphism of commutative Hopf algebras.
Taking the $0$-th cohomology objects and passing to the associated spectra, we get a homomorphism
\begin{equation}\label{eq:Gmot-to-Gal_Ayoub}
	\GmotAy(k) \to \Gal(k'/k).
\end{equation}
We have the following analogue of Proposition~\ref{prop:ses_Gmot-Gmot-Gal}:

\begin{prop}[{\cite[Lem.~1.4.9]{AyoAnab}}]\label{prop:ses_Gmot'-Gmot-Gal_Ayoub}
	If $k'/k$ is a Galois extension, the sequence of pro-algebraic groups
	\begin{equation*}
		1 \to \GmotAy(k') \to \GmotAy(k) \to \Gal(k'/k) \to 1
	\end{equation*}
	defined by the homomorphisms \eqref{eq:Gmot'-to-Gmot_Ayoub} and \eqref{eq:Gmot-to-Gal_Ayoub} is exact.
\end{prop}

\begin{proof}
	The argument of~\cite[Lem.~1.4.9]{AyoAnab} establishes the exactness of the analogous sequence of spectral motivic Galois groups.
	This immediately implies the version stated here, because the spectral group $\Gal(k'/k)$ is already classical.
\end{proof}

The similarity with the behaviour of Nori's motivic Galois group is by no means coincidental.
Indeed, an important result by Choudhury--Gallauer asserts that the pro-algebraic groups $\GmotAy(k)$ and $\GmotNo(k)$ are canonically isomorphic (see Theorem~\ref{thm:ChoudGal} below).
This result, which plays a crucial role in Subsection~\ref{subsect:comp-point}, is implicitly used in Example~\ref{ex:Gmot-action_unit_X'} and in Example~\ref{ex:Artin/X_Ayoub} below in order to apply results from Section~\ref{sect:IM_cats} in Ayoub's setting.
This operation is legitimate in view of Lemma~\ref{lem:equiv-point} below, which implies that the $\GmotAy(k)$-action on the underlying $\Q$-vector space of a Nori motive matches with the $\GmotNo(k)$-action defined by Tannakian formalism.
As the reader can check, the results of Subsection~\ref{subsect:comp-point} (in particular, Lemma~\ref{lem:equiv-point}) are logically independent of the new results collected in the present section, so there are no circular arguments.

\subsection{The motivic Galois action on local systems}
\label{subsect:Locgeo^Gmot}

The stable $\infty$-categories of sheaves of geometric origin, introduced in Sections~\ref{subsect:Dbgeo} and~\ref{subsect:Pervgeo}, are closely related to the theory of Voevodsky motives.
The following result describes Ayoub's motivic Galois group as the group of symmetries of these categories:
\begin{thm}[{\cite[Thm.~2.2.3, Cor.~2.2.7]{AyoAnab}}]
	\label{thm:Anabel}
	The following hold:
	\begin{enumerate}
		\item The spectral motivic Galois group $\Gmot(k)$ can be naturally identified with the group of autoequivalences of the six functor formalism $X \mapsto D_{\geo}(X)$.
		\item Under the previous identification, the classical motivic Galois group $\GmotAy(k)$ can be identified with the subgroup of exact autoequivalences:
		for every connective commutative $\Q$-algebra $\Lambda$, an element $g \in \Gmot(k)(\Lambda)$ belongs to $\GmotAy(k)(\Lambda)$ if and only if, for every $k$-variety $X$, the triangulated functor
		\begin{equation*}
			g \cdot - \colon D_{\geo}(X)_{\Lambda} \rightarrow D_{\geo}(X)_{\Lambda}
		\end{equation*}
		is $t$-exact with respect to the ordinary $t$-structure.
	\end{enumerate}
\end{thm}

Here, $\Dgeo(X)$ denotes the stable $\infty$-category of
unbounded complexes of geometric origin
(see Rmk.~\ref{rem:Dgeo}).
In the present paper, we are only interested in the motivic Galois action on $\Dbgeo(X)$ and subcategories thereof.

The construction of the $\Gmot(k)$-action on sheaves of geometric origin is based on Proposition~\ref{prop:Dbgeo=Mod_DA}.
Indeed, by~\cite[Thm.~1.3.21]{AyoAnab} one can regard $\Gmot(k)$ as the spectral group parameterizing the self-equivalences of the Betti algebra $\mathcal{B}_k$, as defined in~\cite[Defn.~1.3.13]{AyoAnab}.
A self-equivalence of $\mathcal{B}_k$ defines a self-equivalence of the Betti algebra $\mathcal{B}_X$ (via Lemma~\ref{lem:BtiAlg_defn}(2)), and the latter determines an autoequivalence of the stable $\infty$-category $\DA(X;\mathcal{B}_X)$ by transport of structure.
By construction, as $X$ varies, these assemble into an autoequivalence of the six functor formalism $X \mapsto \Dgeo(X)$.
In order to deduce the first point of Theorem~\ref{thm:Anabel}, one has to show that this construction identifies autoequivalences of the six functor formalism $X \mapsto \DA(X;\mathcal{B}_X)$ with self-equivalences of $\mathcal{B}_k$:
this is essentially a consequence of Drew--Gallauer's theorem (see Theorem~\ref{thm:Drew-Gallauer} below).

In principle, Theorem \ref{thm:Anabel} gives information about the $\Lambda$-valued points of $\Gmot(k)$ whenever $\Lambda$ is a connective commutative $\Q$-algebra.
Since we are only interested in the action of the classical group $\GmotAy(k)$, we may restrict our attention to classical $\Q$-algebras.
Let $\CAlg_{\Q}$ denote the category of classical commutative $\Q$-algebras.

In this section we only need to consider the $\GmotAy(k)$-action on local systems of geometric origin, regarded as perverse sheaves.
We start from the following observation: 
\begin{lem}
	For every smooth $k$-variety $X$, the full subcategory $\Locgeo^{\rm{A}}(X) \subset \Dgeo(X)$ is $\GmotAy(k)$-stable.
\end{lem}
\begin{proof}
	Without loss of generality, we may assume $X$ connected, say of dimension $d$.
	In this situation, for every $\Lambda \in \CAlg_{\Q}$, the category $\Locgeo^{\rm{A}}(X)_{\Lambda}$ can be identified with the full subcategory of $\Dgeo(X)_{\Lambda}$ spanned by the dualizable objects concentrated in degree $d$.
\end{proof}

\begin{rem}\label{rem:Anabel-Loc}
	In fact, Ayoub's Theorem~\ref{thm:Anabel} admits a version tailored to local systems:
	this is~\cite[Thm.~4.4.2, Cor.~4.4.15]{AyoAnab}, which identifies $\Gmot(k)$ with the autoequivalence group of the monoidal fibered category $X \mapsto \Locgeo^{\rm{A}}(X)$.
	For us, it is more convenient to use the full version from the beginning, since it allows us to use functoriality with respect to direct image functors without further comments.
\end{rem}

We want to study the $\GmotAy(k)$-action on local systems of geometric origin.
Let $X$ be a smooth, geometrically connected $k$-variety of dimension $d$, with structural morphism $a_X \colon X \to \Spec(k)$.

\begin{nota}
	We let $\Locgeo^{\rm{A}}(X)^{\GmotAy(k)}$ denote the category of $\Q$-algebraic equivariant objects in the sense of Definition~\ref{defn:htpy-fixed_linear},
	and we write
	\begin{equation}\label{eq:forget_AMLoc}
		\omega_X \colon \Locgeo^{\rm{A}}(X)^{\GmotAy(k)} \to \Locgeo^{\rm{A}}(X)
	\end{equation}
	for the associated forgetful functor $\for_{\GmotAy(k)}$ (see Construction~\ref{constr:res_htpy-fixed}).
\end{nota}

\begin{lem}
	The category $\Locgeo^{\rm{A}}(X)^{\GmotAy(k)}$ is naturally neutral Tannakian over $\Q$, and the forgetful functor~\eqref{eq:forget_AMLoc} is an exact tensor-functor.
\end{lem}

\begin{proof}
	This follows from Lemma~\ref{lem:htpy-fixed_monoidal} and Proposition~\ref{prop:htpy-fixed_Tannakian}.
\end{proof}

\begin{ex}\label{ex:trivializ-Gmot}
If $X = \Spec(k)$, the motivic Galois action is quite easy to understand.
As in Example~\ref{ex:Locgeo(k)=vect},
let us identify $\Locgeo^{\rm{A}}(\Spec(k))$ with $\vect_{\Q}$.
Any monoidal action of a pro-algebraic group on $\vect_{\Q}$
admits a trivialization (in the sense of Definition~\ref{defn:isom_actions-trivializ}), because it preserves the unit object and induces the identity map on its endomorphism ring.
Passing to the equivariant objects, we obtain a canonical equivalence of neutral Tannakian categories
\begin{align*}
\Locgeo^{\rm{A}}(\Spec(k))^{\GmotAy(k)} &= \vect_{\Q}^{\GmotAy(k)} \\
&= \Rep_{\Q}(\GmotAy(k)). && \textup{(by Example~\ref{ex:Tann_as_htpyfixed})}
\end{align*}
In particular, the Tannaka dual of $\Locgeo^{\rm{A}}(\Spec(k))^{\GmotAy(k)}$
is again $\GmotAy(k)$.
\qed
\end{ex}

This example justifies the following notation:

\begin{nota}
	Let $x \in X(\bar{k})$ be a closed point.
	\begin{itemize}
		\item We let $\pi_1^{\rm{A}}(X,x)$ denote the Tannaka dual of $\Locgeo^{\rm{A}}(X)$ with respect to the fibre functor at $x$.
		\item We let $\GmotAy(X,x)$ denote the Tannaka dual of $\Locgeo^{\rm{A}}(X)^{\GmotAy(k)}$ with respect to the fibre functor at $x$, and we call it \textit{Ayoub's motivic Galois group} of $X$ with base-point $x$. 
	\end{itemize}
\end{nota}

Under Tannaka duality, the forgetful functor~\eqref{eq:forget_AMLoc} corresponds to a homomorphism of pro-algebraic groups
\begin{equation}\label{eq:pigeo-to-Gmot(X)_Ayoub}
	\pi_1^{\rm{A}}(X,x) \rightarrow \GmotAy(X,x).
\end{equation}

\begin{rem}
	As a consequence of Proposition~\ref{prop:Dgeo^Gmot-6ff} below, the shifted inverse image functor
	\begin{equation*}
		a_X^*[d] \colon \Locgeo^{\rm{A}}(\Spec(k))^{\GmotAy(k)} \to \Locgeo^{\rm{A}}(X)^{\GmotAy(k)}
	\end{equation*}
	admits a right adjoint, defined once again by the formula
	\begin{equation*}
		\pH^0 \circ a_{X,*}[-d] \colon \Locgeo^{\rm{A}}(X)^{\GmotAy(k)} \to \Locgeo^{\rm{A}}(\Spec(k))^{\GmotAy(k)}.
	\end{equation*}
	Using the same argument as in Remark~\ref{rem:constantLoc}, we deduce that $a_X^*[d]$ is fully faithful, with essential image stable under subquotients.
	This yields a canonical surjection between Tannaka dual groups
	\begin{equation}\label{eq:Gmot(X)-Gmot(k)_Ayoub}
		\GmotAy(X,x) \twoheadrightarrow \GmotAy(k),
	\end{equation}
	where the right-hand side was computed in Example~\ref{ex:trivializ-Gmot}.
\end{rem}

By Theorem~\ref{thm:Anabel}, the motivic Galois action on constructible complexes is faithful;
in view of Remark~\ref{rem:Anabel-Loc}, this happens already on local systems.
One can thus expect that it is not straightforward to describe the $\GmotAy(k)$-orbit of a general object $L \in \Locgeo^{\rm{A}}(X)$.
The simplest non-trivial case is the following:

\begin{ex}\label{ex:Gmot-action_unit_X'}
	As in Example~\ref{ex:ArtinMot_X}, fix a finite Galois extension $k'/k$, set $X' := X \times_k k'$, and write $e_X \colon X' \to X$ for the corresponding finite étale covering.
	We regard $X'$ as a $k$-variety, and we want to analyse the $\GmotAy(k)$-action on the unit local system
	\begin{equation*}
		{^\rom{p} \Q}_{\sigma,X'} := \Q_{\sigma,X'}[d] \in \Locgeoemb{\sigma}^{\rm{A}}(X').
	\end{equation*}
	Since the $\GmotAy(k)$-action on sheaves of geometric origin is monoidal, it necessarily preserves ${^\rom{p} \Q}_{\sigma,X'}$:
	more precisely, the unit local system canonically underlies an object of $\Locgeoemb{\sigma}^{\rm{A}}(X)^{\GmotAy(k)}$ (see Corollary~\ref{cor:Bti-enriched} below).
	Under the monoidal equivalence
	\begin{equation*}
		\Locgeoemb{\sigma}^{\rm{A}}(X') = \prod_{\sigma' \in \Hom_{\sigma}(k',\C)} \Locgeoemb{\sigma'}^{\rm{A}}(X')
	\end{equation*}
	induced by Lemma~\ref{lem:Dbgeo_sigma(X')-decomp},
	the object ${^\rom p \Q}_{\sigma,X'}$ corresponds to the unit tuple
	$\oplus_{\sigma' \in \Hom_{\sigma}(k',\C)} {^\rom p \Q}_{\sigma',X'}$.
	
	Even though the $\GmotAy(k)$-action preserves ${^\rom p \Q}_{\sigma,X'}$,
	it does not preserve its direct summands individually,
	as we now explain.
	For every $\sigma' \in \Hom_{\sigma}(k',\C)$,
	the $\GmotAy(k',\sigma')$-action induced by
	the homomorphism~\eqref{eq:Gmot'-to-Gmot_Ayoub}
	necessarily preserves the corresponding summand ${^\rom{p} \Q}_{\sigma',X'}$.
	The image of $\GmotAy(k',\sigma')$ in $\GmotAy(k)$ coincides with the kernel of the projection $\GmotAy(k) \twoheadrightarrow \Gal(k'/k)$ (by Proposition~\ref{prop:ses_Gmot'-Gmot-Gal_Ayoub}) and,
	in particular, is independent of the choice of $\sigma'$.
	It follows that the $\GmotAy(k)$-action on ${^\rom{p} \Q}_{\sigma,X'}$ factors through a $\Gal(k'/k)$-action permuting its direct summands.
	
	We want to describe this permutation action explicitly.
	Recall from Example~\ref{ex:ArtinMot_k} that the identification of $\Gal(k'/k)$ with a quotient of $\GmotAy(k)$ depends on the choice of an element $\sigma'_0 \in \Hom_{\sigma}(k',\C)$.
	From now on, fix such a choice. 
	We rename the various unit summands by setting 
	\begin{equation*}
		{^\rom p \Q}_{\sigma',X'} := s {^\rom p \Q}_{\sigma'_0,X'}
	\end{equation*}
	where $s \in \Gal(k'/k)$ is the unique element
	satisfying $\sigma' = \sigma'_0 \circ s$.
	We claim that,
	with this notation,
	the $\Gal(k'/k)$-action on ${^\rom{p} \Q}_{\sigma,X'}$ satisfies the formula
	\begin{equation*}
		r \cdot (s {^\rom p \Q}_{\sigma'_0,X'}) = (rs) {^\rom p \Q}_{\sigma'_0,X'}
	\end{equation*}
	for all $r, s \in \Gal(k'/k)$.
	Using the compatibility of the $\GmotAy(k)$-action with inverse images, it suffices to establish the claim when $X = \Spec(k)$.
	In this case, using the compatibility of the $\GmotAy(k)$-action with direct images, the claim reduces
	to the computations of Example~\ref{ex:ArtinMot_k} (see also Remark~\ref{rem:ArtinMot_two-Gal-actions}).
	Here, we are implicitly invoking Theorem~\ref{thm:ChoudGal} (and Lemma~\ref{lem:equiv-point}) in order to use the input from Nori's theory.
	\qed
\end{ex}

\subsection{Proof of the fundamental sequence}\label{subsect:fundses_Ayoub}

As before, let $X$ be
a smooth, geometrically connected $k$-variety of dimension $d$
and with structural morphism $a_X \colon X \to \Spec(k)$.
The goal is to establish the analogue of Theorem~\ref{thm:fund_ses_general} for Ayoub's motivic Galois group $\GmotAy(X,x)$.
Before stating the theorem, we discuss some preliminary results, along the lines of Section~\ref{subsect:fundses_MNori}.

The easiest way to produce an object of $\Locgeo^{\rm{A}}(X)^{\GmotAy(k)}$ is by considering the Betti realization of a dualizable motivic sheaf
(this is justified by Corollary~\ref{cor:Bti-enriched} below).

\begin{ex}\label{ex:Artin/X_Ayoub}
	Let us consider again the setting of Example~\ref{ex:ArtinMot_X}:
	fix a finite Galois extension $k'/k$, set $X' \times_k k'$, and write $e_X \colon X' \to X$ for the corresponding finite étale covering.
	The local system
	\begin{equation*}
		{^\rom{p} \Q}_{X'/X} := \Bti_X^* (e_{X,*} {\Q}_{X'}[d]) \in \Locgeo^{\rm{A}}(X)
	\end{equation*}
	canonically underlies an equivariant object ${^\rom{p} \Q}_{X'/X} \in \Locgeo^{\rm{A}}(X)^{\GmotAy(k)}$.
	As in Example~\ref{ex:ArtinMot_X}, the Tannakian subcategory $\langle {^\rom{p} \Q}_{X'/X} \rangle^{\otimes} \subset \Locgeo^{\rm{A}}(X)^{\GmotAy(k)}$ is equivalent to $\Rep_{\Q}(\Gal(X'/X))$, with the generator ${^\rom{p} \Q}_{X'/X}$ corresponding to the regular representation.
	This determines a surjection between Tannaka dual groups
	\begin{equation}\label{eq:Gmot(X)-to-Gal(X'/X)_Ayoub}
		\GmotAy(X,x) \twoheadrightarrow \Gal(X'/X).
	\end{equation}
	To see this, note that the $\Gal(X'/X)$-action on
	the local system $\omega_X({^\rom{p} \Q}_{X'/X})$
	can be expressed in terms of the six operations,
	as in Example~\ref{ex:Artin-motive}.
	Since the $\GmotAy(k)$-action of Theorem~\ref{thm:Anabel} commutes with the six operations, the $\Gal(X'/X)$-action lifts to the equivariant level.
	Hence, we obtain a canonical $\Q$-algebra homomorphism
	\begin{equation*}
		\Q[\Gal(X'/X)] \to \End_{\Locgeo^{\rm{A}}(X)^{\GmotAy(k)}}({^\rom{p} \Q}_{X'/X}),
	\end{equation*}
	and, as in Example~\ref{ex:ArtinMot_X}, the key point is to show that this is an isomorphism.
	This can be achieved by repeating the argument of Example~\ref{ex:ArtinMot_X} in the setting of $\GmotAy(k)$-equivariant local systems.
	Here again, we are implicitly invoking Theorem~\ref{thm:ChoudGal} (and Lemma~\ref{lem:equiv-point}) to use the input from Nori's theory.
	\qed
\end{ex}

Now choose a complex embedding $\sigma' \colon k' \hookrightarrow \C$ extending $\sigma$, and let $\GmotAy(k')$ denote the associated motivic Galois group.
Recall that, by Proposition~\ref{prop:ses_Gmot'-Gmot-Gal_Ayoub}, we have a short exact sequence
\begin{equation*}
	1 \to \GmotAy(k') \to \GmotAy(k) \to \Gal(k'/k) \to 1.
\end{equation*}
Note that the canonical equivalence $\Locgeoemb{\sigma}^{\rm{A}}(X) \xrightarrow{\sim} \Locgeoemb{\sigma'}^{\rm{A}}(X')$ of Lemma~\ref{lem:Locgeo_X'/X}(1) is $\GmotAy(k')$-equi\-vari\-ant:
this is because, under the identifications of Corollary~\ref{cor:Dbgeo=DAct(Btialg)}, it comes from the functor
\begin{equation*}
	\DAct(X;\mathcal{B}_{\sigma,X}) \to \DAct(X';\mathcal{B}_{\sigma',X'})
\end{equation*}
induced by the isomorphism of Proposition~\ref{prop:e_*BtiAlg'}.
This leads to the following analogue of Lemma~\ref{lem:NMLoc_X'_X}:
\begin{lem}\label{lem::AMLoc_X'/X}
	Let $k'/k$ be a finite Galois extension.
	Given a smooth $k$-variety $X$, set $X' := X \times_k k'$, and write $e_X \colon X' \to X$ for the corresponding finite Galois covering.
	Then:
	\begin{enumerate}
		\item The Tannakian categories $\Locgeoemb{\sigma'}^{\rm{A}}(X')^{\GmotAy(k')}$ and $\Locgeoemb{\sigma}^{\rm{A}}(X')^{\GmotAy(k)}$ (defined by regarding $X'$ as a $k'$-variety or as a $k$-variety, respectively) are canonically equivalent.
		\item The inverse image functor $e_X^* \colon \Locgeoemb{\sigma}^{\rm{A}}(X)^{\GmotAy(k)} \rightarrow \Locgeoemb{\sigma}^{\rm{A}}(X')^{\GmotAy(k)}$ induces a canonical equivalence
		\begin{equation*}
			\Locgeoemb{\sigma}^{\rm{A}}(X)^{\GmotAy(k)} \xrightarrow{\sim} (\Locgeoemb{\sigma}^{\rm{A}}(X')^{\GmotAy(k)})^{\Gal(k'/k)},
		\end{equation*}
		where the $\Gal(k'/k)$-action on $\Locgeoemb{\sigma}^{\rm{A}}(X')^{\GmotAy(k)}$ is induced from its action on $X'$ via $k$-auto\-morph\-isms.
	\end{enumerate}
\end{lem}
\begin{proof}
	As usual, the second part follows from the fact that the fibered category $X \mapsto \Locgeoemb{\sigma}^{\rm{A}}(X)^{\GmotAy(k)}$ is a stack for the étale topology:
	same argument as for Lemma~\ref{lem:Locgeo_X'/X}(2).
	
	Let us focus on the first statement.
	For the duration of the proof, we rename the chosen complex embedding $\sigma' \colon k' \hookrightarrow \C$ as $\sigma'_0$ and let the symbol $\sigma'$ denote a generic element of $\Hom_{\sigma}(k',\C)$.
	Consider the direct product decomposition
	\begin{equation*}
		\Locgeoemb{\sigma}^{\rm{A}}(X') = \prod_{\sigma' \in \Hom_{\sigma}(k',\C)} \Locgeoemb{\sigma'}^{\rm{A}}(X')
	\end{equation*}
	induced by Lemma~\ref{lem:Dbgeo_sigma(X')-decomp}.
	By construction, the projection $\Locgeoemb{\sigma}^{\rm{A}}(X') \to \Locgeoemb{\sigma'_0}^{\rm{A}}(X')$ is $\GmotAy(k')$-equivariant.
	Passing to the equivariant objects, we get a canonical monoidal functor
	\begin{equation*}
		\Locgeoemb{\sigma}^{\rm{A}}(X')^{\GmotAy(k)} \to \Locgeoemb{\sigma}^{\rm{A}}(X')^{\GmotAy(k')} \to \Locgeoemb{\sigma'_0}^{\rm{A}}(X')^{\GmotAy(k')},
	\end{equation*}
	where the first passage is restriction under the homomorphism $\GmotAy(k') \to \GmotAy(k)$.
	We want to show that the composite functor is an equivalence.
	By faithfully flat descent, it suffices to show this after extending coefficients from $\Q$ to $\bar{\Q}$.
	We want to apply Proposition~\ref{prop:C^Q=(C^Q')^Q''-proalg} to the short exact sequence of Proposition~\ref{prop:ses_Gmot'-Gmot-Gal_Ayoub}.
	To this end, we claim that the $\GmotAy(k)$-action on $\Locgeoemb{\sigma}^{\rm{A}}(X')$ is induced from the $\GmotAy(k')$-action on $\Locgeoemb{\sigma'}^{\rm{A}}(X')$ in the sense of Example~\ref{ex:ind-rep}.
	The advantage of working with $\bar{\Q}$-coefficients is that it ensures the existence of a scheme-theoretic section to the projection $\GmotAy(k) \twoheadrightarrow \Gal(k'/k)$, which is needed to make the discussion of Example~\ref{ex:ind-rep} go through (see Remark~\ref{rem:ind-rep_proalg}).
	
	Let us prove the claim.
	Following Example~\ref{ex:Gmot-action_unit_X'}, we rename the various factors by setting
	\begin{equation*}
		\Locgeoemb{\sigma'}^{\rm{A}}(X') =: s\Locgeoemb{\sigma'_0}^{\rm{A}}(X'),
	\end{equation*}
	where $s \in \Gal(k'/k)$ is the unique element satisfying $\sigma' = \sigma'_0 \circ s$.
	By Remark~\ref{rem:ind-rep}, proving the claim reduces to checking that, given an element $g \in \GmotAy(k)$ mapping to $r \in \Gal(k'/k)$, the functor $g \cdot - \colon \Locgeoemb{\sigma}^{\rm{A}}(X') \to \Locgeoemb{\sigma}^{\rm{A}}(X')$ sends $s\Locgeoemb{\sigma'_0}^{\rm{A}}(X') \mapsto (rs)\Locgeoemb{\sigma'_0}^{\rm{A}}(X')$ for all $s \in \Gal(k'/k)$.
	Recall that the subcategory $\Locgeoemb{\sigma'}^{\rm{A}}(X') \subset \Locgeoemb{\sigma}^{\rm{A}}(X')$ can be characterized as the image of the endofunctor $\Q_{\sigma',X'} \otimes -$ (see Remark~\ref{rem:unit-decomp_X'}).
	Since the $\GmotAy(k)$-action is compatible with the tensor product, proving the claim reduces to showing that, given an element $g \in \GmotAy(k)(\bar{\Q})$ mapping to $r \in \Gal(k'/k)$, the formula
	\begin{equation*}
		g \cdot s\Q_{\sigma'_0,X'} = (rs)\Q_{\sigma'_0,X'}
	\end{equation*}
	holds for each $s \in \Gal(k'/k)$.
	But this has been already proved in Example~\ref{ex:Gmot-action_unit_X'}.
\end{proof}

Using Lemma~\ref{lem::AMLoc_X'/X}(1), we obtain a monoidal functor
\begin{equation*}
	\Locgeoemb{\sigma}^{\rm{A}}(X)^{\GmotAy(k)} \xrightarrow{e_X^*} \Locgeoemb{\sigma}^{\rm{A}}(X')^{\GmotAy(k)} = \Locgeoemb{\sigma'}^{\rm{A}}(X')^{\GmotAy(k')}
\end{equation*}
compatible with the fibre functors at any closed point $x \in X(\bar{k})$.
This induces a canonical homomorphism between Tannaka dual groups
\begin{equation}\label{eq:Gmot(X')-to-Gmot(X)_Ayoub}
	\GmotAy(X',x') \to \GmotAy(X,x),
\end{equation}
where $x' \in (X')^{\sigma'}$ denotes the point corresponding to $x \in X^{\sigma}$. 
We have the following generalization of Proposition~\ref{prop:ses_Gmot'-Gmot-Gal_Ayoub}:
\begin{prop}\label{prop:ses_Gmot(X')-Gmot(X)-Gal(X'/X)_Ayoub}
	Keep the notation and assumptions of Example~\ref{ex:Artin/X_Ayoub}.
	The sequence of pro-algebraic groups
	\begin{equation*}
		1 \to \GmotAy(X',x') \to \GmotAy(X,x) \to \Gal(X'/X) \to 1
	\end{equation*}
	defined by the homomorphisms~\eqref{eq:Gmot(X)-to-Gal(X'/X)_Ayoub} and~\eqref{eq:Gmot(X')-to-Gmot(X)_Ayoub} is exact.
\end{prop}
\begin{proof}
	Same argument as for Proposition~\ref{prop:ses_Gmot(X')-Gmot(X)-Gal(X'/X)}, using Example~\ref{ex:Artin/X_Ayoub} in place of Example~\ref{ex:ArtinMot_X}.
\end{proof}

We can finally state the main result of this section:
\begin{thm}\label{thm:ayoub_fund_seq}
	Let $X$ be a smooth, geometrically connected $k$-variety.
	Then, for every point $x \in X(\bar{k})$, the sequence of pro-algebraic groups
	\begin{equation}\label{ses:Loc-Ayoub}
		1 \to \pi_1^{\rm{A}}(X,x)
		\to \GmotAy(X,x)
		\to \GmotAy(k)
		\to 1
	\end{equation}
	defined by the homomorphisms \eqref{eq:pigeo-to-Gmot(X)_Ayoub} and \eqref{eq:Gmot(X)-Gmot(k)_Ayoub} is exact.
\end{thm}
\begin{proof}
	It suffices to treat the case when the base-point $x$ is a $k$-rational point:
	once the thesis is proved in this case, the general case follows from the same argument used in the proof of  Theorem~\ref{thm:fund_ses_general}, with Proposition~\ref{prop:ses_Gmot(X')-Gmot(X)-Gal(X'/X)_Ayoub} playing here the same role as Proposition~\ref{prop:ses_Gmot(X')-Gmot(X)-Gal(X'/X)} there.
	
	Let us focus on the case when $x$ is $k$-rational. 
	Since the sequence in the statement is obtained via Construction~\ref{constr:1KGQ1}, it suffices to check that the hypotheses of Proposition~\ref{prop:KrtimesQ} are satisfied:
	we have to show that the $\GmotAy(k)$-action on $\Locgeo^{\rm{A}}(X)$ is induced by an algebraic action of $\GmotAy(k)$ on $\pi_1^{\rm{A}}(X,x)$ (in the sense of Example~\ref{ex:Q_acts_K}).
	By Lemma~\ref{lem:concretization}, this is the case if and only if the $\GmotAy(k)$-action on $\Locgeo^{\rm{A}}(X)$ admits a concretization $\gamma$ with respect to the fibre functor at $x$ (in the sense of Definition~\ref{defn:concretization}). 
	
	To finish the proof, we exhibit such a concretization.
	For every $\Lambda \in \CAlg_{\Q}$ and every $g \in \GmotAy(k)(\Lambda)$, the natural isomorphism between functors $\Locgeo^{\rm{A}}(X)_{\Lambda} \rightarrow \modules_{\Lambda}$
	\begin{equation*}
		\gamma_{\Lambda,g} \colon x^*[-d](g \cdot -) \xrightarrow{\sim} x^*[-d]
	\end{equation*}
	is defined to be the one filling the composite diagram
	\begin{equation}\label{dia:concretization_x^*}
		\begin{tikzcd}
			\Locgeo^{\rm{A}}(X)_{\Lambda} \arrow{rr}{g \cdot -} \arrow{d}{x^*[-d]} && \Locgeo^{\rm{A}}(X)_{\Lambda} \arrow{d}{x^*[-d]} \\
			\Locgeo^{\rm{A}}(\Spec(k))_{\Lambda} \arrow{rr}{g \cdot -}\arrow[equal]{dr} && \Locgeo^{\rm{A}}(\Spec(k))_{\Lambda} \arrow[equal]{dl} \\
			& \modules_{\Lambda}.
		\end{tikzcd}
	\end{equation}
	Here, the upper rectangle witnesses the compatibility of the action of $g$ with the inverse image functor $x^*[-d]$, while the lower triangle comes from the trivialization of the $\GmotAy(k)$-action on $\Locgeo^{\rm{A}}(\Spec(k))$ discussed in Example~\ref{ex:trivializ-Gmot}.
	In order to check that the natural isomorphisms $\gamma_{\Lambda,g}$ are compatible with composition in $\GmotAy(k)$ and functorial with respect to $\Q$-algebra homomorphisms (as required in Definition~\ref{defn:concretization}), it suffices to check that the same coherence properties hold both for the upper rectangles and for the lower triangles figuring in \eqref{dia:concretization_x^*}. 
	For the upper rectangles, the two coherence properties follow from the very description of $\GmotAy(k)$ as a $\Q$-group scheme parameterizing the autoequivalences of the six functor formalism $X \mapsto \Dbgeo(X)$.
	For the lower triangles, each of the two coherence properties follows from the corresponding coherence property of the trivialization of the $\GmotAy(k)$-action on $\Locgeo^{\rm{A}}(\Spec(k))$ (as stated in Definition~\ref{defn:iso-actions_F}(1)).
\end{proof}

In the following section, we put the categories of motivic local systems just studied into a six functor formalism.
In~\cite{Jacobsen}, the extra functoriality provided by the six operations leads to a method to establish the exactness of the fundamental sequence in the case of a $k$-rational base-point.
Unfortunately, we were not able to implement this method in the proof of Theorem~\ref{thm:ayoub_fund_seq}:
the main issue is that, to do so, we would need to know in advance that the forgetful functor~\eqref{eq:forget_AMLoc} sends semisimple objects to semisimple objects.
For Ivorra--Morel's categories, this semi-simplicity condition follows from the existence of a realization into mixed Hodge modules, obtained by Tubach in~\cite{Tub23} (see~\cite[Lem.~7.6]{Jacobsen}).

\section{Ayoub's categories of Nori motivic sheaves}\label{sect:ayoub_cat}

Motivated by the results of Section~\ref{sect:fund_ses_Voe}, it is natural to look for an alternative construction of motivic perverse sheaves based on the motivic Galois action on perverse sheaves of geometric origin.
In fact, it is useful to start from the motivic Galois action on the stable $\infty$-categories of constructible complexes of geometric origin:
there is a natural $\infty$-categorical notion of homotopy-fixed points, which recovers the $1$-categorical notion of equivariant objects when applied to ordinary categories.
This allows us to study the perverse $t$-structure on the stable $\infty$-categories of homotopy-fixed points.
As for the classical categories of perverse sheaves, the associated realization functor turns out to be an equivalence.

As usual, we work over a field $k$ of characteristic $0$ endowed with a complex embedding $\sigma \colon k \hookrightarrow \C$.

\subsection{The motivic Galois action on constructible complexes}\label{subsect:Dbgeo^Gmot}

Let $X$ be a $k$-variety.
Recall that Theorem~\ref{thm:Anabel} describes a canonical action by the spectral motivic Galois group $\Gmot(k)$ on the stable $\infty$-category $\Dbgeo(X)$, such that the induced action by the underlying classical group $\GmotAy(k)$ respects the ordinary $t$-structure.
In the same spirit as in Subsection~\ref{subsect:Locgeo^Gmot}, we want to consider the associated category of equivariant objects.

However, the $1$-categorical notion of equivariant object is not adapted to the triangulated setting, due to the usual technical defects of triangulated categories:
in order to get a well-behaved triangulated equivariant category, one has to work in the setting of stable $\infty$-categories.
Let us review Ayoub's construction of the $\infty$-category of homotopy-fixed points $D_{\geo}(X)^{\GmotAy(k)}$ in a way adapted to our main applications: 
\begin{nota}
	We introduce the category $\CAlg_{B\GmotAy(k)}$ defined as follows:
	\begin{itemize}
		\item Objects are classical commutative $\Q$-algebras $\Lambda$.
		\item A morphism $\Lambda \rightarrow \Lambda'$ is a pair $(\phi,g)$ consisting of a $\Q$-algebra homomorphism $\phi \colon \Lambda \rightarrow \Lambda'$ and an element $g \in \GmotAy(k)(\Lambda')$.
		\item The composite of two morphisms $(\phi_1,g_1) \colon \Lambda \rightarrow \Lambda'$ and $(\phi_2,g_2) \colon \Lambda' \rightarrow \Lambda''$ is defined as
		\begin{equation*}
			(\phi_2,g_2) \circ (\phi_1,g_1) := (\phi_2 \circ \phi_1, g_2 \phi_1(g_1)) \colon \Lambda \rightarrow \Lambda''.
		\end{equation*}
	\end{itemize}
\end{nota}

There is a natural functor 
\begin{equation*}
	B\GmotAy(k) \colon \CAlg_{\Q} \rightarrow \mathsf{Grpd}
\end{equation*} 
associating to each $\Lambda \in \CAlg_{\Q}$ the group of $\Lambda$-valued points $\GmotAy(k)(\Lambda)$, regarded as a groupoid.
By construction, $\CAlg_{B\GmotAy(k)}$ is the $1$-category of coCartesian sections of the associated coCartesian fibration with $1$-categorical fibres
\begin{equation*}
	\int_{\CAlg_{\Q}} B\GmotAy(k) \rightarrow \CAlg_{\Q}.
\end{equation*}
Here is Ayoub's construction of homotopy-fixed points under the motivic Galois group:

\begin{constr}[{\cite[Constr.~2.3.4]{AyoAnab}}]\label{constr:Dbgeo^Gmot-infty}
	The stable $\infty$-category $\Dgeo(X)$ fits into a functor
	\begin{equation*}
		\Dgeo(X)_{?}: \CAlg_{\Q} \rightarrow \mathsf{Cat}_{\infty}, \quad \Lambda \mapsto \Dgeo(X)_{\Lambda},
	\end{equation*}
	with associated coCartesian fibration
	\begin{equation*}
		\int_{\CAlg_{\Q}} \Dgeo(X)_{?} \rightarrow \CAlg_{\Q}.
	\end{equation*}
	The $\Gmot(k)$-action on $\Dgeo(X)$ promotes the latter to a coCartesian fibration
	\begin{equation*}
		\int_{\Lambda \in \CAlg_{B \GmotAy(k)}} \Dgeo(X)_{?} \rightarrow \CAlg_{B\Gmot^A(k)},
	\end{equation*}
	and we define $\Dgeo(X)^{\GmotAy(k)}$ as the associated stable $\infty$-category of coCartesian sections.
	\qed
\end{constr}

More generally, for any full $\GmotAy(k)$-stable sub-$\infty$-category $\Category \subset \Dgeo(X)$, the same method yields an $\infty$-category of homotopy-fixed points $\Category^{\GmotAy(k)}$, which is stable as soon as $\Category$ is stable.
 
\begin{nota}
	We let
	\begin{equation*}
		\omega_X \colon \Dgeo(X)^{\GmotAy(k)} \rightarrow \Dgeo(X)
	\end{equation*}
	denote the forgetful functor defined by evaluating a coCartesian section at the initial object $\Q \in \CAlg_{\Q}$.
	We use the same notation for the forgetful functor $\Category^{\GmotAy(k)} \to \Category$ on a $\GmotAy(k)$-stable full sub-$\infty$-category $\Category \subset \Dgeo(X)$.
\end{nota}

The main example of $\GmotAy(k)$-stable sub-$\infty$-category of $\Dgeo(X)$ is given by the constructible complexes:

\begin{lem}\label{lem:Dbgeo_Gmot-stable}
	The full sub-$\infty$-category $\Dbgeo(X) \subset \Dgeo(X)$ is $\GmotAy(k)$-stable.
\end{lem}

\begin{proof}
	Indeed, for every $\Lambda \in \CAlg_{\Q}$, the stable $\infty$-category $\Dbgeo(X)_{\Lambda}$ can be identified with the full sub-$\infty$-category of $\Dgeo(X)_{\Lambda}$ spanned by the compact objects, and therefore is stable under self-equivalences.
\end{proof}

The functor $\Dbgeo(X)^{\GmotAy(k)} \rightarrow \Dgeo(X)^{\GmotAy(k)}$ identifies $\Dbgeo(X)^{\GmotAy(k)}$ with the full sub-$\infty$-category of $\Dgeo(X)^{\GmotAy(k)}$ spanned by those objects $K$ whose underlying complex $\omega_X(K) \in \Dgeo(X)$ is constructible.

\begin{prop}[{\cite[Prop.~2.3.5]{AyoAnab}}]\label{prop:Dgeo^Gmot-6ff}
	As $X$ varies among $k$-varieties, the stable $\infty$-categories $\Dgeo(X)^{\GmotAy(k)}$ inherit the six operations, and the forgetful functors 
	\begin{equation*}
		\omega_X \colon \Dgeo(X)^{\GmotAy(k)} \to \Dgeo(X)
	\end{equation*} 
	are canonically compatible with them.
\end{prop}

In order to appreciate the significance of this result, we need to recall a key property of Voevodsky motives proved by Drew--Gallauer:

\begin{thm}[{\cite[Thm.~7.14]{DG22}}]\label{thm:Drew-Gallauer}
	The six functor formalism $X \mapsto \DA(X)$ is $2$-initial among all six functor formalisms on $k$-varieties taking values in presentable stable $\infty$-categories and satisfying étale descent.
\end{thm}

The stable $\infty$-categories $\Dgeo(X)^{\GmotAy(k)}$ are presentable, since so are $\Dgeo(X)$.
As a consequence: 

\begin{cor}[{\cite[Cor.~2.3.6]{AyoAnab}}]\label{cor:Bti-enriched}
	For every $k$-variety $X$, there exists an enriched Betti realization functor
	\begin{equation*}
		\widetilde{\Bti}_X^* \colon \DA(X) \to \Dgeo(X)^{\GmotAy(k)}
	\end{equation*}
	such that the Betti realization over $X$ factors as
	\begin{equation*}
		\Bti_X^* \colon \DA(X) \xrightarrow{\widetilde{\Bti}_X^*} \Dgeo(X)^{\GmotAy(k)} \xrightarrow{\omega_X} \Dgeo(X)
	\end{equation*}
	up to canonical natural isomorphism.
	As $X$ varies, the functors $\widetilde{\Bti}_X^*$ are canonically compatible with the six operations.
\end{cor}

In the following, we are mostly interested in the restriction of the enriched Betti realization to compact motives
\begin{equation*}
	\widetilde{\Bti}_X^* \colon \DAct(X) \to \Dbgeo(X)^{\GmotAy(k)}.
\end{equation*}
Note that the sub-$\infty$-categories $\Dbgeo(X)^{\GmotAy(k)} \subset \Dgeo(X)^{\GmotAy(k)}$ are stable under the six operations, since the analogous property holds for the underlying categories of constructible complexes.
They also inherit Beilinson's gluing functors from the underlying categories $\Dbgeo(X)$.

Informally, $\Dgeo(X)^{\GmotAy(k)}$ can be thought as an $\infty$-category of $\GmotAy(k)$-representations inside $\Dbgeo(X)$. However, since the representation structure is induced by the $\GmotAy(k)$-action, not every object of $\Dbgeo(X)$ can underlie an object of $\Dbgeo(X)^{\GmotAy(k)}$:

\begin{ex}
	Consider the setting of Example~\ref{ex:Artin/X_Ayoub} in the case when $X = \Spec(k)$ and $k'/k$ is a non-trivial finite Galois extension.
	Recall that the complex-analytic space $\Spec(k')^{\sigma}$ consists of $[k' \colon k]$ points, indexed by the set $\Hom_{\sigma}(k',\C)$.
	Every constructible complex of geometric origin over it decomposes into skyscraper complexes accordingly.
	
	As a consequence of Corollary~\ref{cor:Bti-enriched}, the unit sheaf $\Q_{\Spec(k')^{\sigma}} = \Bti_{\sigma,\Spec(k')}^*(\Q_{\Spec(k')}) \in \Pervgeoemb{\sigma}^{\rm{A}}(\Spec(k))$ canonically underlies an object of $\Dbgeo(\Spec(k'))^{\GmotAy(k)}$.
	However, its restrictions to the single points of $\Spec(k')^{\sigma}$ individually do not.
	Indeed, the $\GmotAy(k)$-representation on $\Q_{\Spec(k')^{\sigma}}$ factors through $\Gal(k'/k)$ and matches with the natural $\Gal(k'/k)$-action on $\Spec(k')^{\sigma}$.
	Since the latter is transitive, we deduce that the $\GmotAy(k)$-action on $\Q_{\Spec(k')^{\sigma}}$ permutes the various restrictions transitively.
	\qed 
\end{ex}

\subsection{The perverse $t$-structure}\label{subsect:PervAyoub}

Recall that the perverse $t$-structure on $\Dbct(X)$ restricts to a $t$-structure on $\Dbgeo(X)$ (by Proposition~\ref{prop:Dbgeo_perv_t-str}), which we indicate here as $(\Dbgeo(X)^{\leq_p 0}, \Dbgeo(X)^{\geq_p 0})$.
We want to lift the perverse $t$-structure to Ayoub's category of homotopy-fixed points and describe its heart explicitly.
Recall from Subsection~\ref{subsect:Dbgeo} that we have the filtered union
	\begin{equation*}
	\Dbgeo(X) = \bigcup_{\Sigma \in \Strat_{X/k}} \Dbgeo(X,\Sigma)
\end{equation*}
indexed by the poset $\Strat_{X/k}$ of $k$-stratifications.
Following~\cite[\S~2]{BBDG}, the perverse $t$-structure on $\Dbgeo(X)$ can be recovered by first constructing the so-called $\Sigma$-perverse $t$-structure on $\Dbgeo(X,\Sigma)$ for each $\Sigma \in \Strat_{X/k}$ and then passing to the colimit, exactly as one does for $\Dbct(X)$.

We have the following stability results for the motivic Galois action:

\begin{lem}\label{lem:Gmot^A-respects-strat}
For every $\Sigma \in \Strat_{X/k}$,
the full sub-$\infty$-category
$\Dbgeo(X,\Sigma) \subset \Dbgeo(X)$
is $\GmotAy(k)$-stable.
\end{lem}

\begin{proof}
Fix $\Lambda \in \CAlg_{\Q}$. 
The triangulated category $\Dbgeo(X,\Sigma)_{\Lambda}$ can be identified with the full subcategory of $\Dbgeo(X)_{\Lambda}$ spanned by those complexes $K$ such that, for each stratum $S \in \Sigma$, with inclusion $s \colon S \hookrightarrow X$, the complex $s^* K \in \Dbgeo(S)_{\Lambda}$ is dualizable.
Since the $\GmotAy(k)$-action on constructible complexes
commutes with inverse images,
and the $\GmotAy(k)(\Lambda)$-action
on each $\Dbgeo(S)_{\Lambda}$ preserves dualizable objects
(being monoidal),
the thesis follows.
\end{proof}

\begin{cor}\label{cor:Gmot_respects_Perv}
	The subcategories $\Dbgeo(X)^{\leq_p 0}, \Dbgeo(X)^{\geq_p 0} \subset \Dbgeo(X)$ are both $\GmotAy(k)$-stable.
\end{cor}

\begin{proof}
	Since the $\GmotAy(k)$-action on $\Dbgeo(X)$ respects each $\Dbgeo(X,\Sigma)$ (by Lemma~\ref{lem:Gmot^A-respects-strat}), it suffices to show that the $\Gmot^A(k)$-action on $\Dbgeo(X,\Sigma)$ respects the $\Sigma$-perverse $t$-structure.
	But this follows from the fact that the $\GmotAy(k)$-action preserves the ordinary $t$-structure on each $\Dbgeo(S)$ (by Theorem~\ref{thm:Anabel}(2)), commutes with inverse images, and respects shifts in the triangulated category $\Dbgeo(X)$.
\end{proof}

This implies formally the existence of the sought-after perverse $t$-structure on homotopy-fixed points:

\begin{prop}\label{prop:Perv^Gmot}
	The stable $\infty$-category $\Dbgeo(X)^{\GmotAy(k)}$ carries a unique $t$-structure making the forgetful functor $\omega_X \colon \Dbgeo(X)^{\GmotAy(k)} \rightarrow \Dbgeo(X)$ $t$-exact with respect to the perverse $t$-structure on $\Dbgeo(X)$.
\end{prop}

\begin{proof}
	Since the forgetful functor $\omega_X$ is conservative, any $t$-structure as in the statement must be defined by the formula
	\begin{equation*}
		(\Dbgeo(X)^{\GmotAy(k)})^{\leq_p 0} := \omega_X^{-1}(\Dbgeo(X)^{\leq_p 0}), \quad (\Dbgeo(X)^{\GmotAy(k)})^{\geq_p 0} := \omega_X^{-1}(\Dbgeo(X)^{\geq_p 0}).
	\end{equation*}
	Hence, the sought-after $t$-structure is uniquely determined,
	provided it exists. 
	For every $\Lambda \in \CAlg_{\Q}$,
	the perverse $t$-structure on $\Dbgeo(X)$ induces
	a $t$-structure on $\Dbgeo(X)_{\Lambda}$
	(since $\Lambda$ is flat over $\Q$).
	For every morphism $(\phi,g) \colon \Lambda \to \Lambda'$ in $\CAlg_{B\Gmot^A(k)}$, the induced functor $\Dbgeo(X)_{\Lambda} \rightarrow \Dbgeo(X)_{\Lambda'}$ is $t$-exact with respect to these $t$-structures.
	We deduce that the limit stable $\infty$-category $\Dbgeo(X)^{\GmotAy(k)}$ inherits a uniquely determined compatible $t$-structure.
\end{proof}

As a consequence of Corollary~\ref{cor:Gmot_respects_Perv}, the abelian category $\Dbgeo(X)^{\leq_p 0} \cap \Dbgeo(X)^{\geq_p 0} =: \Pervgeo^{\rm{A}}(X)$ is $\GmotAy(k)$-stable.
By~\cite[Thm.~1.3.6]{BBDG}, the intersection
\begin{equation*}
	(\Dbgeo(X)^{\GmotAy(k)})^{\leq_p 0} \cap (\Dbgeo(X)^{\GmotAy(k)})^{\geq_p 0} = \omega_X^{-1}(\Pervgeo^{\rm{A}}(X))
\end{equation*} 
is an admissible abelian subcategory of $\Dbgeo(X)^{\GmotAy(k)}$, which we refer to as its \textit{perverse heart}.
By construction, it coincides with the essential image of the natural fully faithful functor on homotopy-fixed points
\begin{equation*}
	\Pervgeo^{\rm{A}}(X)^{\GmotAy(k)} \rightarrow \Dbgeo(X)^{\GmotAy(k)}.
\end{equation*}
In the rest of this paper, we tacitly identify $\Pervgeo^{\rm{A}}(X)^{\GmotAy(k)}$ with the perverse heart inside $\Dbgeo(X)^{\GmotAy(k)}$.
The following result offers a more explicit description of this abelian category: 

\begin{lem}\label{lem:Pervgeo^Gmot_naive}
	The abelian category $\Pervgeo^{\rm{A}}(X)^{\GmotAy(k)}$ (defined as in Construction~\ref{constr:Dbgeo^Gmot-infty}) is canonically equivalent to the category of $\Q$-algebraic equivariant objects of Definition~\ref{defn:htpy-fixed_linear}.
\end{lem}

\begin{proof}
	The category $\Pervgeo^{\rm{A}}(X)^{\GmotAy(k)}$ is defined by considering the $\GmotAy(k)$-action on the $1$-category $\Pervgeo^{\rm{A}}(X)$:
	that is to say, $\Pervgeo^{\rm{A}}(X)^{\GmotAy(k)}$ is the $1$-category of coCartesian sections of the coCartesian fibration with $1$-categorical fibres
	\begin{equation*}
		\int_{\Lambda \in \CAlg_{B \GmotAy(k)}} \Pervgeo^{\rm{A}}(X)_{\Lambda} \rightarrow \CAlg_{B \GmotAy(k)}.
	\end{equation*}
	Note that the associated pseudo-functor $\CAlg_{B \GmotAy(k)} \rightarrow \mathsf{Cat_1}$ is described by the formula
	\begin{equation*}
		\Lambda \mapsto \Pervgeo(X)_{\Lambda}, \quad (\phi,g) \mapsto g \cdot (- \otimes_{\Lambda,\phi} \Lambda').
	\end{equation*}
	Unwinding the construction, one sees that the $1$-category $\Pervgeo^{\rm{A}}(X)^{\GmotAy(k)}$ admits the following explicit description:
	\begin{itemize}
		\item
		Objects are pairs $(\underline{K},\underline{\alpha})$ consisting of 
		\begin{itemize}
			\item
			a family $\underline{K} = (K_{\Lambda})_{\Lambda \in \CAlg_{B\GmotAy(k)}}$ of objects $K_{\Lambda} \in \Pervgeo^{\rm{A}}(X)_{\Lambda}$,

			\item
			for every morphism $(\phi,g) \colon \Lambda \rightarrow \Lambda'$ in $\CAlg_{B \GmotAy(k)}$, an isomorphism in $\Pervgeo^{\rm{A}}(X)_{\Lambda'}$
			\begin{equation*}
				\alpha_{(\phi,g)} \colon g \cdot (K_{\Lambda} \otimes_{\Lambda,\phi} \Lambda') \xrightarrow{\sim} K_{\Lambda'}
			\end{equation*} 
			satisfying the natural cocycle condition with respect to composition in $\CAlg_{B \GmotAy(k)}$.
		\end{itemize}

		\item
		A morphism
		$(\underline{K}_1,\underline{\alpha}_1)
		\to (\underline{K}_2,\underline{\alpha}_2)$
		is the datum of
		\begin{itemize}
			\item
			for every $\Lambda \in \CAlg_{B \GmotAy(k)}$, a morphism $K_{1,\Lambda} \rightarrow K_{2,\Lambda}$ in $\Pervgeo^{\rm{A}}(X)_{\Lambda}$
		\end{itemize}
		satisfying the natural compatibility condition with respect to morphisms in $\CAlg_{B\GmotAy(k)}$.
	\end{itemize}
	We claim that the formula
	\begin{equation}\label{eq:comp_2defns_of_Pervgeo^Gmot}
		K_{\Lambda} := K \otimes \Lambda, \quad \alpha_{(\phi,g)} \colon g \cdot ((K \otimes \Lambda) \otimes_{\Lambda,\phi} \Lambda') = g \cdot (K \otimes \Lambda') \xrightarrow{\alpha_{\Lambda',g}} K \otimes \Lambda'
	\end{equation}
	defines an equivalence from the category of Definition~\ref{defn:htpy-fixed_linear} to the latter category of coCartesian sections.
	To see this, it suffices to note that every morphism $(\phi,g) \colon \Lambda \rightarrow \Lambda'$ in $\CAlg_{B\GmotAy(k)}$ can be factored as
	\begin{equation*}
		(\phi,g) \colon \Lambda \xrightarrow{(\phi,1_{\Lambda'})} \Lambda' \xrightarrow{(\id_{\Lambda'},g)} \Lambda'. 
	\end{equation*}
	Hence, by the cocycle condition, the isomorphism $\alpha_{(\phi,g)}$ associated to a coCartesian section $(\underline{K},\underline{\alpha})$ can be reconstructed from $\alpha_{(\phi,1_{\Lambda'})} \colon K_{\Lambda} \otimes_{\Lambda,\phi} \Lambda' \xrightarrow{\sim} K_{\Lambda'}$ and $\alpha_{(\id_{\Lambda'},g)} \colon g \cdot K_{\Lambda'} \xrightarrow{\sim} K_{\Lambda'}$.
	Similarly, the compatibility condition for a morphism of coCartesian sections $(\underline{K}_1,\underline{\alpha}_1) \rightarrow (\underline{K}_2,\underline{\alpha}_2)$ with respect to $(\phi,g)$ holds as soon as the compatibility conditions with respect to $(\phi,1_{\Lambda'})$ and $(\id_{\Lambda'},g)$ hold.
	This observation allows one to construct a quasi-inverse to the functor~\eqref{eq:comp_2defns_of_Pervgeo^Gmot}, thus proving the claim.
\end{proof}

Since the inclusion $\Locgeo^{\rm{A}}(X) \subset \Pervgeo^{\rm{A}}(X)$ induces a fully faithful inclusion on equivariant objects, we obtain a sequence of fully faithful embeddings
\begin{equation*}
	\Locgeo^{\rm{A}}(X)^{\GmotAy(k)} \subset \Pervgeo^{\rm{A}}(X)^{\GmotAy(k)} \subset \Dbgeo(X)^{\GmotAy(k)}
\end{equation*}
compatible with the forgetful functor $\omega_X$.
This puts the theory of equivariant local systems of Section~\ref{sect:fund_ses_Voe} into a six functor formalism.
As $X$ varies, the abelian categories $\Pervgeo^{\rm{A}}(X)^{\GmotAy(k)}$ enjoy the same functoriality as the classical categories of perverse sheaves of geometric origin:
in particular, they are respected by shifted inverse images under smooth morphisms, direct image under closed immersions, extensions by zero and direct images under affine open immersions, the external tensor product, as well as by Beilinson's gluing functor.

For future reference, let us record the behaviour of the theory under finite extensions of the base field $k$:
\begin{lem}\label{lem:PervvAyoub_X'/X}
	Let $k'/k$ be a finite Galois extension.
	Given a smooth $k$-variety $X$, set $X' := X \times_k k'$, and write $e_X \colon X' \to X$ for the corresponding finite Galois covering.
	Then:
	\begin{enumerate}
		\item The abelian categories $\Pervgeoemb{\sigma'}^{\rm{A}}(X')^{\GmotAy(k')}$ and $\Pervgeoemb{\sigma}^{\rm{A}}(X')^{\GmotAy(k)}$ (defined by regarding $X'$ as a $k'$-variety or as a $k$-variety, respectively) are canonically equivalent.
		\item The inverse image functor $e_X^* \colon \Pervgeoemb{\sigma}^{\rm{A}}(X)^{\GmotAy(k)} \rightarrow \Pervgeoemb{\sigma}^{\rm{A}}(X')^{\GmotAy(k)}$ induces a canonical equivalence
		\begin{equation*}
			\Pervgeoemb{\sigma}^{\rm{A}}(X)^{\GmotAy(k)} \xrightarrow{\sim} (\Pervgeoemb{\sigma}^{\rm{A}}(X')^{\GmotAy(k)})^{\Gal(k'/k)},
		\end{equation*}
		where the action of $\Gal(k'/k)$ on $\Pervgeoemb{\sigma}^{\rm{A}}(X')^{\GmotAy(k)}$ is induced by its action on $X'$ via $k$-automorphisms.
	\end{enumerate}
\end{lem}
\begin{proof}
	Same proof as for Lemma~\ref{lem::AMLoc_X'/X}.
\end{proof}

\subsection{Beilinson's equivalence}

In his celebrated paper~\cite{BeilinsonEquivalence}, Beilinson proved that the bounded derived category $\Db(\Perv(X))$ is canonically equivalent to $\Dbct(X)$:
the equivalence is given by the realization functor
\begin{equation*}
	\real_X \colon \Db(\Perv(X)) \rightarrow \Dbct(X),
\end{equation*}
which extends the natural inclusion $\Perv(X) \subset \Dbct(X)$.
The construction of the realization functor uses the canonical filtered enhancement of the triangulated category $\Dbct(X)$, which comes from its structure of stable $\infty$-category.
By construction, Beilinson's realization restricts to a $\GmotAy(k)$-equivariant triangulated functor
\begin{equation*}
	\real_X \colon \Db(\Pervgeo^{\rm{A}}(X)) \rightarrow \Dbgeo(X).
\end{equation*}
The latter is still an equivalence, as one can see by going through Beilinson's argument:
first one establishes a generic version of the sought-after equivalence, then one deduces the full equivalence arguing by induction.
Indeed, both steps use the six operations as well as the gluing functors, but these preserve sheaves of geometric origin (by Proposition~\ref{prop:Dbgeo-6ff}).
We refer to~\cite[Thm.~1.6.36]{AyoAnab} for more details on the proof of the generic equivalence for sheaves of geometric origin.

Beilinson's result has a natural version for equivariant objects:

\begin{thm}[{\cite[Thm.~1.3]{BeilinsonEquivalence}}]\label{thm:Bei-equiv}
	For every $k$-variety $X$, the realization functor
	\begin{equation*}
		\real_X \colon \Db(\Pervgeo^{\rm{A}}(X)^{\GmotAy(k)}) \rightarrow \Dbgeo(X)^{\GmotAy(k)}
	\end{equation*}
	is an equivalence of stable $\infty$-categories.
\end{thm}

\begin{proof}
	We ignore whether it is possible to deduce the result formally from the analogous result for classical perverse sheaves (of geometric origin).
	However, it suffices to check that Beilinson's proof goes through in the $\GmotAy(k)$-equivariant setting.
	This is the case since Beilinson's argument is based on the six operations and the gluing functors, which lift to the $\GmotAy(k)$-equivariant categories;
	we leave the details to the interested reader.
\end{proof}

We conclude the present section with a complementary result about the $\GmotAy(k)$-action on $\Dbgeo(X)$, which does not play any role in the rest of the article but is nonetheless of independent interest:

\begin{prop}\label{prop:Gmot_algebraic_Dbgeo}
	For every $k$-variety $X$, it is possible to write the stable $\infty$-category $\Dbgeo(X)$ as a filtered colimit of the form
	\begin{equation*}
		\Dbgeo(X) = \twocolim_{i \in I} \Der_i(X)
	\end{equation*} 
	where $(\Der_i(X))_{i \in I}$ is a filtered direct system of stable $\infty$-categories endowed with compatible $\GmotAy(k)$-actions such that the action of the pro-algebraic group $\GmotAy(k)$ on each $\Der_i(X)$ factors through some algebraic quotient $Q_i$.
\end{prop}

\begin{proof}
	Since Beilinson's equivalence
	\begin{equation*}
		\real_X \colon \Db(\Pervgeo^{\rm{A}}(X)) \xrightarrow{\sim} \Dbgeo(X)
	\end{equation*}
	is $\GmotAy(k)$-equivariant, it suffices to prove the thesis for the derived $\infty$-category $\Db(\Pervgeo^{\rm{A}}(X))$.
	Moreover, since the formation of bounded derived categories is compatible with filtered colimits of abelian categories (see~\cite[Lem.~2.6]{Gal21}), it suffices to prove the analogous statement for the abelian category $\Pervgeo^{\rm{A}}(X)$:
	we claim that the latter can be written as a filtered union of the form
	\begin{equation*}
		\Pervgeo^{\rm{A}}(X) = \bigcup_{i \in I} \Abelian_i(X)
	\end{equation*}
	where each $\Abelian_i$ is a $\GmotAy(k)$-stable full abelian subcategory of $\Pervgeo^{\rm{A}}(X)$ on which the $\GmotAy(k)$-action factors through some algebraic quotient $Q_i$.
	For this, we proceed as follows.
	
	First, write $\Pervgeo^{\rm{A}}(X)$ as the filtered union
	\begin{equation*}
		\Pervgeo^{\rm{A}}(X) = \bigcup_{\Sigma \in \Strat_{X/k}} \Pervgeo^{\rm{A}}(X,\Sigma).
	\end{equation*}
	Since each $\Pervgeo^{\rm{A}}(X,\Sigma) = \Dbgeo(X,\Sigma) \cap \Pervgeo^{\rm{A}}(X)$ is $\GmotAy(k)$-stable (by Lemma~\ref{lem:Dbgeo_Gmot-stable} and Corollary~\ref{cor:Gmot_respects_Perv}), it suffices to prove the claim for each $\Pervgeo^{\rm{A}}(X,\Sigma)$ individually in place of $\Pervgeo^{\rm{A}}(X)$. 
	
	So fix one stratification $\Sigma \in \Strat_{X/k}$.
	Recall that, by definition, all strata $S \in \Sigma$ are smooth over $k$ and connected.
	For each stratum $S \in \Sigma$, there is a finite field extension $k'/k$ such that the connected components of the smooth $k'$-variety $S \times_k k'$ are geometrically connected over $k'$
	(take $k'$ to be the algebraic closure of $k$ inside the function field $k(S)$).
	Clearly, for every further finite extension $k''/k'$, the $k''$-variety $S \times_k k''$ enjoys the same property with respect to $k''$.
	Hence, it is possible to choose a finite extension $k'/k$ which works for all strata $S \in \Sigma$ at once.
	Now consider the $k'$-variety $X' := X \times_k k'$.
	Note that the collection of all connected components of $S \times_k k'$ for each $S \in \Sigma$ defines a $k'$-algebraic stratification $\Sigma' \in \Strat_{X'/k'}$ in the sense of Notation~\ref{nota:Strat_X/k}:
	indeed, property (i) holds by construction, property (ii) follows from the corresponding property for $\Sigma$, and property (iii) follows from the inclusions between Zariski closures
	\begin{equation*}
		\overline{S \times_k k'} \subset \bar{S} \times_k k'
	\end{equation*} 
	inside $X'$.
	Now choose a complex embedding $\sigma' \colon k' \hookrightarrow \C$ extending $\sigma \colon k \hookrightarrow \C$.
	The $\GmotAy(k')$-equivariant equivalence of Lemma~\ref{lem:Perv_X'/X} restricts to a fully faithful inclusion 
	\begin{equation*}
		\Pervgeo^{\rm{A}}(X,\Sigma) \subset \Pervgeo^{\rm{A}}(X',\Sigma').
	\end{equation*}
	Since $\GmotAy(k')$ has finite index in $\GmotAy(k)$ (by Proposition~\ref{prop:ses_Gmot'-Gmot-Gal_Ayoub}), it suffices to prove the claim for the $\GmotAy(k')$-action on $\Pervgeo^{\rm{A}}(X',\Sigma')$.
	In other words, up to replacing the $k$-variety $X$ by the $k'$-variety $X'$ and $\Sigma$ by $\Sigma'$, we may assume that all strata in $\Sigma$ are geometrically connected over $k$.
	
	In this situation, for each stratum $S \in \Sigma$ there is a finite extension $k'/k$ such that the geometrically connected $k'$-variety $S \times_k k'$ admits a $k'$-rational point.
	As before, it is possible to choose a finite extension $k'/k$ which works for all strata $S \in \Sigma$ at once. 
	Up to replacing the $k$-variety $X$ by the $k'$-variety $X' := X \times_k k'$ and the stratification $\Sigma$ by the corresponding stratification $\Sigma'$ as in the previous step, we may assume that each stratum in $\Sigma$ admits a $k$-rational point.
	
	Since the $\GmotAy(k)$-action on sheaves of geometric origin commutes with the six operations, the $\GmotAy(k)$-action on $\Pervgeo^{\rm{A}}(X,\Sigma)$ is compatible with restriction to each stratum $S \in \Sigma$.
	Hence, it suffices to prove the claim for the abelian category $\Locgeo^{\rm{A}}(S)$ for each $S$ individually in place of $\Pervgeo^{\rm{A}}(X,\Sigma)$.
	In other words, up to replacing $X$ by $S$, we may assume that $X$ is smooth and geometrically connected over $k$, with a rational point $x \in X(k)$, and that $\Sigma = \left\{X\right\}$.
	
	Finally, in this situation, we prove the claim for the abelian category $\Locgeo^{\rm{A}}(X) = \Pervgeo^{\rm{A}}(X,\left\{X\right\})$.
	Following Notation~\ref{nota:pi1geo-Gmot(X)}, let $\pi_1^{\rm{A}}(X,x)$ denote the Tannaka dual of $\Locgeo^{\rm{A}}(X)$ at the $k$-rational point $x$. 
	As explained in the proof of Theorem~\ref{thm:ayoub_fund_seq}, the $\GmotAy(k)$-action on $\Locgeo^{\rm{A}}(X)$ is induced by an algebraic action of $\GmotAy(k)$ on $\pi_1^{\rm{A}}(X,x)$ (in the sense of Example~\ref{ex:Q_acts_K}).
	Form the resulting semi-direct product
	$\pi_1^{\rm{A}}(X,x) \rtimes \GmotAy(k)$,
	and let $\left\{G_i\right\}_{i \in I}$ denote
	the filtered poset of all its algebraic quotients;
	for each $i \in I$, let
	$K_i$ and $Q_i$
	denote the image of
	$\pi_1^{\rm{A}}(X,x)$ and $\GmotAy(k)$,
	respectively,
	in $G_i$.
	Since the algebraic group $K_i$ is a quotient of $\pi_1^{\rm{A}}(X,x)$, it is the Tannaka dual of a $\GmotAy(k)$-stable Tannakian subcategory $\Abelian_i$ of $\Locgeo^{\rm{A}}(X)$.
	By construction, the $\GmotAy(k)$-action on $\Abelian_i$ is induced by the conjugation action of $\GmotAy(k)$ on $K_i$, and the latter factors through its algebraic quotient $Q_i$.
	Lastly, since $K_i$ is algebraic, the Tannakian category $\Abelian_i$ is tensor-generated by one object $L_i \in \Locgeo^{\rm{A}}(X)$.
	But, since $\pi_1^{\rm{A}}(X,x)$ is a closed subgroup of the semi-direct product, every object of $\Locgeo^{\rm{A}}(X)$ is a subquotient of some representation of the semi-direct product (by~\cite[Prop.~2.21(b)]{deligne-milne}).
	This implies that we have the filtered union
	\begin{equation*}
		\Locgeo^{\rm{A}}(X) = \bigcup_{i \in I} \Abelian_i,
	\end{equation*}
	as wanted.
\end{proof}

\section{The comparison theorem}\label{sect:comp-thm}

The present section contains the main result of this paper,
Theorem~\ref{thm:comp} below.
That is,
we show that Ayoub's categories
of equivariant sheaves under the motivic Galois group
recover Ivorra--Morel's categories,
in a very precise sense.

Exploiting the universal property of Ivorra--Morel's categories, we easily construct a system of comparison functors to Ayoub's categories:
our main result asserts that these functors are all equivalences.
Using the six functor formalism, we deduce the result from its generic variant, asserting that our comparison functors induce equivalences at the generic points of all $k$-varieties.
As a byproduct, we show that the two notions of perverse sheaf of geometric origin introduced in the previous sections define the same category, and similarly for the two notions of local system of geometric origin.

As usual, we work over a field $k$ of characteristic $0$,
endowed with a complex embedding $\sigma \colon k \hookrightarrow \C$.

\subsection{Construction of the comparison functors}

Let $X$ be a (quasi-projective) $k$-variety.
Recall from Section~\ref{subsect:IM_cat} that Ivorra--Morel's category $\MNori(X)$ is universal among all abelian categories factoring the homological functor $\pH^0 \circ \Bti_X^* \colon \DAct(X) \to \Perv(X)$.
Since the image of the latter is contained in $\Pervgeo^{\rm{A}}(X)$, the forgetful functor $\iota_X \colon \MNori(X) \to \Perv(X)$ factors through $\Pervgeo^{\rm{A}}(X)$ as well;
for sake of simplicity, we use the same notation for the refined forgetful functor
\begin{equation*}
	\iota_X \colon \MNori(X) \to \Pervgeo^{\rm{A}}(X).
\end{equation*}
As we explain now, the latter factors through the category of equivariant objects $\Pervgeo^{\rm{A}}(X)^{\GmotAy(k)}$.
The idea is that Ayoub's construction defines a theory of enhanced perverse sheaves through which the Betti realization factors.

\begin{nota}
	For sake of simplicity, in the rest of the present section we write the category of equivariant objects $\Pervgeo^{\rm{A}}(X)^{\GmotAy(k)}$ as $\PervAyoub(X)$.
	Similarly, when $X$ is smooth over $k$, we write its full subcategory $\Locgeo^{\rm{A}}(X)^{\GmotAy(k)}$ as $\AMLoc(X)$.
\end{nota}

\begin{constr}
	Recall that Ayoub's Corollary~\ref{cor:Bti-enriched} provides a factorization of the Betti realization $\Bti_X^* \colon \DAct(X) \to \Dbct(X)$ as
	\begin{equation*}
		\Bti_X^* \colon \DAct(X) \xrightarrow{\widetilde{\Bti}_X^*} \Dbgeo(X)^{\GmotAy(k)} \xrightarrow{\omega_X} \Dbct(X).
	\end{equation*}
	We deduce that the homological functor $\pH^0 \circ \Bti_X^* \colon \DAct(X) \to \Perv(X)$ factors as
	\begin{equation*}
		\pH^0 \circ \Bti_X^* \colon \DAct(X) \xrightarrow{\pH^0 \circ \widetilde{\Bti}_X^*} \PervAyoub(X) \xrightarrow{\omega_X} \Perv(X)
	\end{equation*}
	up to natural isomorphism.
	The abelian category $\PervAyoub(X)$, together with the homological functor $\pH^0 \circ \widetilde{\Bti}_X^* \colon \DAct(X) \to \PervAyoub(X)$ and the faithful exact functor $\omega_X \colon \PervAyoub(X) \to \Perv(X)$, satisfies the assumptions on the abelian category $\Abelian(X)$ appearing in the statement of the universal property of $\MNori(X)$ in Subsection~\ref{subsect:IM_cat}; 
	the natural isomorphism between functors $\DAct(X) \to \Perv(X)$
	\begin{equation*}
		\kappa_X \colon \pH^0 \circ \Bti_X^* \xrightarrow{\sim} \omega_X \circ (\pH^0 \circ \widetilde{\Bti}_X^*) = \pH^0 \circ (\omega_X \circ \widetilde{\Bti}_X^*)
	\end{equation*} 
	mentioned in the statement is induced by the natural isomorphism $\Bti_X^* \xrightarrow{\sim} \omega_X \circ \widetilde{\Bti}_X^*$.
	Applying the universal property of $\MNori(X)$, we obtain a canonical faithful exact functor
	\begin{equation*}
		o_X \colon \MNori(X) \rightarrow \PervAyoub(X)
	\end{equation*}
	making the diagram
	\begin{equation*}
		\begin{tikzcd}
			\MNori(X) \arrow{rr}{o_X} \arrow{drr}[']{\iota_X} && \PervAyoub(X) \arrow{d}{\omega_X} \\
			&& \Perv(X)
		\end{tikzcd}
	\end{equation*}
	commute up to the induced natural isomorphism $\tilde{\kappa}_X \colon \iota_X \xrightarrow{\sim} \omega_X \circ o_X$.
	\qed
\end{constr}

\begin{rem}\label{rem:MNori_inside_PervAyoub}
	As mentioned in Section~\ref{subsect:IM_cat}, in this situation $\MNori(X)$ can be identified with the universal abelian factorization of the homological functor $\pH^0 \circ \widetilde{\Bti}_X^*$.
\end{rem}

Passing to the bounded derived categories, the functor $o_X \colon \MNori(X) \to \PervAyoub(X)$ extends to a conservative triangulated functor
\begin{equation*}
	o_X \colon \Db(\MNori(X)) \rightarrow \Db(\PervAyoub(X))
\end{equation*}
making the diagram
\begin{equation*}
	\begin{tikzcd}
		\Db(\MNori(X)) \arrow{rr}{o_X} \arrow{drr}[']{\iota_X} && \Db(\PervAyoub(X)) \arrow{d}{\omega_X} \\
		&& \Db(\Perv(X))
	\end{tikzcd}
\end{equation*}
commute up to canonical natural isomorphism. 
We can finally state our comparison result as follows:

\begin{thm}\label{thm:comp}
	For every (quasi-projective) $k$-variety $X$, the triangulated comparison functor
	\begin{equation*}
		o_X \colon \Db(\MNori(X)) \rightarrow \Db(\PervAyoub(X)) = \Dbgeo(X)^{\GmotAy(k)}
	\end{equation*}  
	is an equivalence.
\end{thm}

The equivalence $\Db(\PervAyoub(X)) = \Dbgeo(X)^{\GmotAy(k)}$ appearing in the formula is the one of Theorem~\ref{thm:Bei-equiv}.

The rest of the present section is devoted to proving the comparison theorem.
For sake of simplicity, in the following we limit ourselves to considering quasi-projective $k$-varieties;
since the assignments $X \mapsto \MNori(X)$ and $X \mapsto \PervAyoub(X)$, regarded as fibered categories over open immersions between $k$-varieties, are both stacks for the Zariski topology, the equivalence over general bases follows formally from the equivalence over quasi-projective (or even just affine) bases.
The advantage of working in the quasi-projective setting is that we can exploit the six functor formalism of Ivorra--Morel's categories systematically.

The following observation plays a crucial role in each step of the proof of the comparison theorem:

\begin{prop}\label{prop:o_S-6ff}
	As $X$ varies among (quasi-projective) $k$-varieties, the triangulated comparison functors $o_X$ are canonically compatible with the six operations, as well as with Beilinson's gluing functors.
\end{prop}

Essentially, this is just a formal consequence of how the six operations and the gluing functors on Ivorra--Morel's derived categories are constructed:
indeed, one can repeat the whole construction of~\cite{IM19} and~\cite{TerenziNori} using the description of $\MNori(X)$ as the universal abelian factorization of $\pH^0 \circ \widetilde{\Bti}_X^*$ in place of its original definition in terms of $\pH^0 \circ \Bti_X^*$ (see Remark~\ref{rem:MNori_inside_PervAyoub} above).
To make this point clear, let us discuss the case of inverse images under open immersions in detail, along the lines of Example~\ref{ex:j^*}:
\begin{ex}\label{ex:o_X-open_immersion}
	Fix an open immersion of (quasi-projective) $k$-varieties $j \colon U \hookrightarrow X$.
	The triangulated functor
	\begin{equation*}
		j^* \colon \Dbgeo(X)^{\GmotAy(k)} \rightarrow \Dbgeo(U)^{\GmotAy(k)}
	\end{equation*}
	is $t$-exact for the perverse $t$-structure defined in Section~\ref{subsect:PervAyoub} (as the underlying functor on constructible complexes is $t$-exact).
	Thus, it restricts to an exact functor
	\begin{equation*}
		j^* \colon \PervAyoub(X) \rightarrow \PervAyoub(U),
	\end{equation*}
	and in fact, by~\cite[Thm.~1]{Vologodsky}, it can be identified with the derived functor
	\begin{equation*}
		j^* \colon \Db(\PervAyoub(X)) \rightarrow \Db(\PervAyoub(U))
	\end{equation*}
	modulo the equivalences of Theorem~\ref{thm:Bei-equiv}.
	Since the upper half of the diagram
	\begin{equation}\label{dia:j^*-Bti(2)}
		\begin{tikzcd}
			\DAct(X) \arrow{rr}{j^*} \arrow{d}[']{\widetilde{\Bti}^*_X} && \DAct(U) \arrow{d}{\widetilde{\Bti}^*_U} \\
			\Dbgeo(X)^{\GmotAy(k)} \arrow{rr}{j^*} \arrow{d}[']{\pH^0} && \Dbgeo(U)^{\GmotAy(k)} \arrow{d}{\pH^0} \\
			\PervAyoub(X) \arrow{rr}{j^*} && \PervAyoub(U)
		\end{tikzcd}
	\end{equation}
	commutes up to canonical natural isomorphism (by Corollary~\ref{cor:Bti-enriched}), the same holds for the outer rectangle.
	Applying~\cite[Prop.~2.5]{TerenziFunctoriality}, one gets a canonical exact functor
	\begin{equation}\label{formula:j^*-MNori(2)}
		j^* \colon \MNori(X) \rightarrow \MNori(U)
	\end{equation}
	rendering both halves of the diagram
	\begin{equation}\label{dia:j^*-MNori(2)}
		\begin{tikzcd}
			\DAct(X) \arrow{rr}{j^*} \arrow{d}[']{h_X} && \DAct(U) \arrow{d}{h_U} \\
			\MNori(X) \arrow{rr}{j^*} \arrow{d}[']{o_X} && \MNori(U) \arrow{d}{o_U} \\
			\PervAyoub(X) \arrow{rr}{j^*} && \PervAyoub(U)
		\end{tikzcd}
	\end{equation}
	commutative up to natural isomorphism.
	More precisely, \eqref{formula:j^*-MNori(2)} is the unique exact functor which makes the upper half of~\eqref{dia:j^*-MNori(2)} commute on the nose (see Remark~\ref{rem:j^*-unique}), hence it coincides with the functor~\eqref{formula:j^*-MNori} of Example~\ref{ex:j^*}.
	The natural isomorphism filling the lower half of~\eqref{dia:j^*-MNori(2)} is then uniquely determined by its compatibility with the natural isomorphism filling~\eqref{dia:j^*-Bti(2)} (see again Remark~\ref{rem:j^*-unique}):
	this means that the composite of the natural isomorphisms filling the two halves of the diagram
	\begin{equation*}
		\begin{tikzcd}
			\MNori(X) \arrow{rr}{j^*} \arrow{d}[']{o_X} && \MNori(U) \arrow{d}{o_U} \\
			\PervAyoub(X) \arrow{rr}{j^*} \arrow{d}[']{\omega_X} && \PervAyoub(U) \arrow{d}{\omega_U} \\
			\Perv(X) \arrow{rr}{j^*} && \Perv(U)
		\end{tikzcd}
	\end{equation*}
	coincides with the natural isomorphism filling the lower half of~\eqref{dia:j^*-MNori}.
	Passing to the bounded derived categories, we obtain a similar decomposition of~\eqref{dia:j^*-iota} as
	\begin{equation*}
		\begin{tikzcd}
			\Db(\MNori(X)) \arrow{rr}{j^*} \arrow{d}[']{o_X} && \Db(\MNori(U)) \arrow{d}{o_U} \\
			\Db(\PervAyoub(X)) \arrow{d}[']{\omega_X} \arrow{rr}{j^*} && \Db(\PervAyoub(U)) \arrow{d}{\omega_U} \\
			\Db(\Perv(X)) \arrow{rr}{j^*} && \Db(\Perv(U)).
		\end{tikzcd}
	\end{equation*}
	The natural isomorphism filling the upper half of the latter diagram witnesses the compatibility of the comparison functors with the inverse image functors $j^*$.
	By construction, these natural isomorphisms are compatible with composition of open immersions, so they define a morphism of triangulated fibered categories over the category of open immersions between $k$-varieties.
	\qed
\end{ex}

\begin{proof}[Proof of Proposition~\ref{prop:o_S-6ff}]
	The argument in Example~\ref{ex:o_X-open_immersion} only uses the $t$-exactness of inverse images under open immersions.
	The same argument applies to shifted inverse images under smooth morphisms and to direct images under closed immersions:
	this yields analogous natural isomorphisms witnessing the compatibility of the comparison functors with these two classes of functors, compatibly with composition in each of the two classes separately.
	Starting from these two partial definitions, and repeating the arguments of~\cite[\S~4.2]{IM19} word-by-word, one obtains a canonical system of natural isomorphisms witnessing the compatibility of the comparison functors with inverse images under arbitrary morphisms of (quasi-projective) $k$-varieties.
	
	To this end, in view of how inverse images under closed immersions are constructed in~\cite[\S~4.1]{IM19}, one needs to know in advance that the comparison functors are compatible with Beilinson's gluing functors, in a similar way.
	But, since the gluing functors on perverse Nori motives are defined by direct application of the lifting result~\cite[Prop.~2.5]{TerenziFunctoriality}, the needed compatibility is obtained exactly as in the case of open immersions.
	Once the compatibility of the comparison functors with inverse images is established, their compatibility with the four operations of type $f^*$, $f_*$, $f_!$ and $f^!$ follows formally:
	for instance, compatibility with direct images follows from the validity of the same property after composition with the conservative functors $\omega_X \colon \Dbgeo(X)^{\GmotAy(k)} \to \Dbct(X)$. 

	The above argument for inverse images under open immersion and, more generally, for $t$-exact functors, extends naturally to the case of multilinear functors which are $t$-exact in each variable (using~\cite[\S~4]{TerenziFunctoriality}).
	In particular, for every choice of (quasi-projective) $k$-varieties $X_1$ and $X_2$, one obtains a new construction of the external tensor product functor
	\begin{equation*}
		- \boxtimes - \colon \MNori(X_1) \times \MNori(X_2) \to \MNori(X_1 \times X_2),
	\end{equation*}
	in such a way that the diagram
	\begin{equation*}
		\begin{tikzcd}
			\MNori(X_1) \times \MNori(X_2) \arrow{rr}{- \boxtimes -} \arrow{d}[']{o_{X_1} \times o_{X_2}} && \MNori(X_1 \times X_2) \arrow{d}{o_{X_1 \times X_2}}  \\
			\PervAyoub(X_1) \times \PervAyoub(X_2) \arrow{rr}{- \boxtimes -} && \PervAyoub(X_1 \times X_2)
		\end{tikzcd}
	\end{equation*}
	commutes up to natural isomorphism.
	In order to see that the natural isomorphisms just constructed are compatible with inverse image functors as well as with associativity, commutativity, and unit constraints, it now suffices to repeat the arguments of~\cite[\S\S~2,4]{TerenziNori} word-by-word.
	This shows that the triangulated comparison functors are unitary symmetric monoidal, compatibly with inverse images.
	Moreover, they automatically respect internal homomorphisms, because, again, this is true after composition with the conservative functors $\omega_X \colon \Dbgeo(X)^{\GmotAy(k)} \to \Dbgeo(X)$.
\end{proof}

\subsection{Comparison over a point}\label{subsect:comp-point}

The first step in the proof of Theorem~\ref{thm:comp} is to establish the sought-after equivalence in the case when $X = \Spec(k)$ or, more generally, when $X = \Spec(k')$ for a finite extension $k'/k$.

Recall from Example~\ref{ex:trivializ-Gmot} that the Tannakian category $\PervAyoub(k)$ is canonically equivalent to $\Rep_{\Q}(\GmotAy(k))$.
Since the comparison functor
\begin{equation*}
	o_k \colon \MNori(k) \to \PervAyoub(k)
\end{equation*}
is canonically monoidal (by Proposition~\ref{prop:o_S-6ff}) and compatible with the forgetful functors to $\Perv(\Spec(k)) = \vect_{\Q}$ (by construction), it determines a homomorphism of Tannaka dual groups
\begin{equation*}
	\GmotAy(k) \to \GmotNo(k),
\end{equation*}
and saying that $o_k$ is an equivalence amounts to saying that this homomorphism is an isomorphism.
Before proceeding further, we need a digression on the relation between Nori's and Ayoub's motivic Galois groups.
For sake of simplicity, we limit ourselves to the motivic Galois group of the base field $k$, but of course the same discussion applies to finite extensions of $k$ as well.

The starting point is the following important result:

\begin{thm}[{\cite[Thm.~10.1.1]{HMS17}}, {\cite[Prop.~7.12]{CG17}}]\label{thm:Nori_real_k}
	There exists a canonical tensor-triangulated functor
	\begin{equation*}
		\Nri_k^* \colon \DAct(k) \rightarrow \Db(\MNori(k))
	\end{equation*}
	such that the Betti realization over $k$ factors as
	\begin{equation*}
		\Bti_k^* \colon \DAct(k) \xrightarrow{\Nri_k^*} \Db(\MNori(k)) \xrightarrow{\iota_k} \Db(\vect_{\Q})
	\end{equation*}
	up to monoidal natural isomorphism.
\end{thm}

This was essentially proved by Nori, using his own construction of $\MNori(k)$.
The fact that Nori's category coincides with the universal abelian factorization of the homological functor $H^0 \circ \Bti_k^*$ (as in Ivorra--Morel's definition) is a consequence of this result.

Let $\Ocal(\GmotNo(k))$ denote the Hopf algebra of regular functions on the pro-algebraic group $\GmotNo(k)$.
By construction, Nori's realization functor extends canonically to a tensor-triangulated functor
\begin{equation*}
	\Nri_k^* \colon \DA(k) \rightarrow D(\Ind\MNori(k)) = \mathsf{coMod}_{\Ocal(\GmotNo(k))}(D(\Q))
\end{equation*}
making the two halves of the diagram
\begin{equation*}
	\begin{tikzcd}
		D(\Q)  \arrow{d}[']{s_k} \arrow{drr} \\
		\DA(k) \arrow{rr}{\Nri_k^*} \arrow{drr}[']{\Bti_k^*} && \mathsf{coMod}_{\Ocal(\GmotNo(k))}(D(\Q)) \arrow{d}{\for_{\Ocal(\GmotNo(k))}} \\
		&& D(\Q)
	\end{tikzcd}
\end{equation*}
commute up to monoidal natural isomorphism.
Recall from Section~\ref{subsect:GmotAyoub} that Ayoub's derived Hopf algebra $\mathcal{H}_k$ is universal among all Hopf algebras giving rise to a commutative diagram of this form.
Therefore, we obtain a canonical bialgebra morphism in $D(\Q)$
\begin{equation}\label{eq:H_k-to-O(Gmot^N)}
	\mathcal{H}_k \to \Ocal(\GmotNo(k))
\end{equation}
making the two halves of the diagram
\begin{equation*}
	\begin{tikzcd}
		\DA(k) \arrow{rr}{\Nri_k^*} \arrow{d} && \mathsf{coMod}_{\Ocal(\GmotNo(k))}(D(\Q)) \arrow{d}{\for_{\Ocal(\GmotNo(k))}} \\
		\mathsf{coMod}_{\mathcal{H}_k}(D(\Q)) \arrow{urr} \arrow{rr}[']{\for_{\mathcal{H}_k}} && D(\Q)
	\end{tikzcd}
\end{equation*}
commute up to monoidal natural isomorphism.
Since the complex $\Ocal(\GmotNo(k)) \in D(\Q)$ is concentrated in degree $0$, the morphism~\eqref{eq:H_k-to-O(Gmot^N)} factors through $\mathcal{H}_k^0$.
This in turn determines a canonical homomorphism of pro-algebraic groups
\begin{equation}\label{eq:Gmot^N-to-Gmot^A}
	\GmotNo(k) = \Spec(\Ocal(\GmotNo(k))) \to \Spec(\mathcal{H}_k^0) =: \GmotAy(k) .
\end{equation}
We have the following comparison result by Choudhury--Gallauer: 

\begin{thm}[{\cite[Thm.~9.1]{CG17}}]\label{thm:ChoudGal}
	The homomorphism~\eqref{eq:Gmot^N-to-Gmot^A} is an isomorphism.
\end{thm}

Following~\cite{CG17}, one can construct the inverse isomorphism
\begin{equation}\label{eq:Gmot^A-to-Gmot^N}
	\GmotAy(k) \to \GmotNo(k)
\end{equation}
explicitly, using the universal property of Nori's abelian category.
Let us explain this using Ivorra--Morel's definition of $\MNori(k)$.
Recall from Section~\ref{subsect:IM_cat} that, since the functor $H^0 \circ \Bti_k^* \colon \DAct(k) \to \vect_{\Q}$ is not monoidal, the description of $\MNori(k)$ as the universal abelian factorization of this functor is not well-suited to constructing the monoidal structure.
Instead, following~\cite[Thm.~1.12]{TerenziNori}, one has to regard $\MNori(k)$ as the universal abelian factorization of the monoidal functor $\Bti_k^* \colon \DAct(k)^0 \to \vect_{\Q}$.
The latter fits into the diagram of monoidal functors
\begin{equation*}
	\begin{tikzcd}
		\DAct(k)^0 \arrow{rr} \arrow{drr}[']{\Bti_k^*} && \Rep_{\Q}(\GmotAy(k)) \arrow{d}{\for_{\GmotAy(k)}} \\
		&& \vect_{\Q},
	\end{tikzcd}
\end{equation*}
which commutes up to monoidal natural isomorphism.
Applying the universal property of $\MNori(k)$, we obtain a canonical faithful exact monoidal functor
\begin{equation}\label{eq:MNori-to-Rep(Gmot^A)}
	\MNori(k) \to \Rep_{\Q}(\GmotAy(k))
\end{equation} 
making the whole diagram
\begin{equation*}
	\begin{tikzcd}
		\DAct(k)^0 \arrow{d}[']{h_k} \arrow{drr} \\
		\MNori(k) \arrow{rr} \arrow{drr}[']{\iota_k} && \Rep_{\Q}(\GmotAy(k)) \arrow{d}{\rom{for}_{\GmotAy(k)}} \\
		&& \vect_{\Q}
	\end{tikzcd}
\end{equation*}
commute up to monoidal natural isomorphism.
It is in this way that one gets the homomorphism \eqref{eq:Gmot^A-to-Gmot^N}.
The fact that this homomorphism is an isomorphism (by the proof of Theorem~\ref{thm:ChoudGal}) implies that the functor \eqref{eq:MNori-to-Rep(Gmot^A)} is an equivalence. 

We are now ready to address the comparison problem over fields:

\begin{lem}\label{lem:equiv-point}
	For every finite extension $k'/k$, the triangulated functor
	\begin{equation*}
		o_{\Spec(k')} \colon \Db(\MNori(\Spec(k'))) \rightarrow \Db(\PervAyoub(\Spec(k'))) = \Dbgeo(\Spec(k'))^{\GmotAy(k)}
	\end{equation*}
	is an equivalence.
\end{lem}

\begin{proof}
	We prove the equivalent statement that, for every finite extension $k'/k$, the exact functor
	\begin{equation*}
		o_{\Spec(k')} \colon \MNori(\Spec(k')) \rightarrow \PervAyoub(\Spec(k'))
	\end{equation*}
	is an equivalence.
	Let us divide the argument into three steps.
	
	First, assume that $k' = k$.
	In this case, in view of the above discussion, it suffices to make sure that the functor $o_k \colon \MNori(k) \to \PervAyoub(k)$ coincides with the functor \eqref{eq:MNori-to-Rep(Gmot^A)} modulo the canonical equivalence of Example~\ref{ex:trivializ-Gmot}.
	To this end, since both functors are obtained by applying the universal property of $\MNori(k)$, it suffices to observe that the enriched Betti realization
	\begin{equation*}
		\widetilde{\Bti}_k^* \colon \DAct(k) \to \Dbgeo(\Spec(k))^{\GmotAy(k)}
	\end{equation*}
	is nothing but the canonical monoidal functor $\DAct(X) \to \mathsf{comod}_{\mathcal{H}_k^0}(\Db(\Q))$ defined by the motivic Hopf algebra.
	
	Next, assume that $k'/k$ is a finite Galois extension.
	Choose a complex embedding $\sigma' \colon k' \hookrightarrow \C$ extending $\sigma$.
	We have a commutative diagram of the form
\begin{equation*}
\begin{tikzcd}
\MNori_{\sigma}(\Spec(k'))
	\arrow{rr}{o_{\Spec(k')}}
	\ar[d, "\sim" dsiml]
&&
\PervAyoub_{\sigma}(\Spec(k'))
	\ar[d, "\sim" dsimr]
\\
\MNori_{\sigma'}(\Spec(k'))
	\arrow{rr}{o_{k'}}
&&
\PervAyoub_{\sigma'}(\Spec(k')),
\end{tikzcd}
\end{equation*}
	where the vertical arrows witness the equivalences of
	Lemma~\ref{lem:MNori_X'/X}(1) and Lemma~\ref{lem:PervvAyoub_X'/X}(1).
	Since the lower horizontal arrow is already known to be an equivalence (by the previous step, applied to $k'$ rather than to $k$), the upper horizontal arrow must be an equivalence as well.
	
	Lastly, consider an arbitrary finite extension $k'/k$.
	Let $k''$ denote the Galois closure of $k'/k$ inside $\bar{k}$.
	The functor $o_{\Spec(k'')}$ is $\Gal(k''/k)$-equivariant with respect to the $\Gal(k''/k)$-action on $\Spec(k'')$ by $k$-automorphisms (by Proposition~\ref{prop:o_S-6ff}).
	We have a commutative diagram of the form
	\begin{equation*}
		\begin{tikzcd}
			\MNori(\Spec(k')) \arrow{rr}{o_{\Spec(k')}} \arrow{d}[dsiml]{\sim} && \PervAyoub(\Spec(k')) \arrow{d}[dsimr]{\sim} \\
			\MNori(\Spec(k''))^{\Gal(k''/k')} \arrow{rr}{o_{\Spec(k'')}} && \PervAyoub(\Spec(k''))^{\Gal(k''/k')},
		\end{tikzcd}
	\end{equation*} 
	where the vertical equivalences are obtained by applying the obvious variants of Lemma~\ref{lem:MNori_X'/X}(2) and Lemma~\ref{lem::AMLoc_X'/X}(2) to the finite Galois extension $k''/k'$ (but still with respect to $\sigma$-analytification of $k$-varieties).
	Since the lower horizontal arrow is already known to be an equivalence (by the previous step), the upper horizontal arrow must be an equivalence as well.
	This concludes the proof.
\end{proof}

\subsection{Comparison over function fields}

In this technical subsection, as an intermediate step towards the full comparison theorem, we establish the following generic variant:

\begin{prop}\label{prop:comp_generic}
	For every (quasi-projective) $k$-variety $X$, the functor
	\begin{equation*}
		\twocolim_{U \in \Open_X^\rom{op}} \Db(\MNori(U)) \to \twocolim_{U \in \Open_X^\rom{op}} \Db(\PervAyoub(U))
	\end{equation*}
	is an equivalence.
\end{prop}

We stated this result is terms of triangulated categories in order to make its relation to Theorem~\ref{thm:comp} transparent.
However, its proof takes place at the level of abelian categories.
More precisely, we want to compare the categories of motivic local systems at the generic points of $k$-varieties, using the fundamental sequences of Theorem~\ref{thm:fund_ses_general} and Theorem~\ref{thm:ayoub_fund_seq}.

We start by studying local systems of geometric origin at the generic points.
Recall that, for every $k$-variety $X$, we have an inclusion $\Pervgeo^{\rm{N}}(X) \subset \Pervgeo^{\rm{A}}(X)$ (see \S~\ref{subsect:IM_cat}), which refines to an inclusion $\Locgeo^{\rm{N}}(X) \subset \Locgeo^{\rm{A}}(X)$ if $X$ is smooth (see \S~\ref{subsect:fundses_MNori}).
Passing to the colimit over the (smooth) dense open subsets of $X$, we get the commutative diagram
\begin{equation*}
\begin{tikzcd}
\twocolim_{U \in \Open_X^\rom{op}} \Locgeo^{\rm{N}}(U)
	\arrow{r}
	\arrow{d}[dsiml]{\sim}
&
\twocolim_{U \in \Open_X^\rom{op}} \Locgeo^{\rm{A}}(U)
	\arrow{d}[dsimr]{\sim}
	\arrow{r}
&
\twocolim_{U \in \Open_X^\rom{op}} \Loc(U)
	\arrow{d}[dsimr]{\sim}
\\
\twocolim_{U \in \Open_X^\rom{op}} \Pervgeo^{\rm{N}}(X)
	\arrow{r}
&
\twocolim_{U \in \Open_X^\rom{op}} \Pervgeo^{\rm{A}}(U)
	\arrow{r}
&
\twocolim_{U \in \Open_X^\rom{op}} \Perv(U)
,
\end{tikzcd}
\end{equation*}
where all horizontal arrows are fully faithful exact functors (as filtered colimits of such).
Moreover, all vertical arrows are equivalences,
as every perverse sheaf
restricts to a local system
over some smooth dense open subset.

\begin{lem}\label{lem:comp_Locgeo_generic}
	For every (quasi-projective) $k$-variety $X$, the two functors
	\begin{equation*}
		\twocolim_{U \in \Open_X^\rom{op}} \Locgeo^{\rm{N}}(U) \to \twocolim_{U \in \Open_X^\rom{op}} \Locgeo^{\rm{A}}(U), \qquad \twocolim_{U \in \Open_X^\rom{op}} \Pervgeo^{\rm{N}}(U) \to \twocolim_{U \in \Open_X^\rom{op}} \Pervgeo^{\rm{A}}(U)
	\end{equation*}
	are equivalences.
\end{lem}
\begin{proof}
	As a consequence of the above discussion,
	the left-most functor is an equivalence
	if and only if
	the right-most functor is.
	Since fully faithfulness is already known,
	it remains to show essential surjectivity.
	
	For every $U \in \Open_X$, the subcategory $\Pervgeo^{\rm{A}}(U) \subset \Perv(U)$ coincides with the smallest abelian subcategory containing $\Pervgeo^{\rm{N}}(U)$ and stable subquotients and extensions (by Corollary~\ref{cor:PervgeoAyoub_pH^0Bti}).
	This implies that $\twocolim_{U \in \Open_X^\rom{op}} \Pervgeo^{\rm{A}}(U)$ is generated by its abelian subcategory $\twocolim_{U \in \Open_X^\rom{op}} \Pervgeo^{\rm{N}}(U)$ under subquotients and extensions inside $\twocolim_{U \in \Open_X^\rom{op}} \Perv(U)$.
	Therefore, it suffices to check that $\twocolim_{U \in \Open_X^\rom{op}} \Pervgeo^{\rm{N}}(U)$ is already stable under subquotients and extensions inside $\twocolim_{U \in \Open_X^\rom{op}} \Perv(U)$.
	This is equivalent to checking that $\twocolim_{U \in \Open_X^\rom{op}} \Locgeo^{\rm{N}}(X)$ is stable under subquotients and extensions inside $\twocolim_{U \in \Open_X^\rom{op}} \Loc(U)$.
	And this follows formally from the fact that each abelian subcategory $\Locgeo^{\rm{N}}(U) \subset \Loc(U)$ is already stable under subquotients (by definition) and under extensions (by Proposition~\ref{prop:LocgeoNori_stable_ext}).
\end{proof}

Let us now focus on the behaviour of motivic local systems at the generic points.

\begin{lem}\label{lem:j^*_fullfaith_locsys}
	For every dense open immersion $j \colon U \hookrightarrow X$ between smooth (quasi-projective) $k$-varieties, the inverse image functors
	\begin{equation*}
		j^* \colon \NMLoc(X) \to \NMLoc(U), \quad j^* \colon \AMLoc(X) \to \AMLoc(U), \quad j^* \colon \Loc(X) \to \Loc(U)
	\end{equation*}
	are fully faithful, with essential image stable under subquotients. 
\end{lem}

\begin{proof}
	For sake of simplicity, we write down the proof in the setting of local systems;
	as the reader can easily check, the argument carries over to the motivic and equivariant settings.
	
	Let us start by proving fully faithfulness.
	For this, we are allowed to replace $U$ by some smaller dense open subset $U'$.
	Hence, up to replacing $U$ by some dense affine open subset, we may assume $U$ affine, in which case $j$ is an affine open immersion.
	This allows us to use the adjunctions
	\begin{equation*}
		j_! \colon \Perv(U) \leftrightarrows \Perv(X) \colon j^*, \qquad j^* \colon \Perv(X) \leftrightarrows \Perv(U) \colon j_*
	\end{equation*}
	as well as the intermediate extension functor
	\begin{equation*}
		j_{!*} \colon \Perv(U) \to \Perv(X), \quad K \mapsto \textup{im}\left\{j_! K \to j_* K\right\}.
	\end{equation*} 
	For every $L \in \Loc(X)$, the counit morphism $\epsilon \colon j_! j^* L \to L$ (resp. the unit morphism $\eta \colon L \to j_* j^* L$) is an epimorphism (resp. a monomorphism) in $\Perv(X)$.
	In particular, $L$ is canonically isomorphic to $j_{!*}j^* L$ (see~\cite[Lem.~4.3.2]{BBDG}, and observe that the irreducibility assumption in the statement plays no role in the proof). 
	Therefore, the fully faithfulness of $j^*$ on $\Loc(X)$ follows formally from the fully faithfulness of $j_{!*}$.
	To check the latter property, fix two objects $K_1, K_2 \in \Perv(U)$, and consider the chain of maps
	\begin{equation*}
		\Hom_{\Perv(U)}(K_1,K_2) \to \Hom_{\Perv(X)}(j_{!*} K_1, j_{!*} K_2) \to \Hom_{\Perv(X)}(j_{!*} K_1, j_* K_2) \to \Hom_{\Perv(X)}(j_! K_1, j_* K_2).
	\end{equation*}
	Note that the second (resp. third) map is injective, since the morphism $j_{!*} K_2 \to j_* K_2$ (resp. $j_! K_1 \to j_{!*} K_1$) is a monomorphism (resp. an epimorphism).
	Since the composite of the chain is a bijection, with inverse given by the composite map
	\begin{equation*}
		\Hom_{\Perv(X)}(j_! K_1,j_* K_2) = \Hom_{\Perv(U)}(K_1, j^* j_* K_2) \xrightarrow{\sim} \Hom_{\Perv(U)}(K_1,K_2),
	\end{equation*}
	we conclude that all maps in the chain are in fact bijections.	
	
	Let us now prove stability under subquotients.
	Since $j^*$ is exact, it suffices to show that it is bijective on subobject lattices of local systems.
	For this, we are allowed to replace $U$ by some smaller dense open subset.
	Thus, we may again assume $U$ affine, in which case $j$ is an affine open immersion.
	In this situation, we claim that $j_{!*}$ is a two-sided inverse of $j^*$ on subobject lattices.
	Since $j^*$ is right-inverse to $j_{!*}$ as a functor of perverse sheaves (being right-inverse to both $j_!$ and $j_*$), it is also its right-inverse on subobject lattices.
	To see that $j_{!*}$ is right-inverse to $j^*$ on subobject lattices of local systems, it suffices to observe that $j_{!*}$ preserves monomorphisms (being a subfunctor of the exact functor $j_*$):
	a given monomorphism $L' \hookrightarrow L$ in $\Loc(X)$ gets identified with the monomorphism $j_{!*} j^* L' \hookrightarrow j_{!*} j^* L$ (using again~\cite[Lem.~4.3.2]{BBDG}).
\end{proof}

\begin{constr}\label{constr:limit_fund_ses}
	Let $X$ be a smooth, geometrically connected (quasi-projective) $k$-variety with function field $k(X)$.
	Following \cite[\S\S~4.4-4.5]{BT25}, one can define a family of fibre functors for the Tannakian categories $\Loc(U)$ which are compatible with restriction along open immersions in $\Open_X$;
	the usual fibre functor at a point $x \in X(\bar{k})$ does not satisfy this property, since no such point belongs to $U^{\sigma}$ for each $U \in \Open_X$.
	
	In the terminology of~\cite[Defn.~4.5, Lem.~4.6]{BT25}, one can define the sought-after fibre functor as restriction of local systems to a \textit{stable arc point} $\xi$ of $X$: 
	an equivalence class of continuous paths on $X^{\sigma}$, defined on the open interval $(0,\delta)$ for some $\delta > 0$, satisfying suitable conditions ensuring that the path has finite intersection with any closed subset of the form $Z^{\sigma}$ with $Z$ a proper Zariski closed subset of $X$.
	Thus, a stable arc point of $X$ restricts to a stable arc point of each $U \in \Open_X$.
	
	From now on, fix such a stable arc point $\xi$;
	let $\pi_1^{\rm{N}}(X,\xi)$ (resp. $\pi_1^{\rm{A}}(X,\xi)$) denote the Tannaka dual of $\Locgeo^{\rm{N}}(X)$ (resp. of $\Locgeo^{\rm{A}}(X)$) with respect to the fibre functor at $\xi$, and let $\GmotNo(X,\xi)$ (resp. $\GmotAy(X,\xi)$) denote the Tannaka dual of $\NMLoc(X)$ (resp. of $\AMLoc(X)$) with respect to the induced fibre functor.
	Note that the sequences of group homomorphisms
	\begin{equation*}
		1 \to \pi_1^{\rm{N}}(X,\xi) \to \GmotNo(X,\xi) \to \GmotNo(k) \to 1
	\end{equation*}
	and
	\begin{equation*}
		1 \to \pi_1^{\rm{A}}(X,\xi) \to \GmotAy(X,\xi) \to \GmotAy(k) \to 1
	\end{equation*}
	are exact:
	after fixing a path between $\xi$ and the constant arc point at some algebraic point $x \in X(\bar{k})$ (which always exists, by~\cite[Lem.~4.11]{BT25}), this reduces to Theorem~\ref{thm:fund_ses_general} and Theorem~\ref{thm:ayoub_fund_seq}, respectively.
	The same holds when $X$ is replaced by some $U \in \Open_X$.
	In conclusion, we obtain a cofiltered poset of short exact sequences indexed by $\Open_X$.
	
	As $U$ varies in $\Open_X$, the categories $\Loc(U)$ form a filtered family where all transition functors are fully faithful with subobject stable image (by Lemma~\ref{lem:j^*_fullfaith_locsys}), hence the same holds for the subcategories $\Locgeo^{\rm{N}}(U)$ and $\Locgeo^{\rm{A}}(U)$.
	Their Tannaka duals with respect to the fibre functors at $\xi$ thus form a cofiltered inverse system of pro-algebraic groups with surjective transition maps (by~\cite[Prop.~2.21(a)]{deligne-milne}).
	The Tannaka duals of $\twocolim_{U \in \Open_X^\rom{op}} \Locgeo^{\rm{N}}(U)$ and $\twocolim_{U \in \Open_X^\rom{op}} \Locgeo^{\rm{A}}(U)$ are given by the inverse limits
	\begin{equation*}
		\pi_1^{\rm{N}}(k(X),\xi) := \varprojlim_{U \in \Open_X} \pi_1^{\rm{N}}(U,\xi), \qquad \pi_1^{\rm{A}}(k(X),\xi) := \varprojlim_{U \in \Open_X} \pi_1^{\rm{A}}(U,\xi),
	\end{equation*}
	respectively.
	The same discussion applies to the categories $\NMLoc(U)$ and $\AMLoc(U)$ (again by Lemma~\ref{lem:j^*_fullfaith_locsys}), so the Tannaka duals of $\twocolim_{U \in \Open_X^\rom{op}} \NMLoc(U)$ and $\twocolim_{U \in \Open_X^\rom{op}} \AMLoc(U)$ are given by the inverse limits
	\begin{equation*}
		\GmotNo(k(X),\xi) := \varprojlim_{U \in \Open_X} \GmotNo(U,\xi), \qquad \GmotAy(k(X),\xi) := \varprojlim_{U \in \Open_X} \GmotAy(U,\xi),
	\end{equation*}
	respectively.
	Finally, the two limit fundamental sequences
	\begin{equation*}
		1 \to \pi_1^{\rm{N}}(k(X),\xi) \to \GmotNo(k(X),\xi) \to \GmotNo(k) \to 1
	\end{equation*}
	and
	\begin{equation*}
		1 \to \pi_1^{\rm{A}}(k(X),\xi) \to \GmotAy(k(X),\xi) \to \GmotNo(k) \to 1
	\end{equation*}
	are still exact:
	using the surjectivity of the transition homomorphisms, it suffices to show the analogous property for inverse limits of abstract groups, which is easy to check.
	\qed
\end{constr}

\begin{proof}[Proof of Proposition~\ref{prop:comp_generic}]
	Since the formation of bounded derived categories is compatible with filtered $2$-colimits of abelian categories (see~\cite[Lem.~2.6]{Gal21}), it suffices to show that the exact functor
	\begin{equation*}
		\twocolim_{U \in \Open_X^\rom{op}} \MNori(U) \to \twocolim_{U \in \Open_X^\rom{op}} \PervAyoub(U)
	\end{equation*}
	is an equivalence.
	Up to replacing the $k$-variety $X$ by some smooth dense open subset, we may assume $X$ smooth.
	In this case, we have the commutative diagram
	\begin{equation*}
		\begin{tikzcd}
			\twocolim_{U \in \Open_X^\rom{op}} \NMLoc(U) \arrow{rr} \arrow{d}[dsiml]{\sim} && \twocolim_{U \in \Open_X^\rom{op}} \AMLoc(U) \arrow{d}[dsimr]{\sim} \\
			\twocolim_{U \in \Open_X^\rom{op}} \MNori(X) \arrow{rr} && \twocolim_{U \in \Open_X^\rom{op}} \PervAyoub(X),
		\end{tikzcd}
	\end{equation*}
	where the two vertical arrows are both equivalences (by cofinality).
	Hence, it is equivalent to show that the exact functor
	\begin{equation}\label{eq:o_X-generic_locsys}
		\twocolim_{U \in \Open_X^\rom{op}} \NMLoc(U) \to \twocolim_{U \in \Open_X^\rom{op}} \AMLoc(U)
	\end{equation}
	is an equivalence.
	To prove the latter claim, we proceed as follows.
	
	Of course, we may assume the smooth $k$-variety $X$ to be non-empty.
	Up to considering each connected component of $X$ separately, we may even assume $X$ connected.
	
	We further reduce to the case when $X$ is geometrically connected over $k$.
	To this end, let $k'$ denote the algebraic closure of $k$ inside the function field $k(X)$.
	By construction, the connected components of the smooth $k'$-variety $X' := X \times_k k'$ are all geometrically connected over $k'$;
	similarly for each $U \in \Open_X$.
	Let $k''$ denote the Galois closure of $k'$ inside $\bar{k}$, and consider the smooth $k''$-variety $X'' := X \times_k k''$.
	We have a commutative diagram of the form
	\begin{equation*}
		\begin{tikzcd}
			\NMLoc(X) \arrow{rr}{\sim} \arrow{d}[']{o_X} && \NMLoc(X'')^{\Gal(k''/k)} \arrow{d}{o_{X''}} \\
			\AMLoc(X)\arrow{rr}{\sim} && \AMLoc(X'')^{\Gal(k''/k)},
		\end{tikzcd}
	\end{equation*} 
	where the horizontal arrows witness the equivalences of Lemma~\ref{lem:NMLoc_X'_X}(2) and Lemma~\ref{lem::AMLoc_X'/X}(2), respectively; 
	similarly for each $U \in \Open_X$.
	Thus, the claim holds for $X$ as soon as it holds for the $k$-variety $X''$.
	In fact, after choosing a complex embedding $\sigma'' \colon k'' \hookrightarrow \C$ extending $\sigma$, we are even allowed to regard $X''$ as a $k''$-variety (by Lemma~\ref{lem:NMLoc_X'_X}(1) and Lemma~\ref{lem::AMLoc_X'/X}(1)).
	But each connected component of $X''$ is geometrically connected over $k''$ (since $X'$ enjoys the analogous property over $k'$).
	Therefore, up to replacing the $k$-variety $X$ by the $k''$-variety $X''$ and considering each connected component of the latter separately, we may assume $X$ smooth and geometrically connected over $k$, say of dimension $d$.
	
	In this situation, the functor~\eqref{eq:o_X-generic_locsys} is a monoidal exact functor between neutral Tannakian categories.
	Hence, in order to show that it is an equivalence, it suffices to show that the corresponding homomorphism between Tannaka dual groups is an isomorphism. 
	To this end, consider the diagram of monoidal functors
	\begin{equation*}
		\begin{tikzcd}
			\Locgeo^{\rm{N}}(X) \arrow{d} 
			&& \NMLoc(X) \arrow{ll}[']{\iota_X} \arrow{d}{o_X} 
			&& \MNori(k) \arrow{ll}[']{a_X^*[d]} \arrow{d}{o_k} \\
			\Locgeo^{\rm{A}}(X) 
			&& \AMLoc(X) \arrow{ll}{\omega_X} 
			&& \PervAyoub(k), \arrow{ll}{a_X^*[d]}
		\end{tikzcd}
	\end{equation*}
	where $a_X \colon X \to \Spec(k)$ denotes the structural morphism.
	Both squares are commutative up to monoidal natural isomorphism (the left-most one by construction, the right-most one by Proposition~\ref{prop:o_S-6ff}), and the right-most vertical arrow is already known to be an equivalence (by Lemma~\ref{lem:equiv-point}).
	Taking the Tannaka dual groups with respect to the fibre functors at some chosen stable arc point $\xi$ of $X$ (see Construction~\ref{constr:limit_fund_ses}), we obtain the commutative diagram of pro-algebraic groups
	\begin{equation*}
		\begin{tikzcd}
			1 \arrow{r} & \pi_1^{\rm{N}}(X,\xi) \arrow{r} & \GmotNo(X,\xi) \arrow{r} & \GmotNo(k) \arrow{r} & 1 \\
			1 \arrow{r} & \pi_1^{\rm{A}}(X,\xi) \arrow{r} \arrow{u} & \GmotAy(X,\xi) \arrow{r} \arrow{u} & \GmotAy(k) \arrow{r} \arrow{u}[dsimr]{\sim} & 1,
		\end{tikzcd}
	\end{equation*}
	and similarly for each $U \in \Open_X$.
	Passing to the inverse limit, we obtain the commutative diagram
	\begin{equation*}
		\begin{tikzcd}
			1 \arrow{r} & \pi_1^{\rm{N}}(k(X),\xi) \arrow{r} & \GmotNo(k(X),\xi) \arrow{r} & \GmotNo(k) \arrow{r} & 1 \\
			1 \arrow{r} & \pi_1^{\rm{A}}(k(X),\xi) \arrow{r} \arrow{u}& \GmotAy(k(X),\xi) \arrow{r} \arrow{u} & \GmotAy(k) \arrow{r} \arrow{u}[dsimr]{\sim} & 1,
		\end{tikzcd}
	\end{equation*}
	where both rows are exact (as explained at the end of Construction~\ref{constr:limit_fund_ses}).
	This corresponds to the colimit diagram of Tannakian categories
	\begin{equation*}
		\begin{tikzcd}
			\twocolim_{U \in \Open_X^\rom{op}} \Locgeo^{\rm{N}}(U) \arrow{d} 
			&& \twocolim_{U \in \Open_X^\rom{op}} \NMLoc(U) \arrow{ll} \arrow{d}
			&& \MNori(k) \arrow{ll} \arrow{d}{o_k} \\
			\twocolim_{U \in \Open_X^\rom{op}} \Locgeo^{\rm{A}}(U) 
			&& \twocolim_{U \in \Open_X^\rom{op}} \AMLoc(U) \arrow{ll} 
			&& \PervAyoub(k), \arrow{ll}
		\end{tikzcd}
	\end{equation*}
	where the left-most vertical arrow is already known to be an equivalence (by Lemma~\ref{lem:comp_Locgeo_generic}).
	Hence, the left-most vertical homomorphism in the previous diagram is an isomorphism. 
	But then the middle vertical homomorphism in the same diagram must be an isomorphism as well.
\end{proof}

\subsection{Proof of the comparison theorem}

The final step in the proof of Theorem~\ref{thm:comp} is to promote the generic equivalence of Proposition~\ref{prop:comp_generic} to an integral result.
Before proving the full comparison theorem, let us settle the case of motivic local systems:

\begin{thm}\label{thm:comp_motlocsys}
	For every smooth (quasi-projective) $k$-variety $X$, the exact functor
	\begin{equation}\label{eq:comp_motlocsys}
		o_X \colon \NMLoc(X) \to \AMLoc(X)
	\end{equation}
	is an equivalence.
\end{thm}

\begin{proof}
	Arguing as in the proof of Proposition~\ref{prop:comp_generic}, we readily reduce to the case when $X$ is geometrically connected over $k$.
	
	In this situation, fix a closed point $x \in X(\bar{k})$, and consider the associated commutative diagram with exact rows
	\begin{equation*}
		\begin{tikzcd}
			1 \arrow{r} & \pi_1^{\rm{N}}(X,x) \arrow{r} & \GmotNo(X,x) \arrow{r} & \GmotNo(k) \arrow{r} & 1 \\
			1 \arrow{r} & \pi_1^{\rm{A}}(X,x) \arrow{r} \arrow{u} & \GmotAy(X,x) \arrow{r} \arrow{u} & \GmotAy(k) \arrow{r} \arrow{u}[dsimr]{\sim} & 1,
		\end{tikzcd}
	\end{equation*}
	obtained as in the proof of Proposition~\ref{prop:comp_generic}.
	Since $\Locgeo^{\rm{N}}(X)$ is a Tannakian subcategory of $\Locgeo^{\rm{A}}(X)$ (by construction), the left-most vertical arrow is an epimorphism.
	By diagram-chasing, we deduce that the middle vertical arrow is an epimorphism as well:
	this means that $o_X$ identifies $\NMLoc(X)$ with a Tannakian subcategory of $\AMLoc(X)$, in particular it is fully faithful (see~\cite[Prop.~2.21(a)]{deligne-milne}).
	In order to conclude, it remains to show that it is essentially surjective.
	
	To this end, fix an object $L \in \AMLoc(X)$.
	As a consequence of Proposition~\ref{prop:comp_generic}, there exists a dense open immersion $j \colon U \hookrightarrow X$ such that the object $j^* L \in \AMLoc(U)$ is isomorphic to $o_U(M)$ for some $M \in \NMLoc(U)$;
	without loss of generality we may assume $U$ affine, in which case $j$ is an affine morphism.
	In this situation, we have the chain of isomorphisms in $\PervAyoub(X)$
	\begin{align*}
		L &= j_{!*} j^* L && \textup{(by the proof of Lemma \ref{lem:j^*_fullfaith_locsys})}\\
		&\simeq j_{!*} o_U(M) \\
		&= o_X(j_{!*} M), && \textup{(by Proposition \ref{prop:o_S-6ff})}
	\end{align*}
	which shows that $L$ belongs to the essential image of $o_X$.
	This concludes the proof.
\end{proof}

Thus, there is no longer any ambiguity about the notion of motivic local system over $X$.
Consequently, the same happens for the notion of local system of geometric origin:

\begin{cor}\label{cor:LocNori=LocAyoub}
	For every smooth (quasi-projective) $k$-variety $X$, the inclusion
	\begin{equation*}
		\Locgeo^{\rm{N}}(X) \subset \Locgeo^{\rm{A}}(X)
	\end{equation*}
	is in fact an equivalence.
\end{cor}

\begin{proof}
	Indeed, as a consequence of Theorem~\ref{thm:comp_motlocsys}, both categories in question coincide with the smallest abelian subcategory of $\Loc(X)$ containing all local systems underlying motivic local systems and stable under subquotients.
\end{proof}

\begin{rem}
	In particular, this implies that $\Locgeo^{\rm{N}}(X)$ equals the intersection $\Pervgeo^{\rm{N}}(X) \cap \Loc(X)$:
	if a local system on $X$ can be written as a subquotient of the realization of some motivic perverse sheaf, it can in fact be written as a subquotient of the realization of some motivic local system.
\end{rem}

We are ready to complete the proof of our main result.
Throughout, we systematically use the compatibility of our comparison functors with the six functor formalism (as stated in Proposition~\ref{prop:o_S-6ff}).
Our argument is formally the same as the one used by Ayoub in the proof of~\cite[Thm.~1.93]{AyoAnab}.

\begin{proof}[Proof of Theorem~\ref{thm:comp}]
	By Lemma~\ref{lem:equiv-point}, the triangulated comparison functor $o_X$ is an equivalence in the case when $X = \Spec(k')$ for some finite extension $k'/k$.
	In order to prove the same over a general (quasi-projective) $k$-variety $X$, we establish fully faithfulness and then essential surjectivity.
	
	Firstly, we show that $o_X$ is fully faithful for every $X$.
	To this end, fix two objects $M_1^{\bullet}, M_2^{\bullet} \in \Db(\MNori(X))$.
	Writing $a_X \colon X \to \Spec(k)$ for the structural morphism, we have the usual identification
	\begin{equation*}
		\begin{aligned}
			\Hom_{\Db(\MNori(X))}(M_1^{\bullet},M_2^{\bullet}) = & \Hom_{\Db(\MNori(X))}(\Q_X,\Homint_X(M_1^{\bullet},M_2^{\bullet})) = \\
			& \Hom_{\Db(\MNori(X))}(a_X^* \Q_k,\Homint_X(M_1^{\bullet},M_2^{\bullet})) = \\
			& \Hom_{\Db(\MNori(k))}(\Q_k,a_{X,*} \Homint_X(M_1^{\bullet},M_2^{\bullet})).
		\end{aligned}
	\end{equation*}
	Similarly, we have
	\begin{equation*}
		\begin{aligned}
			\Hom_{\Db(\PervAyoub(X))}(o_X M_1^{\bullet},o_X M_2^{\bullet}) = & \Hom_{\Db(\PervAyoub(X))}(\Q_X,\Homint_S(o_X M_1^{\bullet},o_X M_2^{\bullet})) = \\
			& \Hom_{\Db(\PervAyoub(X))}(a_X^* \Q_k,\Homint_X(o_X M_1^{\bullet},o_X M_2^{\bullet})) = \\
			& \Hom_{\Db(\PervAyoub(k))}(\Q_k,a_{X,*} \Homint_X(o_X M_1^{\bullet},o_X M_2^{\bullet})) = \\
			& \Hom_{\Db(\PervAyoub(k))}(\Q_k,a_{X,*} o_X \Homint_X(M_1^{\bullet},M_2^{\bullet})) = \\
			& \Hom_{\Db(\PervAyoub(k))}(o_k \Q_k,o_k a_{X,*} \Homint_X(M_1^{\bullet},M_2^{\bullet})),
		\end{aligned}
	\end{equation*}
	where the last two passages witness the compatibility of the comparison functors with the six operations (given by Proposition~\ref{prop:o_S-6ff}).
	Under these identifications, the map
	\begin{equation*}
		\Hom_{\Db(\MNori(X))}(M_1^{\bullet},M_2^{\bullet}) \rightarrow \Hom_{\Db(\PervAyoub(X))}(o_X M_1^{\bullet},o_X M_2^{\bullet})
	\end{equation*}
	defined by $o_X$ corresponds to the map
	\begin{equation*}
		\Hom_{\Db(\MNori(k))}(\Q_k, a_{X,*} \Homint_X(M_1^{\bullet},M_2^{\bullet})) \rightarrow \Hom_{\Db(\PervAyoub(k))}(o_k \Q_k ,o_k a_{X,*} \Homint_X(M_1^{\bullet},M_2^{\bullet}))
	\end{equation*}
	defined by $o_k$, which is already known to be a bijection (by Lemma~\ref{lem:equiv-point}).
	This shows that $o_X$ is fully faithful, as wanted.
	
	Secondly, we show that $o_X$ is essentially surjective for every $X$.
	We proceed by Noetherian induction on $X$.
	The base step is when $\dim(X) = 0$, in which case $X$ is a disjoint union of spectra of finite extensions of $k$, so that the result is already known (essentially by Lemma~\ref{lem:equiv-point}).
	For the inductive step, assume that $\dim(X) > 0$ and that the result is known to hold over any proper closed subvariety $Z$ of $X$.
	Fix an object $N^{\bullet} \in \Db(\PervAyoub(X))$.
	As a consequence of Proposition~\ref{prop:comp_generic}, there exist a dense open immersion $j \colon U \hookrightarrow X$ and an object $M_U^{\bullet} \in \Db(\MNori(U))$ with an isomorphism in $\Db(\PervAyoub(U))$
	\begin{equation*}
		o_U (M_U^{\bullet}) \xrightarrow{\sim} j^* N^{\bullet}.
	\end{equation*}
	Let $i\colon Z \hookrightarrow X$ denote the complementary closed immersion.
	By inductive hypothesis, there exists an object $M_Z^{\bullet} \in \Db(\MNori(Z))$ with an isomorphism in $\Db(\PervAyoub(Z))$
	\begin{equation*}
		o_Z (M_Z^{\bullet}) \xrightarrow{\sim} i^* N^{\bullet}.
	\end{equation*}
	Now consider the localization triangle in $\Db(\MNori(X))$
	\begin{equation*}
		j_! j^* N^{\bullet} \rightarrow N^{\bullet} \rightarrow i_* i^* N^{\bullet} \xrightarrow{\delta} j_! j^* N^{\bullet}[1].
	\end{equation*}
	Since we have already shown that $o_X$ is fully faithful, there is a unique morphism in $\Db(\MNori(X))$
	\begin{equation}\label{formula:delta-hat}
		\hat{\delta}\colon i_* M_Z^{\bullet} \rightarrow j_! M_U^{\bullet}[1]		
	\end{equation}
	making the diagram in $\Db(\PervAyoub(X))$
	\begin{equation}\label{dia:delta}
		\begin{tikzcd}
			o_X(i_* M_Z^{\bullet}) \arrow{d}{\hat{\delta}} \arrow[equal]{r} & i_* o_Z(M_Z^{\bullet}) \arrow{r}{\sim} & i_* i^* N^{\bullet} \arrow{d}{\delta} \\
			o_X(j_! M_U^{\bullet})[1] \arrow[equal]{r} & j_! o_U(M_U^{\bullet})[1] \arrow{r}{\sim} & j_! j^* N^{\bullet}[1]
		\end{tikzcd}
	\end{equation}
	commute.
	Let us complete the morphism~\eqref{formula:delta-hat} to a distinguished triangle in $\Db(\MNori(X))$
	\begin{equation*}
		j_! M_U^{\bullet} \rightarrow M_X^{\bullet} \rightarrow i_* M_Z^{\bullet} \xrightarrow{\hat{\delta}} j_! M_U^{\bullet}[1].
	\end{equation*}
	It is then possible to complete the commutative diagram \eqref{dia:delta} to a morphism of triangles in $\Db(\PervAyoub(X))$
	\begin{equation*}
		\begin{tikzcd}
			o_X(j_! M_U^{\bullet}) \arrow{r} \arrow[equal]{d} & o_X (M_X^{\bullet}) \arrow{r} \arrow[dashed]{dd} & o_X(i_* M_Z^{\bullet}) \arrow{r}{\hat{\delta}} \arrow[equal]{d} & o_X(j_! M_U^{\bullet})[1] \arrow[equal]{d} \\
			j_! o_U(M_U^{\bullet}) \arrow{d}[dsiml]{\sim} && i_* o_Z(M_Z^{\bullet}) \arrow{d}[dsimr]{\sim} & j_! o_U(M_U^{\bullet})[1] \arrow{d}[dsimr]{\sim} \\
			j_! j^* N^{\bullet} \arrow{r} & N^{\bullet} \arrow{r} & i_* i^* N^{\bullet} \arrow{r}{\delta} & j_! j^* N^{\bullet}[1].
		\end{tikzcd}
	\end{equation*}
	Since any arrow $o_X (M_X^{\bullet}) \rightarrow N^{\bullet}$ filling the latter diagram is automatically an isomorphism (because so are the other two vertical arrows), we see that $N^{\bullet}$ lies in the essential image of $o_X$.
	This shows that $o_X$ is essentially surjective, thereby completing the proof.
\end{proof}

To conclude the present section,
we get rid of the last ambiguity
regarding perverse sheaves of geometric origin.

\begin{prop}\label{prop:PervNori=PervAyoub}
	For every (quasi-projective) $k$-variety $X$, the inclusion
	\begin{equation*}
		\Pervgeo^{\rm{N}}(X) \subset \Pervgeo^{\rm{A}}(X)
	\end{equation*}
	is in fact an equivalence.
\end{prop}

\begin{proof}
	We want to use the comparison criterion provided by~\cite[Prop.~1.9]{TerenziNori}.
	Note that, while this criterion is stated for subcategories of perverse Nori motives, both the formulation and the proof rely solely on the functoriality of perverse sheaves. 
	
	We already know that, for every $k$-variety $X$, the inclusions $\Pervgeo^{\rm{N}}(U) \subset \Pervgeo^{\rm{A}}(U)$ as $U$ varies in $\Open_X$ induce an equivalence in the colimit (by Lemma~\ref{lem:comp_Locgeo_generic}).
	Therefore, in order to apply~\cite[Prop.~1.9]{TerenziNori}, it suffices to check that, as $X$ varies, the subcategories $\Pervgeo^{\rm{N}}(X) \subset \Pervgeo^{\rm{A}}(X)$ are stable under 
	inverse images under open immersions, extensions by zero and direct images under affine open immersions, direct images under closed immersions, Tate twists, as well as Beilinson's gluing functors.
	This follows formally from the fact that all these functors are $t$-exact for the perverse $t$-structures and lift to the categories of perverse Nori motives.
\end{proof}

\begin{rem}
	This implies, in particular, that $\Pervgeo^{\rm{N}}(X)$ is already stable under extensions inside $\Perv(X)$:
	any extension between perverse sheaves which can be written as subquotients of the realization of some motivic perverse sheaves can be itself written as such a subquotient.
\end{rem}

\begin{rem}
	The comparison criterion of~\cite[Prop.~1.9]{TerenziNori} also offers an alternative way to deduce Theorem~\ref{thm:comp} from Proposition~\ref{prop:comp_generic}:
	instead of comparing the derived categories via localization triangles, one compares the perverse hearts via Beilinson's gluing functors (which are compatible with the comparison functors, by Proposition~\ref{prop:o_S-6ff}).
\end{rem}

\section{Applications and complements}\label{sect:appli_compl}

In this final section, we establish some interesting consequences of our comparison theorem.
Most notably, we obtain an $\infty$-categorical enhancement for the six functor formalism of Nori motivic sheaves.
This yields a canonical system of realization functors of Voevodsky motivic sheaves into Nori motivic sheaves, which greatly extends Nori's results over the base field.
We deduce a rigidity property of Nori motivic sheaves as a six functor formalism.
In the other direction, our comparison result implies strong independence properties of Ayoub's categories of homotopy-fixed points from the chosen complex embedding $\sigma \colon k \hookrightarrow \C$.

\subsection{A new construction of the Nori realization}\label{subsec:real}

The construction of the six operations on Ivorra--Morel's categories of Nori motivic sheaves is quite involved, for example due to the lack of $t$-exactness of general inverse image functors and of the tensor product.
In particular, it is not clear at first sight whether the resulting monoidal fibered category over $k$-varieties admits an $\infty$-categorical enhancement in the sense of~\cite{DG22}.

This issue was solved by Tubach in~\cite[\S\S~1,2]{Tub23}:
using a motivic version of an argument due to Nori, he showed that the triangulated category $\Db(\MNori(X))$ is equivalent to the bounded derived category of its so-called constructible heart.
This yields the sought-after enhancement, because inverse images and the tensor product respect the ordinary $t$-structure.
Our comparison theorem gives an alternative way to obtain the $\infty$-categorical enhancement:
\begin{thm}\label{thm:infty}
	The six functor formalism $X \mapsto \Db(\MNori(X))$ admits a canonical enhancement to a functor
	\begin{equation*}
		\Db(\MNori(-)) \colon \Var_k^{op} \rightarrow \CAlg(\mathsf{Cat}_{\infty}^{st}), \quad X \mapsto \Db(\MNori(X))
	\end{equation*}
	with values in the $\infty$-category of stably symmetric monoidal $\infty$-categories.
\end{thm}
\begin{proof}
	As recalled in Section~\ref{subsect:Dbgeo^Gmot}, the six functor formalism $X \mapsto \Dbgeo(X)^{\GmotAy(k)}$ comes equipped with a canonical $\infty$-categorical enhancement as in the statement. 
	Therefore, the thesis follows from Theorem~\ref{thm:comp}.
\end{proof}

\begin{rem}
	The $\infty$-categorical enhancement described in Theorem~\ref{thm:infty} coincides with that of~\cite{Tub23}.
	This follows from Ayoub's study of sheaves of geometric origin in~\cite[\S~1.6]{AyoAnab}:
	in particular, he showed in~\cite[Thm.~1.6.32]{AyoAnab} that the stable $\infty$-category $\Dbgeo(X)$ is equivalent to the bounded derived category of its constructible heart, again by adapting Nori's argument.
\end{rem}

The $\infty$-categorical enhancement of the six operations is a significant technical improvement, since it allows one to define coefficient categories coherently over diagrams of $k$-varieties. 
This opens the way to a study of nearby and vanishing cycles within the axiomatic framework of~\cite[\S~3]{AyoThesis}.
Furthermore, as explained in~\cite[Prop.~A.4]{Tub23}, this provides a canonical extension of the six functor formalism $X \mapsto \Db(\MNori(X))$ to non-necessarily quasi-projective $k$-varieties.

In~\cite[Thm.~4.5]{Tub23}, the $\infty$-categorical enhancement is used to extend Nori's Theorem~\ref{thm:Nori_real_k} to motivic sheaves, as an application of Drew--Gallauer's Theorem~\ref{thm:Drew-Gallauer}.
Our comparison theorem yields a different perspective on this important result:
\begin{thm}\label{thm:real}
	For every $k$-variety $X$, there exists a canonical triangulated realization functor
	\begin{equation*}
		\Nri_X^* \colon \DAct(X) \rightarrow \Db(\MNori(X))
	\end{equation*}
	such that the Betti realization over $X$ factors as
	\begin{equation*}
		\Bti_X^* \colon \DAct(X) \xrightarrow{\Nri_X^*} \Db(\MNori(X)) \xrightarrow{\iota_X} \Db(\Perv(X)) = \Dbct(X)
	\end{equation*}
	up to canonical natural isomorphism.
	As $X$ varies, the functors $\Nri_X^*$ are compatible with the six operations, as well as with Beilinson's gluing functors.
\end{thm}
\begin{proof}
	Thanks to Theorem~\ref{thm:comp}, this follows from Ayoub's Corollary~\ref{cor:Bti-enriched}.
\end{proof}

This result shows the power of $\infty$-categorical methods in the study of Nori motives:
it provides all realization functors $\Nri_X^*$ at once, without need to construct them individually.
In fact, at present, no direct construction of $\Nri_X^*$ for a single $k$-variety $X$ with $\dim(X) > 0$ is available:
to the authors' knowledge, the only candidate construction (over smooth $X$) was proposed by Ivorra in~\cite{Ivo16};
however, its compatibility with the Betti realization and with the six operations is only partially understood.

\begin{rem}\label{rem:Nri-real_uniqueness}
	Recall that Ayoub's Corollary~\ref{cor:Bti-enriched}, and thus also our Theorem~\ref{thm:real}, is essentially a consequence of Drew--Gallauer's Theorem~\ref{thm:Drew-Gallauer}. 
	By the uniqueness part in the latter result, the Nori realization is the unique possible morphism of six functor formalisms between Voevodsky and Nori motivic sheaves up to $2$-isomorphism (in the sense of Definition~\ref{defn:mor-6ff}(2) below). 
\end{rem}

The existence of the Nori realization over bases has one important consequence for the conjectural motivic picture.
Recall that, for every $k$-variety $X$, the triangulated category $\DAct(X)$ is expected to carry a motivic perverse $t$-structure, characterized by its compatibility with the perverse $t$-structure on $\Dbct(X)$ under the Betti realization.
As noted in Remark~\ref{rem:DA(X)^heart=MNori(X)}, if this is the case, then the perverse heart of $\DAct(X)$ is exactly $\MNori(X)$.
In fact, as proved by Tubach in~\cite[Thm.~4.22]{Tub23}, the existence of the motivic perverse $t$-structures implies that the functors $\Nri_X^*$ are all equivalences;
the proof uses the universal property of Ivorra--Morel's categories, and again Drew--Gallauer's Theorem~\ref{thm:Drew-Gallauer}.
Combining Tubach's result with our comparison theorem, we deduce an interesting formula for Voevodsky motivic sheaves:
\begin{cor}
	Suppose that, for every $k$-variety $X$, the conjectural motivic perverse $t$-structure on $\DAct(X)$ exists.
	Then there is a canonical system of equivalences
	\begin{equation*}
		\DAct(X) \xrightarrow{\sim} \DAct(X;\clg{B}_X)^{\Gmot(k)}
	\end{equation*}
	compatible with the six operations.
\end{cor}
\begin{proof}
	If the conjectural motivic $t$-structure on $\DAct(k)$ exists, then by~\cite[Prop.~3.2.9]{AyoConj} the spectral motivic Galois group $\Gmot(k)$ is already classical, thus equal to $\GmotAy(k)$.
	Using Corollary~\ref{cor:Dbgeo=DAct(Btialg)}, we get canonical equivalences
	\begin{equation*}
		\DAct(X;\clg{B}_X)^{\Gmot(k)} \xrightarrow{\sim} \Dbgeo(X)^{\GmotAy(k)},
	\end{equation*}
	compatibly with the six operations.
	If the motivic perverse $t$-structure exists on every $k$-variety $X$, then by~\cite[Thm.~4.22]{Tub23} and Theorem~\ref{thm:comp} we get canonical equivalences
	\begin{equation*}
		\DAct(X) \xrightarrow{\sim} \Db(\MNori(X)) \xrightarrow{\sim} \Dbgeo(X)^{\GmotAy(X)},
	\end{equation*}
	again compatibly with the six operations.
	This implies the thesis. 
\end{proof}

We conclude by discussing the compatibility of the Nori realization with weights.
The triangulated categories $\DAct(X)$ are endowed with the \textit{Chow weight structure}, constructed by Bondarko in~\cite{Bon14}.
As shown in~\cite[Cor.~6.27]{IM19}, the derived categories $\Db(\MNori(X))$ carry a canonical weight structure which is \textit{transversal} to the perverse $t$-structure in the sense of~\cite[Defn.~1.2.2]{BondarkoWeight}:
in particular, objects of $\MNori(X)$ possess a canonical weight filtration so that morphisms in $\MNori(X)$ are strictly filtered (see also~\cite[Cor.~6.16, Prop.~6.17]{IM19}).

\begin{prop}\label{prop:Nri-weights}
	For every $k$-variety $X$, the Nori realization functor $\Nri_X^* \colon \DAct(X) \to \Db(\MNori(X))$ is weight-exact.
\end{prop}

This result was proved in~\cite[Prop.~4.9]{Tub23} using Hodge-theoretic methods.
Here, we offer a different proof of arithmetic nature, based on the formalism of mixed $\ell$-adic perverse sheaves studied by Morel in~\cite{Mor18}.

\begin{nota}
	Let $X$ be a $k$-variety.
	\begin{itemize}
		\item Let $\Db_{\ell}(X)$ denote the triangulated category of \textit{constructible $\ell$-adic complexes} on $X$, as defined in~\cite{Ekedahl} or in~\cite{Bhatt-Scholze}.
		\item Let $\Perv_{\ell}(X) \subset \Db_{\ell}(X)$ denote the abelian category of $\ell$-adic perverse sheaves over $X$.
		\item Let $\Perv_{\ell}^\rom{mf}(X) \subset \Perv_{\ell}(X)$ denote the full abelian subcategory of \textit{mixed $\ell$-adic perverse sheaves admitting a weight filtration}, as defined in~\cite[Defn.~2.6.2]{Mor18}.
	\end{itemize}  
\end{nota}

In~\cite{AyoEtaleReal}, Ayoub constructed a system of $\ell$-adic étale realization functors
\begin{equation*}
	\mathsf{R}_{\ell,X} \colon \DAct(X) \rightarrow \Db_{\ell}(X)
\end{equation*}
compatible with the six operations as well as with Beilinson's gluing functors (see \cite[\S~3.2]{IM19} for the definition of the gluing functors on Voevodsky motivic sheaves).
By~\cite[Thm.~1.3]{BeilinsonEquivalence}, the realization functor
\begin{equation*}
	\Db(\Perv_{\ell}(X)) \rightarrow \Db_{\ell}(X)
\end{equation*}
is an equivalence.
By~\cite[Thm.~3.2.4, Thm.~6.2.2, Prop.~8.2, Prop.~8.3]{Mor18}, as $X$ varies, the derived categories $\Db(\Perv_{\ell}^\rom{mf}(X))$ are endowed with the six operations, as well as with Beilinson's gluing functors, and the forgetful functors
\begin{equation*}
	\Db(\Perv_{\ell}^\rom{mf}(X)) \rightarrow \Db(\Perv_{\ell}(X))
\end{equation*}
commute with them.
However, it is not clear a priori whether the six operations on mixed perverse sheaves lift to the $\infty$-categorical level:
this makes it difficult to show that Ayoub's $\ell$-adic realization functor over $X$ factors through $\Db(\Perv_{\ell}^\rom{mf}(X))$.
However,~\cite[Prop.~6.11(2)]{IM19} allows us to identify $\MNori(X)$ with the universal abelian factorization of the homological functor
\begin{equation*}
	\DAct(X) \xrightarrow{\mathsf{R}_{\ell,X}} \Db_{\ell}(X) \xrightarrow{\pH^0} \Perv_{\ell}(X).
\end{equation*}
Consider the associated $\ell$-adic forgetful functor
\begin{equation*}
	\iota_{\ell,X} \colon \MNori(X) \to \Perv_{\ell}(X),
\end{equation*}
and extend the latter to a conservative triangulated functor
\begin{equation*}
	\iota_{\ell,X} \colon \Db(\MNori(X)) \to \Db(\Perv(X)) = \Db_{\ell}(X).
\end{equation*}
The construction of the six operations of~\cite{IM19} and~\cite{TerenziNori} can be rewritten using the $\ell$-adic realization in place of the Betti realization everywhere:
this means that the triangulated $\ell$-adic forgetful functors commute with the six operations and with the gluing functors.
Lastly,~\cite[Prop.~6.17]{IM19} implies that the $\ell$-adic forgetful functor factors through a faithful exact functor
\begin{equation*}
	\iota_{\ell,X}^\rom{mf} \colon \MNori(X) \to \Perv_{\ell}^\rom{mf}(X).
\end{equation*}
Combining this with Theorem~\ref{thm:real}, we obtain the following result:

\begin{prop}
	For every $k$-variety $X$, there exists a canonical triangulated realization functor
	\begin{equation*}
		\mathsf{R}^\rom{mf}_{\ell,X} \colon \DAct(X) \rightarrow \Db(\Perv_{\ell}^\rom{mf}(X))
	\end{equation*}
	such that the $\ell$-adic realization over $X$ factors as
	\begin{equation*}
		\mathsf{R}_{\ell,X} \colon \DAct(X) \xrightarrow{\mathsf{R}^\rom{mf}_{\ell,X}} \Db(\Perv_{\ell}^\rom{mf}(X)) \rightarrow \Db(\Perv_{\ell}(X)) \xrightarrow{\sim} \Db_{\ell}(X) 
	\end{equation*}
	up to canonical natural isomorphism.
	As $X$ varies, the functors $\mathsf{R}^\rom{mf}_{\ell,X}$ are compatible with the six operations, as well as with Beilinson's gluing functors.
\end{prop}

\begin{proof}
	For every $k$-variety $X$, the functor $\mathsf{R}^\rom{mf}_{\ell,X}$ is defined as the composite
	\begin{equation*}
		\mathsf{R}_{\ell,X}^\rom{mf} \colon \DAct(X) \xrightarrow{\Nri_X^*} \Db(\MNori(X)) \xrightarrow{\iota^\rom{mf}_{\ell,X}} \Db(\Perv_{\ell}^\rom{mf}(X)).
	\end{equation*}
	The compatibility with the six operations and with the gluing functors follows from the analogous compatibilities for the Nori realization and for the $\ell$-adic forgetful functors individually.
	Therefore, the fact that the construction recovers Ayoub's $\ell$-adic realization follows from the uniqueness part of Drew--Gallauer's Theorem~\ref{thm:Drew-Gallauer}.
\end{proof}

\begin{proof}[Proof of Proposition~\ref{prop:Nri-weights}]
	Since weights in $\MNori(X)$ are uniquely determined by their compatibility with weights in $\Perv_{\ell}^\rom{mf}(X)$ (by definition), it suffices to show that the refined $\ell$-adic realization $\mathsf{R}^\rom{mf}_{\ell,X}$ is weight-exact.
	In view of the very definition of the Chow weight structure on $\DAct(X)$, this amounts to checking that, for every proper morphism $p \colon W \to X$ with $W$ a smooth $k$-variety, the object $p_! \Q_W(n)[2n] \in \Db(\Perv_{\ell}^\rom{mf}(X))$ is pure of weight $0$ for every $n \in \Z$.
	By definition of weights on mixed $\ell$-adic complexes, this follows from Deligne's proof of the Weil Conjecture in~\cite{Del80}.
\end{proof}	

\subsection{Autoequivalences of Nori motivic sheaves}\label{subsec:autoeq}

In this section, we use the Nori realization to study autoequivalences of Nori motivic sheaves as a six functor formalism.
We start by recalling the following general notions:

\begin{defn}\label{defn:mor-6ff}
	Let $X \mapsto \Der_1(X)$ and $X \mapsto \Der_2(X)$ be two six functor formalisms over (quasi-projective) $k$-varieties.
	\begin{enumerate}
		\item A \textit{morphism of six functor formalisms} $R \colon \Der_1(-) \rightarrow \Der_2(-)$ is the datum of
		\begin{itemize}
			\item for every $k$-variety $X$, a monoidal triangulated functor
			\begin{equation*}
				R_X \colon \Der_1(X) \rightarrow \Der_2(X),
			\end{equation*}
			\item for every morphism $f \colon X \rightarrow Y$, a natural isomorphism of functors $\Der_1(Y) \rightarrow \Der_2(X)$
			\begin{equation}\label{eq:autoeq-f^*}
				f^* \circ R_Y \xrightarrow{\sim} R_X \circ f^*
			\end{equation}
		\end{itemize}
		satisfying the natural compatibility conditions stated in~\cite{AyoBettiReal}.
		Such a morphism $R$ is called \textit{strong} if the natural transformation
		\begin{equation*}
			\Homint_X(R_X(-),R_X(-)) \rightarrow R_X \Homint_X(-,-)
		\end{equation*}
		obtained from the monoidality of $R_X$
		by adjunction in the second variable,
		and the induced natural transformation 
		\begin{equation*}
			R_Y \circ f_* \rightarrow f_* \circ R_X,
		\end{equation*}
		are both invertible for every choice of $X$ and $f \colon X \to Y$.
		\item Given two morphisms $R_1, R_2 \colon \Der_1(-) \rightarrow \Der_2(-)$, a \textit{$2$-morphism} $\alpha \colon R_1 \rightarrow R_2$ is the datum of
		\begin{itemize}
			\item for every $k$-variety $X$, a natural transformation of functors $\Der_1(X) \rightarrow \Der_2(X)$
			\begin{equation*}
				\alpha_X \colon R_{1,X} \rightarrow R_{2,X}
			\end{equation*}
		\end{itemize}
		compatible with the structure of monoidal fibered category in the following sense:
		\begin{enumerate}
			\item[(i)] For every $k$-variety $X$, the diagram of functors $\Der_1(X) \times \Der_1(X) \rightarrow \Der_2(X)$
			\begin{equation*}
				\begin{tikzcd}
					R_{1,X}(-) \otimes R_{1,X}(-) \arrow{d}[']{\alpha_X \otimes \alpha_X} \arrow{rr}{\sim} && R_{1,X}(- \otimes -) \arrow{d}{\alpha_X(- \otimes -)} \\
					R_{2,X}(-) \otimes R_{2,X}(-) \arrow{rr}{\sim} && R_{2,X}(- \otimes -)
				\end{tikzcd}
			\end{equation*}
			commutes.
			\item[(ii)] For every morphism of $k$-varieties $f \colon X \rightarrow Y$, the diagram of functors $\Der_1(Y) \rightarrow \Der_2(X)$
			\begin{equation*}
				\begin{tikzcd}
					f^* \circ R_{1,Y} \arrow{d}{f^* \alpha_Y} \arrow{rr}{\sim} && R_{1,X} \circ f^* \arrow{d}{\alpha_X f^*} \\
					f^* \circ R_{2,Y} \arrow{rr}{\sim} && R_{2,X} \circ f^*
				\end{tikzcd}
			\end{equation*}
			commutes.
		\end{enumerate}
		Such a $2$-morphism $\alpha$ is called a \textit{$2$-isomorphism} if the natural transformation $\alpha_X$ is invertible for every $k$-variety $X$;
		in this case, we say that $R_1$ and $R_2$ are \textit{isomorphic}.
	\end{enumerate}
\end{defn}

Strong morphisms are automatically compatible with the entire six functor formalism.
We are particularly interested in the case where $\Der_1 = \Der_2$.

\begin{defn}
	Let $X \mapsto \Der(X)$ be a six functor formalism over (quasi-projective) $k$-varieties, and let $R \colon \Der(-) \to \Der(-)$ be a self-morphism.
	\begin{enumerate}
		\item We say that $R$ is an \textit{autoequivalence} if the functor $R_X \colon \Der(X) \to \Der(X)$ is an equivalence for every $k$-variety $X$.
		\item We say that $R$ is a \textit{trivial autoequivalence} if it is $2$-isomorphic to the identity autoequivalence $\id_{\Der}$ in the sense of Definition~\ref{defn:mor-6ff}(2).
		In this case, we call \textit{trivialization} of $R$ every $2$-isomorphism $\gamma \colon R \xrightarrow{\sim} \id_{\Der}$. 
	\end{enumerate}
\end{defn}

Note that any self-morphism $R \colon \Der(-) \to \Der(-)$ which is $2$-isomorphic to an autoequivalence is itself an autoequivalence, and that every autoequivalence is automatically a strong self-morphism.
Autoequivalences of $X \mapsto \Der(X)$ form a $2$-group under composition, and trivial autoequivalences form a normal sub-$2$-group of the latter.

\begin{nota}
	We let $\Auteq(X \mapsto \Der(X))$ denote the quotient of the $2$-group of autoequivalences of $X \mapsto \Der(X)$ by the sub-$2$-group of trivial autoequivalences, and we call it the \textit{autoequivalence $2$-group} of the six functor formalism $X \mapsto \Der(X)$.
\end{nota}

The six functor formalism of motivic sheaves should possess no interesting autoequivalences.
For Voevodsky motivic sheaves, this can be proved without difficulty:

\begin{prop}
	The $2$-group $\Auteq(X \mapsto \DAct(X))$ is trivial.
\end{prop}

\begin{proof}
	Note that every autoequivalence of $X \mapsto \DA(X)$ restricts to an autoequivalence of $X \mapsto \DAct(X)$ and that, conversely, every autoequivalence of the latter extends uniquely to an autoequivalence of the former;
	since this correspondence is compatible with $2$-isomorphisms, it preserves trivial autoequivalences.
	Hence, we have a canonical isomorphism
	\begin{equation*}
		\Auteq(X \mapsto \DAct(X)) = \Auteq(X \mapsto \DA(X)),
	\end{equation*}
	and the latter is trivial as a consequence of Drew--Gallauer's Theorem~\ref{thm:Drew-Gallauer}.
\end{proof}

If the Nori realization functors $\Nri_X^*$ were known to be equivalences, this would imply the triviality of the $2$-group $\Auteq(X \mapsto \Db(\MNori(X)))$.
Unfortunately, we are not able to confirm this expectation in full.
Nevertheless, as a consequence of Theorem~\ref{thm:comp}, we can at least establish a weaker variant of the expected result.
To this end, let us say that an autoequivalence $R$ of the six functor formalism $X \mapsto \Db(\MNori(X))$ is \textit{exact} if, for every $k$-variety $X$, the triangulated functor $R_X \colon \Db(\MNori(X)) \to \Db(\MNori(X))$ is $t$-exact for the perverse $t$-structure.
Since all trivial autoequivalences are obviously exact, we can define the sub-$2$-group
\begin{equation*}
	\Auteq^{ex}(X \mapsto \Db(\MNori)) \subset \Auteq(X \mapsto \Db(\MNori))
\end{equation*}
collecting the exact autoequivalences of $X \mapsto \Db(\MNori(X))$ modulo the trivial ones.
Our result can then be stated as follows:

\begin{thm}\label{thm:exact-Auteq}
	The $2$-group $\Auteq^{ex}(X \mapsto \Db(\MNori(X)))$ is trivial.
\end{thm}

\begin{proof}
	Fix an exact autoequivalence $R$ of the coefficient system $X \mapsto \Db(\MNori(X))$, and let us construct a trivialization $\gamma \colon R \xrightarrow{\sim} \id_{\Db(\MNori)}$.
	For every $k$-variety $X$, the $t$-exact functor
	\begin{equation*}
		R_X: \Db(\MNori(X)) \xrightarrow{\sim} \Db(\MNori(X))
	\end{equation*}
	induces an equivalence of abelian categories
	\begin{equation*}
		R_X^0 \colon \MNori(X) \xrightarrow{\sim} \MNori(X).
	\end{equation*}
	In fact, since the triangulated functor $R_X$ admits a dg-enhancement,~\cite[Thm.~1]{Vologodsky} implies that it is naturally isomorphic to the trivial derived functor of $R_X^0$.
	We claim that, as $X$ varies, these identifications can be chosen compatibly with the structure of monoidal fibered category of $X \mapsto \Db(\MNori(X))$.
	To see this, it suffices to consider the functoriality on quasi-projective $k$-varieties.
	In the quasi-projective setting, using the method of~\cite{TerenziNori}, we are reduced to showing that the identifications in question respect (shifted) inverse images under smooth morphisms, direct images under closed immersions, and the external tensor product functors. 
	Since these functors are all $t$-exact and dg-enhanced, by~\cite[Thm.~1]{Vologodsky} it suffices to check the required compatibilities on the level of abelian categories, where they are obvious.

	Suppose for a moment that we had defined the sought-after trivialization $\gamma$ of $R$. 
	For each $k$-variety $X$, the natural isomorphism between functors $\Db(\MNori(X)) \rightarrow \Db(\MNori(X))$
	\begin{equation*}
		\gamma_X \colon R_X \xrightarrow{\sim} \id_{\Db(\MNori(X))}
	\end{equation*}
	would restrict to a natural isomorphism between exact functors $\MNori(X) \rightarrow \MNori(X)$
	\begin{equation*}
		\gamma_X^0 \colon R_X^0 \xrightarrow{\sim} \id_{\MNori(X)},
	\end{equation*}
	and, in fact,~\cite[Thm.~2]{Vologodsky} would imply that $\gamma_X$ equals the natural isomorphism induced by $\gamma_X^0$.
	We are left to construct the natural isomorphisms $\gamma_X^0$.
	The key point is that, by Remark~\ref{rem:Nri-real_uniqueness}, there is a $2$-isomorphism of six functor formalisms 
	\begin{equation*}
		\beta \colon R \circ \Nri^* \xrightarrow{\sim} \Nri^*.
	\end{equation*}
	For every fixed $k$-variety $X$, we obtain a natural isomorphism between functors $\DAct(X) \rightarrow \MNori(X)$
	\begin{equation}\label{formula:triv_H^0-Nri}
		\beta_X^0 \colon R_X^0 \circ \pH^0 \circ \Nri_X^* = \pH^0 \circ R_X \circ \Nri_X^* \xrightarrow{\beta_X} \pH^0 \circ \Nri_X^*.
	\end{equation}
	By the very construction of the Nori realization, the composite functor $\pH^0\circ \Nri_X^* \colon \DAct(X) \rightarrow \MNori(X)$ coincides with the homological functor $h_X^0$.
	The trick now is to identify $\MNori(X)$ with the universal abelian factorization of the latter:
	this allows us to obtain $\gamma_X^0$ directly from the lifting principles of universal abelian factorizations.
    In detail, since the diagram
    \begin{equation*}
    	\begin{tikzcd}
    		\DAct(X) \arrow{rr}{\id} \arrow{d}{h_X} && \DAct(X) \arrow{d}{h_X} \\
    		\MNori(X) \arrow{rr}{R_X^0} && \MNori(X) 
    	\end{tikzcd}
    \end{equation*}
    is commutative up to the natural isomorphism $\beta_X^0$, we get an exact functor
    \begin{equation*}
    	F_X \colon \MNori(X) \to \MNori(X) 
    \end{equation*}
    making the diagram
    \begin{equation*}
    	\begin{tikzcd}
    		\DAct(X) \arrow{rr}{\id} \arrow{d}{h_X} && \DAct(X) \arrow{d}{h_X} \\
    		\MNori(X) \arrow{rr}{F_X} && \MNori(X)
    	\end{tikzcd}
    \end{equation*} 
    commute strictly, together with a natural isomorphism 
    \begin{equation*}
    	\tilde{\beta}_X^0 \colon R_X^0 \xrightarrow{\sim} F_X.
    \end{equation*}
    Since $F_X$ is uniquely determined by the strict commutativity requirement (see Remark~\ref{rem:j^*-unique}), we must have $F_X = \id_{\MNori(X)}$. 
    Thus, we can set $\gamma_X^0 := \tilde{\beta}_X^0$.
	
	In order to make sure that, as $X$ varies, the resulting natural isomorphisms $\gamma_X$ define a trivialization as wanted, one needs to prove that they are compatible with inverse image functors as well as with the tensor product in the quasi-projective setting.
	Using the method of~\cite{TerenziTensorFibCat,TerenziConstructing}, it suffices to show that they are compatible with (shifted) inverse images under smooth morphisms, direct images under closed immersions, and the external tensor product functors.
	Since these are all $t$-exact and dg-enhanced, one can use~\cite{Vologodsky} to check the required compatibilities on the abelian level, where they are obvious.
	This concludes the proof.
\end{proof}

\begin{rem}
	In view of Theorem~\ref{thm:exact-Auteq}, showing the triviality of the whole group $\Auteq(X \mapsto \Db(\MNori(X)))$ amounts to showing that any autoequivalence of Nori motivic sheaves must be exact.
	It is not clear to us how hard this would be (unconditionally on the existence of the motivic perverse $t$-structures).
	Using Drew--Gallauer's Theorem~\ref{thm:Drew-Gallauer} as in the proof of Theorem~\ref{thm:exact-Auteq}, one sees that all objects of $\Db(\MNori(X))$ lying in the image of $\Nri_X^*$ must stay fixed under any autoequivalence, and~\cite[Thm.~1.12]{TerenziNori} implies that the abelian category $\MNori(X)$ is generated by  objects of this kind under (co)kernels.
	However, this does not allow us to conclude formally that any autoequivalence must keep the whole of $\MNori(X)$ fixed, precisely because one does not know in advance that it commutes with (co)kernels but only with cones.
\end{rem}

\subsection{Independence properties of Ayoub's construction}\label{subsect:indep}

Recall that the very construction of Nori motivic sheaves over $k$-varieties, both in Ivorra--Morel's and in Ayoub's sense, depends on the choice of a complex embedding $\sigma \colon k \hookrightarrow \C$.
In this subsection, we discuss this aspect more closely, leading to some interesting independence-of-$\sigma$ results.
Throughout, we restore the more precise Notation~\ref{nota:MNori_emb} for Ivorra--Morel's categories.

The starting point is the following strong independence result in Ivorra--Morel's setting:

\begin{prop}[{\cite[Prop.~6.11]{IM19}}]\label{prop:MNori_sigma1=MNori_sigma2}
	Let $X$ be a (quasi-projective) $k$-variety.
	For every choice of two complex embeddings $\sigma_1, \sigma_2 \colon k \hookrightarrow \C$, the abelian categories $\MNori_{\sigma_1}(X)$ and $\MNori_{\sigma_2}(X)$ are canonically equivalent.
\end{prop}

\begin{proof}
	Indeed, the abelian category $\MNori_{\sigma}(X)$ is defined in~\cite[\S~2.1]{IM19} as the Serre quotient of the abelian hull $\mathbf{A}(\DAct(X))$ by the kernel of the homological functor $\pH^0 \circ \Bti_{\sigma,X}^* \colon \DAct(k) \to \Perv(X)$.
	Therefore, it suffices to show that the kernels associated to $\sigma_1$ and $\sigma_2$ coincide.
	This ultimately follows from the existence of comparison isomorphisms between the corresponding Betti cohomology theories after scalar extension:
	for example, one can use the comparison isomorphism between Betti and de Rham cohomology.
\end{proof}

We deduce the following result in Ayoub's setting:

\begin{thm}\label{thm:Dbgeo^Gmot_sigma1=Dbgeo^Gmot_sigma2}
	Let $X$ be a (quasi-projective) $k$-variety.
	For every choice of two complex embeddings $\sigma_1, \sigma_2 \colon k \hookrightarrow \C$, there exists a canonical equivalence of stable $\infty$-categories
	\begin{equation}\label{formula:Dbgeo^Gmot_sigma1=Dbgeo^Gmot_sigma2}
		\Dbgeoemb{\sigma_1}(X)^{\GmotAy(k,\sigma_1)} = \Dbgeoemb{\sigma_2}(X)^{\GmotAy(k,\sigma_2)}
	\end{equation}
	satisfying the natural cocycle condition for triples of embeddings.
	As $X$ varies, the equivalences \eqref{formula:Dbgeo^Gmot_sigma1=Dbgeo^Gmot_sigma2} are compatible with the six operations, as well as with Beilinson's gluing functors.
\end{thm}

\begin{proof}
	We define the equivalence \eqref{formula:Dbgeo^Gmot_sigma1=Dbgeo^Gmot_sigma2} as the composite
\begin{equation*}
\begin{tikzcd}
\Dbgeoemb{\sigma_1}(X)^{\GmotAy(k,\sigma_1)}
&&[-1em]&
\Dbgeoemb{\sigma_2}(X)^{\GmotAy(k,\sigma_2)}
\\
\Db(\Pervgeoemb{\sigma_1}^{\rm{A}}(X)^{\GmotAy(k,\sigma_1)})
	\arrow{u}[dsiml]{\sim}
&
\Db(\MNori_{\sigma_1}(X)) \arrow{l}[']{\sim} \arrow[equal]{r}
&
\Db(\MNori_{\sigma_2}(X)) \arrow{r}{\sim}
&
\Db(\Pervgeoemb{\sigma_2}^{\rm{A}}(X)^{\GmotAy(k,\sigma_2)})
	\arrow{u}[dsimr]{\sim}
\end{tikzcd}
\end{equation*}
	where the central equivalence is given by Proposition~\ref{prop:MNori_sigma1=MNori_sigma2} while the other equivalences are provided by Theorem~\ref{thm:Bei-equiv} and Theorem~\ref{thm:comp};
	by construction, these are all equivalences of stable $\infty$-categories.
	It is easy to see that this definition satisfies the cocycle condition for triples of complex embeddings.
	The compatibility with the six operations and with Beilinson's gluing functors follows from Proposition~\ref{prop:o_S-6ff}.
\end{proof}

One could wonder whether the equivalence \eqref{formula:Dbgeo^Gmot_sigma1=Dbgeo^Gmot_sigma2} comes from an equivalence
\begin{equation}\label{formula:Dbgeo_sigma1=Dbgeo_sigma2}
	\Dbgeoemb{\sigma_1}(X) = \Dbgeoemb{\sigma_2}(X)
\end{equation}
making the diagram
\begin{equation}\label{dia:Dbgeo^Gmot-to-Dbgeo_sigma1,2}
	\begin{tikzcd}
		\Dbgeoemb{\sigma_1}(X)^{\GmotAy(k,\sigma_1)} \arrow[equal]{rr} \arrow{d}[']{\omega_{\sigma_1,X}} && \Dbgeoemb{\sigma_2}(X)^{\GmotAy(k,\sigma_2)} \arrow{d}{\omega_{\sigma_2,X}} \\
		\Dbgeoemb{\sigma_1}(X) \arrow[equal]{rr} && \Dbgeoemb{\sigma_2}(X)
	\end{tikzcd}
\end{equation}
commute up to natural isomorphism.
What would certainly be too much to ask for is a system of equivalences of the form~\eqref{formula:Dbgeo_sigma1=Dbgeo_sigma2} for all (quasi-projective) $k$-varieties $X$ compatible with the six operations.
Indeed, using compatibility with direct images and with the monoidal structure, this would yield functorial isomorphisms of Betti cohomology algebras
\begin{equation}\label{eq:H^*_sigma1=H^*_sigma2}
	H_{\sigma_1}^*(X,\Q) \simeq H_{\sigma_2}^*(X,\Q).
\end{equation}
But it is known that such an isomorphism cannot exist in general.
For instance,
when $k$ is imaginary quadratic,
Charles' work~\cite{Charles} provides examples of
smooth projective $k$-varieties $X$
whose Betti cohomology algebras under $\sigma_1$ and $\sigma_2$
are not isomorphic (not even with $\R$-coefficients).
The closest one can have in general is a canonical isomorphism
\begin{equation*}
	H_{\sigma_1}^*(X,\C) = H_{\sigma_2}^*(X,\C),
\end{equation*}
induced by the comparison isomorphism between Betti and de Rham cohomology.

In the same vein, one can ask how Nori's motivic Galois groups $\Gmot(k,\sigma_1)$ and $\Gmot(k,\sigma_2)$ compare to one another.
As a particular case of Proposition~\ref{prop:MNori_sigma1=MNori_sigma2}, the abelian categories $\MNori_{\sigma_1}(k)$ and $\MNori_{\sigma_2}(k)$ are canonically equivalent, compatibly with the monoidal structures.
Therefore, the Tannaka duals $\Gmot(k,\sigma_1)$ and $\Gmot(k,\sigma_2)$ get identified as pro-algebraic groups over $\Q$.
However, this identification is not induced by a natural isomorphism between the corresponding fibre functors
\begin{equation*}
	\iota_{k,\sigma_1}, \; \iota_{k,\sigma_2} \colon \MNori(k) \to \vect_{\Q}.
\end{equation*}
Indeed, such a natural isomorphism would yield isomorphisms of Betti cohomology algebras as in \eqref{eq:H^*_sigma1=H^*_sigma2}.
In other words, even if the derived categories $\Dbgeoemb{\sigma_1}(\Spec(k))$ and $\Dbgeoemb{\sigma_2}(\Spec(k))$ are canonically equivalent (being both canonically equivalent to $\Db(\vect_{\Q})$), this equivalence does not make the diagram~\eqref{dia:Dbgeo^Gmot-to-Dbgeo_sigma1,2} commute.

In conclusion, the significance of Theorem~\ref{thm:Dbgeo^Gmot_sigma1=Dbgeo^Gmot_sigma2} is that, even though the category $\Dbgeoemb{\sigma}(X)$ and the $\GmotAy(k,\sigma)$-action do depend on the choice of $\sigma: k \hookrightarrow \C$, such ambiguities cancel out when passing to the homotopy-fixed points $\Dbgeoemb{\sigma}(X)^{\GmotAy(k,\sigma)}$.

\begin{rem}
	It would be interesting to see if one can construct the equivalences~\eqref{formula:Dbgeo^Gmot_sigma1=Dbgeo^Gmot_sigma2} without passing through Ivorra--Morel's categories.
\end{rem}

\appendix

\section{Equivariant objects under pro-algebraic groups}\label{sect:App}

In this appendix, we collect general results about actions of pro-algebraic groups on ordinary categories and their associated categories of equivariant objects.
We work systematically with explicit definitions, which allows us to check many facts by direct computations.

For sake of simplicity, we start by discussing the basic constructions in the setting of abstract group actions, then we explain how they can be adapted to the pro-algebraic setting.
In the final part, we specialize to neutral Tannakian categories:
the most important result is Proposition~\ref{prop:KrtimesQ}, in which we compute the Tannaka dual of equivariant Tannakian categories under suitable assumptions.

\subsection{Generalities on categories of equivariant objects}\label{sect:App_abstract-grps}

In this subsection, we consider actions of abstract groups on ordinary categories by self-equivalences.

\begin{defn}\label{defn:action-Q-Cat}
	Let $\Category$ be a category, and let $Q$ be an abstract group.
	By an \textit{action} of $Q$ on $\Category$ we mean the datum of
	\begin{itemize}
		\item for every $q \in Q$, a functor
		\begin{equation*}
			q \cdot - \colon \Category \rightarrow \Category, \quad C \mapsto q \cdot C,
		\end{equation*}
		\item a natural isomorphism of functors $\Category \to \Category$
		\begin{equation}\label{eq:app-action_idQ}
			1_Q \cdot - \xrightarrow{\sim} \id_{\Category},
		\end{equation}
		\item for every $q_1, q_2 \in Q$, a natural isomorphism of functors $\Category \to \Category$
		\begin{equation}\label{eq:app-action_q2q1}
			q_2 \cdot (q_1 \cdot -) \xrightarrow{\sim} (q_2 q_1) \cdot -
		\end{equation}
	\end{itemize}
	satisfying the following coherence conditions:
	for every $q \in Q$, the two diagrams of functors $\Category \to \Category$
	\begin{equation*}
		\begin{tikzcd}
			1_Q \cdot (q \cdot -) \arrow{rr}{\sim} \arrow{drr}{\sim} && (1_Q q) \cdot - \arrow[equal]{d} \\
			&& q \cdot -
		\end{tikzcd}
		\qquad
		\begin{tikzcd}
			(q 1_Q) \cdot - \arrow[equal]{d} && q \cdot (1_Q \cdot -) \arrow{ll}{\sim} \arrow{dll}{\sim} \\
			q \cdot - 
		\end{tikzcd}
	\end{equation*}
	commute and, for every $q_1, q_2, q_3 \in Q$, the diagram of functors $\Category \rightarrow \Category$
	\begin{equation*}
		\begin{tikzcd}
			q_3 \cdot (q_2 \cdot (q_1 \cdot -)) \arrow{rr}{\sim} \arrow{d}{\sim} && q_3 \cdot (q_2 q_1 \cdot -) \arrow{d}{\sim} \\
			q_3 q_2 \cdot (q_1 \cdot -) \arrow{rr}{\sim} && q_3 q_2 q_1 \cdot -
		\end{tikzcd}
	\end{equation*}
	commutes.
\end{defn}

The coherence conditions imply that, for every $q \in Q$, the functor $q \cdot - \colon \Category \to \Category$ is an equivalence with canonical quasi-inverse $q^{-1} \cdot -$. 
In practice, we are allowed to pretend that $Q$ acts strictly on $\Category$ by self-automorphisms, with $1_Q$ acting as the identity.
In order not to overload the notation, from now on we systematically write the additional natural isomorphisms~\eqref{eq:app-action_idQ} and~\eqref{eq:app-action_q2q1} as equalities.

\begin{defn}\label{defn:hom-fixed_points}
	Let $Q$ be an abstract group acting on a category $\Category$ as in Definition~\ref{defn:action-Q-Cat}.
	The associated category of \textit{equivariant objects} $\Category^Q$ is defined as follows:
	\begin{itemize}
		\item Objects are pairs $(C,\alpha)$ consisting of an object $C \in \Category$ and a collection $\alpha = (\alpha_q)_{q \in Q}$ of isomorphisms in $\Category$
		\begin{equation*}
			\alpha_q \colon q \cdot C \xrightarrow{\sim} C
		\end{equation*}
		such that the isomorphism $\alpha_{1_Q}$ is induced by the natural isomorphism \eqref{eq:app-action_idQ}, and the following cocycle condition is satisfied: 
		for every $q_1, q_2 \in Q$, the diagram
		\begin{equation*}
			\begin{tikzcd}
				q_2 (q_1 \cdot C) \arrow{rr}{q_2 \cdot \alpha_{q_1}} \arrow[equal]{d} && q_2 \cdot C \arrow{d}{\alpha_{q_2}} \\
				q_2 q_1 \cdot C \arrow{rr}{\alpha_{q_2 q_1}} && C
			\end{tikzcd}
		\end{equation*}
		commutes.
		\item A morphism $\phi \colon (C_1, \alpha_1) \rightarrow (C_2, \alpha_2)$ is the datum of a morphism $\phi \colon C_1 \rightarrow C_2$ in $\Category$ such that, for every $q \in Q$, the diagram in $\Category$
		\begin{equation*}
			\begin{tikzcd}
				q \cdot C_1 \arrow{rr}{q\cdot\phi} \arrow{d}{\alpha_{1,q}} && q \cdot C_2 \arrow{d}{\alpha_{2,q}} \\
				C_1 \arrow{rr}{\phi} && C_2
			\end{tikzcd}
		\end{equation*}
		commutes.
		\item Composition in $\Category^Q$ is induced by composition in $\Category$.
	\end{itemize}
\end{defn}

An elementary yet useful example is the following:
\begin{ex}\label{ex:htpy_permut}
	Let $\Category$ be a category, and let $Q$ be an abstract group.
	Suppose that, for every $q \in Q$, we are given a copy $q \Category$ of $\Category$ and an equivalence $\epsilon_q \colon \Category \xrightarrow{\sim} q \Category$, and assume for simplicity that $1_Q \Category = \Category$ and $\epsilon_{1_Q} = \id_{\Category}$.
	For every $q_0 \in Q$, consider the functor
	\begin{equation*}
		q_0 \colon - \prod_{q \in Q} q \Category \to \prod_{q \in Q} q_0 q \Category
	\end{equation*}
	defined by the equivalences $q \Category \xleftarrow{\epsilon_q} \Category \xrightarrow{\epsilon_{q_0 q}} q_0 q \Category$ as $q$ varies in $Q$.
	As the reader can check, this canonically defines a $Q$-action on $\prod_{q \in Q} q \Category$ in the sense of Definition \ref{defn:action-Q-Cat};
	it is determined by the left-translation action of $Q$ on itself. 
	The associated category of equivariant objects
	\begin{equation*}
		(\prod_{q \in Q} q \Category)^Q
	\end{equation*}
	is canonically equivalent to $\Category$ via the projection onto the identity factor.
	\qed
\end{ex}

Equivariant objects behave functorially with respect to group homomorphisms:
\begin{constr}\label{constr:res_htpy-fixed}
	Let $f \colon Q_1 \to Q_2$ be a morphism of abstract groups.
	Suppose that we are given a $Q_2$-action on $\Category$, and consider the $Q_1$-action on $\Category$ deduced via $f$.
	There is a natural \textit{restriction functor}
	\begin{equation*}
		\res_f = \res_{Q_2}^{Q_1} \colon \Category^{Q_2} \rightarrow \Category^{Q_1}, \quad (C,\alpha) \mapsto (C,\alpha^f),
	\end{equation*}
	where, for every $q \in Q_1$, the isomorphism $\alpha^f_q \colon q \cdot C \xrightarrow{\sim} C$ is defined as
	\begin{equation*}
		q \cdot C := f(q) \cdot C \xrightarrow{\alpha_{f(q)}} C.
	\end{equation*}
	In the case of the trivial homomorphism $1 \to Q$, we obtain in this way the \textit{forgetful functor}
	\begin{equation*}
		\for_Q \colon \Category^Q \to \Category, \quad (C,\alpha) \mapsto C,
	\end{equation*}
	where we identify $\Category^{1}$ with $\Category$ in the obvious way.
	\qed
\end{constr}

Restriction functors are clearly compatible with composition of group homomorphisms.
As a consequence, they enjoy the basic properties of forgetful functors:
\begin{lem}\label{lem:res_faith_full}
Keep the notation of
Construction~\ref{constr:res_htpy-fixed}.
Then:
\begin{enumerate}
\item The functor $\res_{Q_2}^{Q_1}$ is
faithful and conservative.
\item If $f \colon Q_1 \rightarrow Q_2$ is surjective,
the functor $\res_{Q_2}^{Q_1}$ is also full.
\end{enumerate}
\end{lem}
\begin{proof}
	For the first statement, the result is clear if $Q_1 = 1$.
	In the general case, consider the commutative diagram
	\begin{equation*}\begin{tikzcd}
			\Category^{Q_2}
			\arrow{rr}{\res_{Q_2}^{Q_1}}
			\ar[dr, bend right, "\for_{Q_2}"']
			&& \Category^{Q_1} \arrow[bend left]{dl}{\for_{Q_1}} \\
			& \Category.
	\end{tikzcd}\end{equation*}
	Since the functor $\for_{Q_2}$, which is already known to be faithful (resp. conservative), factors through $\res_{Q_2}^{Q_1}$, the latter must be faithful (resp. conservative) as well.
	
	For the second statement, fix two objects $(C_1,\alpha_1), (C_2,\alpha_2) \in \Category^{Q_2}$.
	Since $f$ is surjective, saying that a morphism $C_1 \rightarrow C_2$ in $\Category$ is compatible with $\alpha^f_{1,q}$ and $\alpha^f_{2,q}$ for all $q \in Q_1$ is the same as saying that it is compatible with $\alpha_{1,q}$ and $\alpha_{2,q}$ for all $q \in Q_2$.
\end{proof}

\begin{nota}
	When an abstract group $Q$ acts on a category $\Category$ and $Q' \leq Q$ is a subgroup, we write the effect of the restriction functor along the inclusion homomorphism as
	\begin{equation*}
		\res_{Q}^{Q'} \colon (C,\alpha) \mapsto (C,\alpha|_{Q'}).
	\end{equation*}
\end{nota}

Inclusions of normal subgroups have an additional feature: 
\begin{constr}\label{constr:Q_acts_C^Q'}
	Let $Q$ be an abstract group acting on a category $\Category$ (as in Definition~\ref{defn:action-Q-Cat}), and let $Q' \leq Q$ be a normal subgroup.
	Then $Q$ acts naturally on $\Category^{Q'}$ by the formula
	\begin{equation*}
		q \cdot (C,\alpha) = (q \cdot C,{^q \alpha}),
	\end{equation*}
	where, for every $q' \in Q'$, the isomorphism ${^q \alpha_{q'}} \colon q' \cdot (q \cdot C) \xrightarrow{\sim} q \cdot C$ is defined as
	\begin{equation*}
		q' \cdot (q \cdot C) = q \cdot (q^{-1} q' q \cdot C) \xrightarrow{q \cdot \alpha_{q^{-1} q' q}} q \cdot C.
	\end{equation*}
	Note that ${^q \alpha_{q'}}$ is well-defined, since the element $q^{-1} q' q \in Q$ belongs to $Q'$ (as $Q'$ is normal in $Q$ by hypothesis).
	In order to make sure that the pair $(q \cdot C, {^q \alpha})$ really defines an object of $\Category^{Q'}$, we have to check that the cocycle condition is satisfied:
	this amounts to showing that, for every $q'_1, q'_2 \in Q'$, the diagram in $\Category$
	\begin{equation*}
		\begin{tikzcd}
			q'_2 \cdot (q'_1 \cdot (q \cdot C)) \arrow{rr}{q'_2 \cdot {^q \alpha_{q'_1}}} \arrow[equal]{d} && q'_2 \cdot (q \cdot C) \arrow{d}{{^q \alpha_{q'_2}}} \\
			q'_2 q'_1 \cdot (q \cdot C) \arrow{rr}{{^q \alpha_{q_2 q_1}}} && q \cdot C
		\end{tikzcd}
	\end{equation*}
	commutes.
	This follows from the cocycle condition for the $Q$-action on $\Category$;
	we leave the details to the interested reader.
	By construction, the forgetful functor $\for_{Q'} \colon \Category^{Q'} \rightarrow \Category$ is $Q$-equivariant.
	\qed
\end{constr}

\begin{nota}
	For sake of simplicity, we write $((C,\alpha'),\alpha)$ for a typical object of the category of equivariant objects $(\Category^{Q'})^Q$ for the $Q$-action on $\Category^{Q'}$ discussed in Construction~\ref{constr:Q_acts_C^Q'}:
	here $(C,\alpha') \in \Category^{Q'}$, $(C,\alpha) \in \Category^Q$ and, for every $q \in Q$, the isomorphism $\alpha_q \colon q \cdot C \xrightarrow{\sim} C$ defines a morphism $q \cdot (C,\alpha') \xrightarrow{\sim} (C,\alpha')$ in $\Category^{Q'}$ in the sense of Definition~\ref{defn:hom-fixed_points}.
\end{nota}

\begin{lem}\label{lem:sect_normal-sbgrp}
	Keep the notation and assumptions of Construction~\ref{constr:Q_acts_C^Q'}.
	Then the functor
	\begin{equation*}
		(\for_{Q'})^Q \colon (\Category^{Q'})^Q \rightarrow \Category^Q
	\end{equation*}
	induced by the $Q$-equivariant functor $\for_{Q'} \colon \Category^{Q'} \to \Category$ admits a canonical fully faithful section
	\begin{equation*}
		\sect^Q_{Q'} \colon \Category^Q \rightarrow (\Category^{Q'})^Q, \quad (C,\alpha) \mapsto ((C,\alpha|_{Q'}),\alpha).
	\end{equation*}
\end{lem}
\begin{proof}
	In the first place, let us check that the formula in the statement really defines a functor $\Category^Q \rightarrow (\Category^{Q'})^Q$.
	We have to show that, for every object $(C,\alpha) \in \Category^Q$ and every element $q \in Q$, the isomorphism $\alpha_q \colon q \cdot C \xrightarrow{\sim} C$ in $\Category$ satisfies the compatibility condition of Definition~\ref{defn:hom-fixed_points} with respect to the $Q'$-action on $\Category$:
	this amounts to saying that, for every $q' \in Q'$, the diagram in $\Category$
	\begin{equation*}
		\begin{tikzcd}
			q' \cdot (q \cdot C) \arrow{rr}{q' \cdot \alpha_q} \arrow{d}{{^q \alpha_{q'}}} && q' \cdot C \arrow{d}{\alpha_{q'}} \\
			q \cdot C \arrow{rr}{\alpha_q} && C
		\end{tikzcd}
	\end{equation*}
	commutes.
	This follows from the cocycle condition for the $Q$-action on $\Category$;
	we leave the details to the interested reader.
    It is clear that the composite
	\begin{equation*}
		\Category^Q \xrightarrow{\sect_{Q'}^Q} (\Category^{Q'})^Q \xrightarrow{(\for_{Q'})^Q} \Category^Q
	\end{equation*}
	is the identity functor, so $\sect_{Q'}^Q$ is a section to $(\for_{Q'})^Q$ as claimed.
	This implies, in particular, that $\sect_{Q'}^Q$ is faithful.
	It is also full, because giving a morphism $\phi \colon ((C_1,\alpha'_1),\alpha_1) \rightarrow ((C_2,\alpha'_2),\alpha_2)$ in $(\Category^{Q'})^Q$ amounts to giving a morphism $\phi \colon C_1 \rightarrow C_2$ in $\Category$ which defines both a morphism $(C_1,\alpha'_1) \rightarrow (C_2,\alpha'_2)$ in $\Category^{Q'}$ and a morphism $(C_1,\alpha_1) \rightarrow (C_2,\alpha_2)$ in $\Category^Q$.
\end{proof}

\begin{rem}
	In general, the functor $\sect_{Q'}^Q$ is not induced by a functor $\Category \rightarrow \Category^{Q'}$: 
	in fact, there is no natural way to define such a functor.
\end{rem}

\subsection{Compatibility with short exact sequences}

Throughout this subsection, we fix a category $\Category$ endowed with an action by an abstract group $Q$ fitting into a short exact sequence of the form
\begin{equation*}
	1 \rightarrow Q' \rightarrow Q \rightarrow Q'' \rightarrow 1.
\end{equation*}
For notational simplicity, we identify $Q'$ with its image in $Q$, which is a normal subgroup.
We want to relate the equivariant objects under $Q$ with those under $Q'$ by means of $Q''$.
Recall the natural $Q$-action on $\Category^{Q'}$ described in Construction \ref{constr:Q_acts_C^Q'}.
Intuitively, the $Q'$-action on $\Category^{Q'}$ should be trivial, so the $Q$-action on $\Category^{Q'}$ should factor through the quotient $Q''$ and induce an equivalence
\begin{equation*}
	\Category^Q = (\Category^{Q'})^{Q''}.
\end{equation*}
This is not literary true in general, but it becomes so under additional assumptions.
A first example is given by the following result:

\begin{prop}\label{prop:C^Q=(C^Q')^Q''-1}
	Suppose that the projection $Q \rightarrow Q''$ admits a splitting;
	use it to view $Q''$ as a subgroup of $Q$, and identify $Q$ with the semi-direct product $Q' \rtimes Q''$.
	Then the composite functor
	\begin{equation*}
		\Category^Q \xrightarrow{\sect^Q_{Q'}} (\Category^{Q'})^{Q} \xrightarrow{\res_Q^{Q''}} (\Category^{Q'})^{Q''}
	\end{equation*}
	is an equivalence.
\end{prop}
\begin{proof}
	By construction, the composite in the statement sends
	\begin{equation*}
		(C, (\alpha_q)_{q \in Q}) \mapsto ((C, (\alpha_{q'})_{q' \in Q'}), (\alpha_{q''})_{q'' \in Q''}).
	\end{equation*}
	Every given element $q \in Q$ can be written in a unique way as a product $q' q''$ with $q' \in Q'$ and $q'' \in Q''$.
	Hence, for every object $(C,(\alpha_q)_{q \in Q}) \in \Category^Q$, the cocycle condition for the $Q$-action on $\Category$ asserts that the diagram in $\Category$
	\begin{equation*}
		\begin{tikzcd}
			q' \cdot (q'' \cdot C) \arrow{rr}{q' \cdot \alpha_{q''}} \arrow[equal]{d} && q' \cdot C \arrow{d}{\alpha_{q'}} \\
			q \cdot C \arrow{rr}{\alpha_q} && C
		\end{tikzcd}
	\end{equation*}
	commutes, which implies that $\alpha_q$ is uniquely determined by $\alpha_{q'}$ and $\alpha_{q''}$.
	This observation allows one to construct a quasi-inverse functor explicitly;
	we  leave the details to the interested reader.
\end{proof}

If one drops the assumption that the projection $Q \rightarrow Q''$ admits a section, the expression $(\Category^{Q'})^{Q''}$ becomes problematic, since there is no natural way to define a $Q''$-action on $\Category^{Q'}$.
In order to make sense of this residual action, one has to make it precise in which sense $Q'$ acts trivially on $\Category^{Q'}$.
This is based on the following notion:
\begin{defn}\label{defn:isom_actions-trivializ}
	Let $\Category$ be a category, and let $Q$ be an abstract group.
	\begin{enumerate}
		\item Suppose that we are given two actions of $Q$ on $\Category$;
		for every $q \in Q$, we write $q \cdot_A -$ and $q \cdot_B -$ for the associated functors $\Category \rightarrow \Category$ in these two actions.
		By an \textit{isomorphism} $\gamma$ between the two actions of $Q$ of $\Category$ we mean the datum of
		\begin{itemize}
			\item for every $q \in Q$, a natural isomorphism between functors $\Category \rightarrow \Category$
			\begin{equation*}
				\gamma_q \colon q \cdot_A - \xrightarrow{\sim} q \cdot_B -
			\end{equation*}
		\end{itemize}
		satisfying the following coherence conditions:
		the diagram of functors $\Category \to \Category$
		\begin{equation*}
			\begin{tikzcd}
				1_Q \cdot_A - \arrow{rr}{\gamma_{1_Q}} \arrow[equal]{dr} && 1_Q \cdot_B - \arrow[equal]{dl} \\
				& \id_{\Category}
			\end{tikzcd}
		\end{equation*}
		commutes and, for every $q_1, q_2 \in Q$, the diagram of functors $\Category \to \Category$
		\begin{equation*}
			\begin{tikzcd}
				(q_2 q_1) \cdot_A - \arrow{d}{\gamma_{q_2 q_1}} \arrow[equal]{r} &  q_2 \cdot_A (q_1 \cdot_A -) \arrow{rr}{q_2 \cdot_A \gamma_{q_1}} && q_2 \cdot_A (q_1 \cdot_B -) \arrow{d}{\gamma{q_2}} \\
				(q_2 q_1) \cdot_B - \arrow[equal]{rrr} &&& q_2 \cdot_B (q_1 \cdot_B -)
			\end{tikzcd}
		\end{equation*}
		commutes.
		\item By a \textit{trivialization} of a given action of $Q$ on $\Category$ we mean an isomorphism between that action and the trivial action where $q \cdot - = \id_{\Category}$ for every $q \in Q$.
	\end{enumerate}
\end{defn}

\begin{ex}\label{ex:trivializ}
	For every action of an abstract group $Q$ on a category $\Category$, the induced $Q$-action on $\Category^Q$ (defined as in Construction \ref{constr:Q_acts_C^Q'}) admits a canonical trivialization.
    Indeed, fix an element $q_0 \in Q$.
	Given $(C,\alpha) \in \Category^Q$, the isomorphism $\alpha_{q_0} \colon q_0 \cdot C \xrightarrow{\sim} C$ in $\Category$ induces an isomorphism in $\Category^Q$
	\begin{equation}\label{eq:trivializ}
		\alpha_{q_0} \colon q_0 \cdot (C,\alpha) \xrightarrow{\sim} (C,\alpha).
	\end{equation}
	To make sure that this is the case, we only need to check that, for every $q \in Q$, the diagram in $\Category$
	\begin{equation*}
		\begin{tikzcd}
			q q_0 \cdot C \arrow{rr}{q\cdot\alpha_{q_0}} \arrow{d}{{^{q_0} \alpha_q}} && q \cdot C \arrow{d}{\alpha_q} \\
			q_0 \cdot C \arrow{rr}{\alpha_{q_0}} && C
		\end{tikzcd}
	\end{equation*}
	commutes, which follows from the cocycle condition for the $Q$-action on $\Category$.
	As $(C,\alpha)$ varies in $\Category^Q$, the isomorphisms \eqref{eq:trivializ} define a natural isomorphism between functors $\Category^Q \rightarrow \Category^Q$
	\begin{equation}\label{eq:triviali-nat}
		\alpha_{q_0} \colon q_0 \cdot - \xrightarrow{\sim} \id_{\Category^Q}.
	\end{equation}
	As $q_0$ varies in $Q$, the natural isomorphisms \eqref{eq:triviali-nat} satisfy the coherence conditions of Definition~\ref{defn:isom_actions-trivializ}, again by the cocycle condition.
	Hence, they define a trivialization as claimed.
	\qed
\end{ex}

Isomorphisms of actions have the expected effect on equivariant objects:
\begin{constr}\label{constr:C^Q_A=C^Q_B}
	Suppose that we are given two actions of $Q$ on $\Category$ as well as an isomorphism between them (in the sense of Definition~\ref{defn:isom_actions-trivializ}(1)).
	This canonically induces an isomorphism between the associated categories of equivariant objects.
	To see this, following the notation of Definition~\ref{defn:isom_actions-trivializ}(1), write
	\begin{equation*}
		\gamma_q \colon q \cdot_A - \xrightarrow{\sim} q \cdot_B -
	\end{equation*}
	for the natural isomorphisms between the two endofunctors of $\Category$ defined by a given element $q \in Q$;
	write $\Category^Q_{(A)}$ and $\Category^Q_{(B)}$ for the categories of equivariant objects associated to the two actions.
	We have a functor
	\begin{equation*}
		\Category^Q_{(A)} \rightarrow \Category^Q_{(B)}, \quad (C,(\alpha_q)_{q \in Q}) \mapsto (C,(\alpha_q \circ \gamma_q^{-1})_{q \in Q}).
	\end{equation*}
	To make sure that this is well-defined, we only need to check that, for every object $(C,\alpha) \in \Category^Q_{(A)}$, the isomorphisms $\alpha_q \circ \gamma_q^{-1} \colon q \cdot_B C \xrightarrow{\sim} C$ satisfy the cocycle condition for the second $Q$-action on $\Category$.
	Using the compatibility of the isomorphisms $\gamma_q$ with respect to composition, this follows from the cocycle condition for the first $Q$-action on $\Category$.
	It is clear that the specular construction
	\begin{equation*}
		\Category^Q_{(B)} \rightarrow \Category^Q_{(A)}, \quad (C,(\beta_q)_{q \in Q}) \mapsto (C,(\beta_q \circ \gamma_q)_{q \in Q})
	\end{equation*}
	defines an inverse to the previous functor.
	Hence, we obtain an isomorphism of categories as stated.
	\qed
\end{constr}

The extra flexibility provided by isomorphic actions allows us to make sense of the formula $\Category^Q = (\Category^{Q'})^{Q''}$ even without assuming the existence of a splitting:
\begin{prop}\label{prop:C^Q=(C^Q')^Q''-2}
	Let $\Category$ be a category endowed with the action of an abstract group $Q$ which fits into the short exact sequence
	\begin{equation*}
		1 \rightarrow Q' \rightarrow Q \rightarrow Q'' \rightarrow 1.
	\end{equation*}
	Assume that we are given, in addition, a $Q''$-action on $\Category^{Q'}$, and consider the two resulting $Q$-actions on $\Category^{Q'}$:
	the one induced by the original $Q$-action on $\Category$ (as in Construction~\ref{constr:Q_acts_C^Q'}), and the one deduced from the additional $Q''$-action on $\Category^{Q'}$.
    Suppose that we are given an isomorphism $\gamma$ between these two $Q$-actions on $\Category^{Q'}$ which, when restricted to $Q'$, coincides with the canonical trivialization for the canonical $Q'$-action on $\Category^{Q'}$ (as described in Example~\ref{ex:trivializ}).
	Then we get a canonical equivalence
	\begin{equation*}
		\Category^Q = (\Category^{Q'})^{Q''}.
	\end{equation*}
\end{prop}
\begin{proof}
	For sake of clarity, given an element $q \in Q$ and object $(C,\alpha) \in \Category^{Q'}$, we write $q \cdot_A (C,\alpha)$ for the action induced by the original $Q$-action on $\Category$, and $q \cdot_B (C,\alpha)$ for the action induced by the additional $Q''$-action on $\Category^{Q'}$;
	similarly for the categories of equivariant objects (as in Construction~\ref{constr:C^Q_A=C^Q_B} above).
    Consider the two functors
	\begin{equation}\label{eq:C^Q_to_(C^Q')^Q}
		\sect_Q^{Q'} \colon \Category^Q \rightarrow (\Category^{Q'})^Q_{(A)}
	\end{equation}
	and
	\begin{equation}\label{eq:(C^Q')^Q''_to_(C^Q')^Q}
		\res_{Q''}^Q \colon (\Category^{Q'})^{Q''} \rightarrow (\Category^{Q'})^Q_{(B)},
	\end{equation}
	which are both known to be fully faithful (by Lemma~\ref{lem:sect_normal-sbgrp} and Lemma~\ref{lem:res_faith_full}(2), respectively).
	By construction, the image of \eqref{eq:C^Q_to_(C^Q')^Q} consists of the objects $((C,\alpha'),\alpha) \in (\Category^{Q'})^Q_{(A)}$ with $\alpha' = \alpha|_{Q'}$, while the image of \eqref{eq:(C^Q')^Q''_to_(C^Q')^Q} consists of the objects $((C,\alpha'),\beta) \in (\Category^{Q'})^Q_{(B)}$ satisfying $\beta_{q'} = \id_C$ for every $q' \in Q$.
    
    Now consider the isomorphism of actions $\gamma$ in the hypothesis: 
    its compatibility with the canonical trivialization over $Q'$ means precisely that $\gamma_{q'} = \alpha_{q'}$ for every $q' \in Q'$.
    Clearly, this implies that the images of the functors \eqref{eq:C^Q_to_(C^Q')^Q} and \eqref{eq:(C^Q')^Q''_to_(C^Q')^Q} correspond to one another under the canonical equivalence
    \begin{equation*}
    	(\Category^{Q'})^Q_{(A)} = (\Category^{Q'})^Q_{(B)}
    \end{equation*}
    obtained as in Construction~\ref{constr:C^Q_A=C^Q_B}.
	This proves the thesis.
\end{proof}

\begin{rem}
	The existence of a genuine $Q''$-action on $\Category^{Q'}$ compatible with the action of $Q$ is the key hypothesis in Proposition~\ref{prop:C^Q=(C^Q')^Q''-2}.
    If the projection $Q \rightarrow Q''$ admits a splitting, then this hypothesis is automatically satisfied, and we recover the result of Proposition~\ref{prop:C^Q=(C^Q')^Q''-1}.
\end{rem}

To conclude this subsection, we discuss a specific situation where Proposition \ref{prop:C^Q=(C^Q')^Q''-2} can be applied.
This is based on a categorical analogue of the notion of induced representation:
\begin{ex}\label{ex:ind-rep}
	Suppose that we are given a $Q'$-action on a category $\Category'$.
	Choose a set-theoretic splitting 
	\begin{equation*}
		Q'' \to Q, \quad r \mapsto q_r
	\end{equation*}
	to the projection $Q \to Q''$, and assume for simplicity that $q_{1_{Q''}} = 1_Q$.
	Suppose that, for each $r \in Q''$, we are given a category $q_r \Category'$ together with an equivalence
	\begin{equation*}
		\epsilon_r \colon \Category' \xrightarrow{\sim} q_r \Category',
	\end{equation*}
	and assume for simplicity that $1_{Q''} \Category' = \Category'$ and $\epsilon_{1_{Q''}} = \id_{\Category'}$.
	Form the direct product category
	\begin{equation*}
		\Category := \prod_{s \in Q''} q_s \Category'.
	\end{equation*}
	Given an element $q \in Q$ mapping to $r \in Q''$, we define a functor $q \cdot - \colon \Category \to \Category$ as follows:
	for every $s \in Q''$, we let $q \cdot - \colon q_s \Category' \to q_{rs} \Category'$ be the composite
	\begin{equation*}
		q_s \Category' \xleftarrow{\epsilon_s} \Category' \xrightarrow{q_{rs}^{-1} q q_s} \Category' \xrightarrow{\epsilon_{rs}} q_{rs} \Category',
	\end{equation*}
	and we define $q \cdot - \colon \prod_{s \in Q''} q_s \Category' \to \prod_{s \in Q''} q_{rs} \Category'$ consequently.
	This defines a $Q$-action on $\Category$ with the property that, for every $q' \in Q'$, the functor $q' \cdot - \colon \Category \to \Category$ is the direct product of the functors
	\begin{equation*}
		q_s \Category' \xleftarrow{\epsilon_s} \Category' \xrightarrow{q_s^{-1} q' q_s} \Category' \xrightarrow{\epsilon_s} q_s \Category.
	\end{equation*}
	In particular, the $Q'$-action on the identity component $\Category'$ is nothing but the original $Q'$-action.
	Moreover, for every $r \in Q''$, the functor $q_r \cdot - \colon \Category \to \Category$ acts on the identity component as $\epsilon_r \colon \Category' \rightarrow q_r \Category'$.
	
	We claim that this $Q$-action on $\Category$ satisfies the hypotheses of Proposition \ref{prop:C^Q=(C^Q')^Q''-2}.
	To this end, note that we have
	\begin{equation*}
		\Category^{Q'} = \prod_{s \in Q''} (q_s \Category')^{Q'},
	\end{equation*}
	where the $Q'$-action on each $q_s \Category'$ is twisted by $q_s$ as above.
	Consider the canonical $Q$-action on $\Category^{Q'}$ (as described in Construction \ref{constr:Q_acts_C^Q'}):
	for every $r \in Q''$, the equivalence $q_r \cdot - \colon \Category^{Q'} \to \Category^{Q'}$ induces an equivalence
	\begin{equation*}
		q_r \colon - \colon (\Category')^{Q'} \xrightarrow{\sim} (q_r \Category')^{Q'}
	\end{equation*}
	making the diagram
	\begin{equation*}
		\begin{tikzcd}
			(\Category')^{Q'} \arrow{rr}{q_r} \arrow{d}{\for_{Q'}}  && (q_r \Category')^{Q'} \arrow{d}{\for_{Q'}} \\
			\Category' \arrow{rr}{q_r} && q_r \Category'
		\end{tikzcd}
	\end{equation*} 
	commute.
	This allows us to define a $Q''$-action on $\Category^{Q'}$ as follows:
	given $r \in Q''$, for every $s \in Q''$ we let $r \cdot - \colon (q_s \Category')^{Q'} \to (q_{rs} \Category')^{Q'}$ be the composite
	\begin{equation}\label{eq:residualaction_indrep-new}
		(q_s \Category')^{Q'} \xleftarrow{q_s \cdot -} (\Category')^{Q'} \xrightarrow{q_{rs} \cdot -} (q_{rs} \Category')^{Q'},
	\end{equation}
	and we define $r \cdot - \colon \prod_{s \in Q''} (q_s \Category')^{Q'} \to \prod_{s \in Q''} (q_{rs} \Category')^{Q'}$ consequently.
	
	To prove the claim, we have to construct an isomorphism between the two $Q$-actions on $\Category^{Q'}$ extending the canonical trivialization over $Q'$.
	The crucial observation is the following:
	given an element $q \in Q$ mapping to $r \in Q''$, the action of $q$ on $\Category^{Q'}$ induced by the $Q$-action on $\Category$ is defined by the functors
	\begin{equation}\label{eq:residualaction_indrep-old}
		(q_s \Category')^{Q'} \xleftarrow{q_s \cdot -} (\Category')^{Q'} \xrightarrow{q_{rs}^{-1} q q_s \cdot -} (\Category')^{Q'} \xrightarrow{q_{rs} \cdot -} (q_{rs} \Category')^{Q'}.
	\end{equation}
	Now consider the canonical trivialization for the $Q'$-action on $(\Category')^{Q'}$ (as discussed in Example \ref{ex:trivializ}): 
	the natural isomorphism of functors $(\Category')^{Q'} \to (\Category')^{Q'}$
	\begin{equation*}
		q_{rs}^{-1} q q_s \cdot - \xrightarrow{\sim} \id_{(\Category')^{Q'}}
	\end{equation*}
	induces a natural isomorphism between the functors \eqref{eq:residualaction_indrep-old} and \eqref{eq:residualaction_indrep-new};
	taking the product over all $s \in Q''$, this defines a natural isomorphism between the two actions of $q$ on $\Category^{Q'}$.
	By construction, these natural isomorphisms assemble into an isomorphism of $Q$-actions with the required properties;
	we leave the details to the interested reader.
	Finally, applying Proposition \ref{prop:C^Q=(C^Q')^Q''-2}, we get a canonical equivalence
	\begin{equation*}
		\Category^Q = (\Category^{Q'})^{Q''} = (\Category')^{Q'},
	\end{equation*}
	where the second passage is induced by the projection onto the identity factor (as in Example \ref{ex:htpy_permut}).
	\qed 
\end{ex}

\begin{rem}\label{rem:ind-rep}
	There is a simple way to recognize categorical induced representations in practice.
	As in Example \ref{ex:ind-rep}, suppose that we are given a $Q'$-action on a category $\Category'$;
	suppose that, for every $r \in Q''$, we are given a category $r\Category'$, and assume that $1_{Q''} \Category' = \Category'$.
	Form the direct product category
	\begin{equation*}
		\Category := \prod_{s \in Q''} s\Category',
	\end{equation*}
	and suppose that we are given a $Q$-action on $\Category$ with the property that, given an element $q \in Q$ mapping to $r \in Q''$, the functor $q \cdot - \colon \Category \to \Category$ is the direct product of suitable functors 
	\begin{equation*}
		q \cdot - \colon s\Category' \to (rs)\Category' \qquad (s \in Q'').
	\end{equation*}
	This means that, in particular, the $Q'$-action on $\Category$ obtained by restriction is the direct product of suitable $Q'$-actions on the single components $s\Category$.
	We claim that, up to isomorphism, the $Q$-action on $\Category$ is induced by the original $Q'$-action on $\Category'$ in the sense of Example \ref{ex:ind-rep}.
	To see this, fix a set-theoretic splitting
	\begin{equation*}
		Q'' \to Q, \quad r \mapsto q_r
	\end{equation*}
	satisfying $q_{1_{Q''}} = 1_Q$.
	Given an element $q \in Q$ mapping to $r \in Q''$, the diagram
	\begin{equation*}
		\begin{tikzcd}
			s\Category' \arrow{rr}{q \cdot -} && (rs)\Category' \\
			\Category' \arrow{u}{q_s \cdot -} \arrow{rr}{q_{rs}^{-1} q q_s \cdot -} && \Category' \arrow{u}{q_{rs} \cdot -}
		\end{tikzcd}
	\end{equation*}
	commutes up to natural isomorphism for every $s \in Q''$.
	Hence, defining the equivalences
	\begin{equation*}
		\epsilon_r \colon \Category' \xrightarrow{q_r \cdot -} r\Category'
	\end{equation*}
	and renaming $r\Category' =: q_r \Category'$ for all $r \in Q''$, we are exactly in the situation of Example \ref{ex:ind-rep}, as wanted.
\end{rem}

\subsection{Equivariant objects under pro-algebraic groups}\label{sect:App_proalg-grps}

From now henceforth, we fix a field $F$ of characteristic $0$.
All categories and functors considered are $F$-linear;
in particular, all actions of groups are defined by $F$-linear functors.
We let $\CAlg_F$ denote the category of classical commutative $F$-algebras.

\begin{nota}\label{nota:Cat_Lambda}
	Let $\Category$ be an $F$-linear category.
	For every $\Lambda \in \CAlg_F$, we let $\Category_{\Lambda}$ denote the category with the same objects as $\Category$ and with homomorphisms defined as
	\begin{equation*}
		\Hom_{\Category_{\Lambda}}(C_1,C_2) = \Hom_{\Category}(C_1,C_2) \otimes_F \Lambda.
	\end{equation*}
	We usually write $C \otimes \Lambda$ for the object of $\Category_{\Lambda}$ corresponding to a given object $C \in \Category$.
\end{nota}

\begin{rem}\label{rem:C_Lambda=fppf-stack}
	Given an $F$-linear category $\Category$, the categories $\Category_{\Lambda}$ as $\Lambda$ varies in $\CAlg_F$ form a fibered category over $\CAlg_F^{op}$:
	for every $F$-algebra homomorphism $\Lambda \rightarrow \Lambda'$, the transition functor $\Category_{\Lambda} \rightarrow \Category_{\Lambda'}$ sends $C \otimes \Lambda \mapsto C \otimes \Lambda'$ for every $C \in \Category$.
	As a consequence of faithfully flat descent for modules, this fibered category is a stack for the fppf topology.
\end{rem}

The natural notion of pro-algebraic group action is based on the associated functor of points, as follows:
\begin{defn}\label{defn:action-linear}
	Let $\Category$ be an $F$-linear category, and let $Q$ be a pro-algebraic group over $F$.
	By an \textit{$F$-algebraic action} of $Q$ on $\Category$ we mean the datum of
	\begin{itemize}
		\item for every $\Lambda \in \CAlg_F$, a $\Lambda$-linear action of the abstract group $Q(\Lambda)$ on $\Category_{\Lambda}$ (in the sense of Definition~\ref{defn:action-Q-Cat})
	\end{itemize}
	functorially with respect to $\Lambda \in \CAlg_F$ in the following sense:
	for every $F$-algebra homomorphism $\phi \colon \Lambda \rightarrow \Lambda'$ and every $q \in Q(\Lambda)$, the diagram of $\Lambda$-linear functors
	\begin{equation*}
		\begin{tikzcd}
			\Category_{\Lambda} \arrow{rr}{q \cdot -} \arrow{d} && \Category_{\Lambda} \arrow{d} \\
			\Category_{\Lambda'} \arrow{rr}{\phi(q) \cdot -} && \Category_{\Lambda'}
		\end{tikzcd}
	\end{equation*}
	commutes.
\end{defn}

The natural notion of equivariant objects in the pro-algebraic setting follows the same lines:
\begin{defn}\label{defn:htpy-fixed_linear}
	Let $Q$ be a pro-algebraic group acting $F$-algebraically on $\Category$ as in Definition~\ref{defn:action-linear}.
	The associated category of \textit{$F$-algebraic equivariant objects} $\Category^Q$ is defined as follows:
	\begin{itemize}
		\item Objects are pairs $(C,\underline{\alpha})$ consisting of an object $C \in \Category$ and a family $\underline{\alpha} = (\alpha_{\Lambda})_{\Lambda \in \CAlg_F}$ where, for each $\Lambda \in \CAlg_F$, $\alpha_{\Lambda} = (\alpha_{\Lambda,q})_{q \in Q(\Lambda)}$ is a collection of isomorphisms
		\begin{equation*}
			\alpha_{\Lambda,q} \colon q \cdot (C \otimes \Lambda) \xrightarrow{\sim} C \otimes \Lambda
		\end{equation*}
		in $\Category_{\Lambda}$ such that the pair $(C \otimes \Lambda,\alpha_{\Lambda})$ defines a homotopy-fixed point for the action of the abstract group $Q(\Lambda)$ on $\Category_{\Lambda}$ (in the sense of Definition~\ref{defn:hom-fixed_points}), functorially with respect to $\Lambda \in \CAlg_F$ in the following sense:
		for every $F$-algebra homomorphism $\phi \colon \Lambda \rightarrow \Lambda'$ and every $q \in Q(\Lambda)$, the $\Lambda'$-module isomorphism
		\begin{equation*}
			\begin{tikzcd}
				\Hom_{\Category_{\Lambda}}(q \cdot (C \otimes \Lambda), C \otimes \Lambda) \otimes_{\Lambda,\phi} \Lambda' \xrightarrow{\sim} \Hom_{\Category_{\Lambda'}}(q \cdot (C \otimes \Lambda'), C \otimes \Lambda')
			\end{tikzcd}
		\end{equation*}
		sends $\alpha_{\Lambda,q} \otimes 1 \mapsto \alpha_{\Lambda',\phi(q)}$.
		\item A morphism $f \colon (C_1,\underline{\alpha}_1) \rightarrow (C_2,\underline{\alpha}_2)$ is the datum of a morphism $f \colon C_1 \rightarrow C_2$ in $\Category$ such that, for every $\Lambda \in \CAlg_F$ and every $q \in Q(\Lambda)$, the diagram in $\Category_{\Lambda}$
		\begin{equation*}
			\begin{tikzcd}
				q \cdot (C_1 \otimes \Lambda) \arrow{rr}{q \cdot f} \arrow{d}{\alpha_{1,\Lambda,q}} && q \cdot (C_2 \otimes \Lambda) \arrow{d}{\alpha_{2,\Lambda,q}} \\
				C_1 \otimes \Lambda \arrow{rr}{f} && C_2 \otimes \Lambda
			\end{tikzcd}
		\end{equation*}
		commutes.
		\item Composition in $\Category^Q$ is induced by composition in $\Category$.
	\end{itemize}
\end{defn}

Note that, given an $F$-algebraic action of a pro-algebraic group $Q$ on an $F$-linear category $\Category$, there is always a canonical faithful functor from the $F$-algebraic equivariant objects to the equivariant objects associated to the abstract action of $Q(F)$ on $\Category$ (as in Definition~\ref{defn:hom-fixed_points}): 
namely, the functor
\begin{equation}\label{eq:algh-tpy_to_abst-htpy}
	\Category^Q \to \Category^{Q(F)}, \quad (C,\underline{\alpha}) \mapsto (C,\alpha_F).
\end{equation}
The two notions match in the case of finite groups:
\begin{ex}\label{ex:htpy-fingrp}
	Let $Q$ be a finite group, and regard it as a finite constant algebraic group over $F$.
	Then, for any $F$-algebraic action of $Q$ on an $F$-linear category $\Category$, the functor \eqref{eq:algh-tpy_to_abst-htpy} is an equivalence.
	Indeed, fix an object $(C,\underline{\alpha}) \in \Category^Q$.
	For every $\Lambda \in \CAlg_F$, we have a canonical bijection $Q(\Lambda) = Q = Q(F)$.
	Hence, the isomorphism $\alpha_{\Lambda,q}$ for a chosen $q \in Q(\Lambda)$ is uniquely determined by the isomorphism $\alpha_{F,q}$ for the corresponding $q \in Q(F)$.
	\qed   
\end{ex}

In general, equivariant objects in the $F$-algebraic setting enjoy stronger finiteness properties than those in the abstract setting.
The difference is apparent for categories in which all Hom groups are finite-dimensional over $F$.

\begin{ex}\label{ex:Tann_as_htpyfixed}
	Let $\Tannakian$ be a neutral Tannakian category over $F$, and let $K$ denote its Tannaka dual with respect to a fixed fibre functor $\omega \colon \Tannakian \to \vect_F$.
	With respect to the trivial $K$-action on $\vect_F$, the category of $F$-algebraic equivariant objects $\vect_F^K$ is just the category $\Rep_F(K)$ of $F$-algebraic representations of $K$.
	Via Tannaka duality, this allows us to interpret $\Tannakian$ as a category of $F$-algebraic equivariant objects. 
	\qed
\end{ex}

The constructions and results collected
in Subsection~\ref{sect:App_abstract-grps}
for abstract group actions
admit natural analogues in the $F$-algebraic setting.
We content ourselves of giving an overview,
modifying the arguments of
Subsection~\ref{sect:App_abstract-grps} when needed.

To begin with, given a morphism of pro-algebraic groups $Q_1 \rightarrow Q_2$ and an $F$-algebraic $Q_2$-action on $\Category$, there is a restriction functor 
\begin{equation*}
	\res_{Q_2}^{Q_1} \colon \Category^{Q_2} \rightarrow \Category^{Q_1},
\end{equation*} 
analogously to Construction~\ref{constr:res_htpy-fixed}.
The conclusions of Lemma~\ref{lem:res_faith_full} hold in the $F$-algebraic setting as well:
\begin{lem}\label{lem:res_faith_full-linear}
	For every morphism $f: Q_1 \rightarrow Q_2$ of pro-algebraic groups over $F$, the following hold:
	\begin{enumerate}
		\item The functor $\res_{Q_2}^{Q_1}$ is faithful and conservative.
		\item If $f$ is an epimorphism, the functor $\res_{Q_2}^{Q_1}$ is also full.
	\end{enumerate}
\end{lem}
\begin{proof}
	The first statement follows by applying Lemma~\ref{lem:res_faith_full}(1) to the group homomorphism $Q_1(\Lambda) \rightarrow Q_2(\Lambda)$ for each $\Lambda \in \CAlg_F$.
	
	For the second statement, fix two objects $(C_1,\underline{\alpha}_1), (C_2,\underline{\alpha}_2) \in \Category^{Q_2}$, and choose a morphism $\phi \colon C_1 \rightarrow C_2$ which is compatible with $\alpha_{1,\Lambda,f(q_1)}$ and $\alpha_{2,\Lambda,f(q_1)}$ for every $q_1 \in Q_1(\Lambda)$ and every $\Lambda \in \CAlg_F$.
	Since the categories $\Category_{\Lambda}$ form an fppf stack over $\CAlg_F$ and the morphism $f \colon Q_1 \rightarrow Q_2$ is an fppf epimorphism, it follows that $\phi$ is automatically compatible with $\alpha_{1,\Lambda,q_2}$ and $\alpha_{2,\Lambda,q_2}$ for every $q_2 \in Q_2(\Lambda)$ and every $\Lambda \in \CAlg_F$.
\end{proof}

If $Q$ acts algebraically on $\Category$ and $Q' \leq Q$ is a closed normal subgroup, the functor $(\for_{Q'})^Q \colon (\Category^{Q'})^Q \rightarrow \Category^Q$ admits a fully faithful section
\begin{equation*}
	\sect_{Q'}^Q \colon \Category^Q \rightarrow (\Category^{Q'})^Q,
\end{equation*}
analogously to Lemma~\ref{lem:sect_normal-sbgrp}.
If moreover $Q$ fits into a split short exact sequence of pro-algebraic groups
\begin{equation*}
	1 \rightarrow Q' \rightarrow Q \rightarrow Q'' \rightarrow 1,
\end{equation*}
we have a canonical equivalence
\begin{equation*}
	\Category^Q = (\Category^{Q'})^{Q''},
\end{equation*}
analogously to Proposition~\ref{prop:C^Q=(C^Q')^Q''-1}.

In order to state the $F$-algebraic analogue of Proposition~\ref{prop:C^Q=(C^Q')^Q''-2}, we need to spell out the precise notion of isomorphism between $F$-algebraic actions:
\begin{defn}\label{defn:iso-actions_F}
	Let $\Category$ be an $F$-linear category, and let $Q$ be a pro-algebraic group over $F$.
	\begin{enumerate}
		\item Suppose that we are given two $F$-algebraic actions of $Q$ on $\Category$;
		for every $\Lambda \in \CAlg_F$ and every $q \in Q(\Lambda)$, we write $q \cdot_A -$ and $q \cdot_B -$ for the associated functors $\Category_{\Lambda} \rightarrow \Category_{\Lambda}$ in these two actions.
		By an \textit{isomorphism} between the two actions of $Q$ on $\Category$ we mean the datum of
		\begin{itemize}
			\item for every $\Lambda \in \CAlg_F$, a $\Lambda$-linear isomorphism $\gamma_{\Lambda}$ between the two actions of the abstract group $Q(\Lambda)$ on $\Category_{\Lambda}$ (in the sense of Definition~\ref{defn:isom_actions-trivializ})
		\end{itemize}
		functorially with respect to $\Lambda \in \CAlg_F$ in the following sense:
		for every $F$-algebra homomorphism $\phi \colon \Lambda \rightarrow \Lambda'$ and every $q \in Q(\Lambda)$, the diagram of functors $\Category_{\Lambda} \rightarrow \Category_{\Lambda'}$
		\begin{equation*}
			\begin{tikzcd}
				(q \cdot_A -) \otimes_{\Lambda} \Lambda' \arrow[equal]{d} \arrow{rr}{\gamma_{\Lambda,q}} && (q \cdot_B -) \otimes_{\Lambda} \Lambda' \arrow[equal]{d} \\
				\phi(q) \cdot_A (- \otimes_{\Lambda,\phi} \Lambda') \arrow{rr}{\gamma_{\Lambda',\phi(q)}} && \phi(q) \cdot_B (- \otimes_{\Lambda,\phi} \Lambda')
			\end{tikzcd}
		\end{equation*}
		commutes.
		\item By a \textit{trivialization} of a given action of $Q$ on $\Category$ we mean an isomorphism between that action and the trivial action where $q \cdot - = \id_{\Category_{\Lambda}}$ for every $q \in Q(\Lambda)$ and every $\Lambda \in \CAlg_F$.
	\end{enumerate}
\end{defn}

For every $F$-algebraic action of $Q$ on $\Category$, the induced $Q$-action on $\Category^Q$ admits a canonical trivialization, analogously to Example~\ref{ex:trivializ}.
Moreover, an isomorphism between two actions of $Q$ on $\Category$ canonically induces an equivalence between the associated categories of $F$-algebraic homotopy-fixed points, analogously to Construction~\ref{constr:C^Q_A=C^Q_B}.

Here is the $F$-algebraic analogue of Proposition~\ref{prop:C^Q=(C^Q')^Q''-2}: 
\begin{prop}\label{prop:C^Q=(C^Q')^Q''-proalg}
	Let $Q$ be a pro-algebraic group acting on an $F$-linear category $\Category$ and fitting into the short exact sequence
	\begin{equation*}
		1 \rightarrow Q' \rightarrow Q \rightarrow Q'' \rightarrow 1.
	\end{equation*}
	Assume that we are given, in addition, a $Q''$-action on $\Category^{Q'}$, and consider the two resulting $Q$-actions  on $\Category^{Q'}$.
	Suppose that we are given an isomorphism between these two $Q$-actions which, when restricted to $Q'$, coincides with the canonical trivialization for the $Q'$-action on $\Category^{Q'}$.
	Then we get a canonical equivalence
	\begin{equation*}
		\Category^Q = (\Category^{Q'})^{Q''}.
	\end{equation*}
\end{prop}
\begin{proof}
This follows by adapting the proof of
Proposition~\ref{prop:C^Q=(C^Q')^Q''-2}
to the $F$-algebraic setting.
The only delicate point concerns the fully faithfulness of
\begin{equation*}
\res_{Q''}^Q
\colon (\Category^{Q'})^{Q''}
\rightarrow (\Category^{Q'})^Q_{(B)},
\end{equation*}
which follows from Lemma~\ref{lem:res_faith_full-linear}(2).
\end{proof}

\begin{rem}\label{rem:ind-rep_proalg}
	As in the case of abstract group actions, Proposition~\ref{prop:C^Q=(C^Q')^Q''-proalg} can be applied to the categorical induced representations.
	Note that, in order to make the discussion of Example~\ref{ex:ind-rep} work in the pro-algebraic setting, one needs to require the existence of a scheme-theoretic splitting to the projection $Q \to Q''$ respecting the unit sections. 
	This needs not exist in general, but it always exists when $Q''$ is a finite group and the ground field $F$ is algebraically closed.
\end{rem}

\subsection{The abstract fundamental sequence}\label{sect:App_Tannaka}

We further specialize the discussion to the setting of monoidal $F$-linear categories.

\begin{nota}
	Let $\Category$ be a (unitary, symmetric) monoidal $F$-linear category;
	given two objects $C_1, C_2 \in \Category$, we let $C_1 \otimes C_2$ denote their tensor product.
	For every $\Lambda \in \CAlg_F$, we regard the category $\Category_{\Lambda}$ introduced in Notation~\ref{nota:Cat_Lambda} as a (unitary, symmetric) monoidal $\Lambda$-linear category in the natural way:
	for every $C_1, C_2 \in \Category$, we set
	\begin{equation*}
		(C_1 \otimes \Lambda) \otimes (C_2 \otimes \Lambda) := (C_1 \otimes C_2) \otimes \Lambda.
	\end{equation*}
\end{nota}

\begin{defn}\label{defn:action-monoidal}
	Let $\Category$ be a (unitary, symmetric) monoidal $F$-linear category, and let $Q$ be a pro-algebraic group over $F$.
	By a \textit{(unitary, symmetric) monoidal $F$-algebraic action} of $Q$ on $\Category$ we mean the datum of
	\begin{itemize}
		\item for every $\Lambda \in \CAlg_F$, a (unitary, symmetric) monoidal $\Lambda$-linear action of the abstract group $Q(\Lambda)$ on $\Category_{\Lambda}$ (in the sense of Definition~\ref{defn:action-Q-Cat})
	\end{itemize}
	functorially with respect to $F$-algebra homomorphisms as in Definition~\ref{defn:action-linear}.
\end{defn}

There is no need to introduce another notion of equivariant objects in the monoidal setting.
Indeed:
\begin{lem}\label{lem:htpy-fixed_monoidal}
	Let $\Category$ be a (unitary, symmetric) monoidal $F$-linear category, and let $Q$ be a pro-algebraic group over $F$ acting on $\Category$ as in Definition~\ref{defn:action-monoidal}.
	Then the category of $F$-algebraic equivariant objects $\Category^Q$ from Definition~\ref{defn:htpy-fixed_linear} carries a canonical (unitary, symmetric) monoidal structure making the forgetful functor $\for_Q \colon \Category^Q \rightarrow \Category$ (unitary, symmetric) monoidal.
\end{lem}
\begin{proof}
	One defines the tensor product on $\Category^Q$ be the formula
	\begin{equation*}
		(C_1,\underline{\alpha}_1) \otimes (C_2,\underline{\alpha}_2) := (C_1 \otimes C_2, \underline{\alpha}_1 \otimes \underline{\alpha}_2)
	\end{equation*}
	where, for every $\Lambda \in \CAlg_F$ and every $q \in Q(\Lambda)$, the isomorphism
	\begin{equation*}
		(\alpha_1 \otimes \alpha_2)_{\Lambda,q} \colon q \cdot ((C_1 \otimes C_2) \otimes \Lambda) = q \cdot ((C_1 \otimes \Lambda) \otimes (C_2 \otimes \Lambda)) = q \cdot (C_1 \otimes \Lambda) \otimes q \cdot (C_2 \otimes \Lambda) \xrightarrow{\sim} (C_1 \otimes \Lambda) \otimes (C_2 \otimes \Lambda)
	\end{equation*}
	is defined as any of the two composites in the commutative diagram
	\begin{equation*}
		\begin{tikzcd}
			q \cdot (C_1 \otimes \Lambda) \otimes q \cdot (C_2 \otimes \Lambda) \arrow{d}{\id_{q \cdot C_1} \otimes \alpha_{2,\Lambda,q}} \arrow{rrr}{\alpha_{1,\Lambda,q} \otimes \id_{q \cdot C_2}} &&& (C_1 \otimes \Lambda) \otimes q \cdot (C_2 \otimes \Lambda) \arrow{d}{\id_{C_1} \otimes \alpha_{2,\Lambda,q}} \\
			q \cdot (C_1 \otimes \Lambda) \otimes (C_2 \otimes \Lambda) \arrow{rrr}{\alpha_{1,\Lambda,q} \otimes \id_{C_2}} &&& (C_1 \otimes \Lambda) \otimes (C_2 \otimes \Lambda).
		\end{tikzcd}
	\end{equation*}
	It is straightforward to check that this indeed gives a well-defined (unitary, symmetric) monoidal structure on $\Category^Q$;
	we leave the details to the interested reader.
	It is also clear that the forgetful functor canonically becomes (unitary, symmetric) monoidal in this way.
\end{proof}

We are particularly interested in monoidal actions on neutral Tannakian categories over $F$.

\begin{nota}\label{nota:Rep(K)_Lambda}
	Let $K$ be a pro-algebraic group over $F$, and let $\Rep_F(K)$ denote the (unitary, symmetric) monoidal $F$-linear category of finite-dimensional $F$-algebraic representations of $K$.
	We write a typical object $V \in \Rep_F(K)$ as a pair $(W,\rho)$ consisting of an object $W \in \vect_F$ and a morphism of pro-algebraic groups
	\begin{equation*}
		\rho \colon K \rightarrow \underline{\GL}_F(W).
	\end{equation*}
	For every $\Lambda \in \CAlg_F$, we write the corresponding object $V \otimes \Lambda \in \Rep_F(K)_{\Lambda}$ as the pair $(W \otimes \Lambda, \rho_{\Lambda})$, where
	\begin{equation*}
		\rho_{\Lambda} \colon K_{\Lambda} \rightarrow \underline{\GL}_{\Lambda}(W \otimes \Lambda)
	\end{equation*}
	denotes the morphism of group schemes over $\Lambda$ obtained by base-change from $\rho$.
\end{nota}

\begin{ex}\label{ex:Q_acts_K}
	Let $K$ be a pro-algebraic group over $F$.
	By an \textit{$F$-algebraic (right) action} of $Q$ on $K$ we mean a morphism of $F$-schemes $a \colon K \times Q \rightarrow K$ such that, for every $\Lambda \in \CAlg_F$, the map on $\Lambda$-valued points
	\begin{equation*}
		a \colon K(\Lambda) \times Q(\Lambda) \rightarrow K(\Lambda)
	\end{equation*}
	defines a right action of $Q(\Lambda)$ on $K(\Lambda)$ by group automorphisms.
	Note that, for every $\Lambda \in \CAlg_F$ and every $q \in Q(\Lambda)$, the action of $q$ canonically extends to a morphism of $\Lambda$-group schemes
	\begin{equation*}
		a_q \colon K_{\Lambda} \rightarrow K_{\Lambda}.
	\end{equation*}
	An $F$-algebraic action of $Q$ on $K$ canonically induces a (unitary, symmetric) monoidal $F$-algebraic action of $Q$ on $\Rep_F(K)$:
	for every $\Lambda \in \CAlg_F$, the action of an element $q \in Q(\Lambda)$ on the category $\Rep_F(K)_{\Lambda}$ is defined by the formula
	\begin{equation*}
		q \cdot (W \otimes \Lambda, \rho_{\Lambda}) = (W \otimes \Lambda, \rho_{\Lambda}^q),
	\end{equation*}
	where $\rho_{\Lambda}^q \colon K_{\Lambda} \rightarrow \underline{\GL}_{\Lambda}(W \otimes \Lambda)$ denotes the composite morphism $K_{\Lambda} \xrightarrow{a_q} K_{\Lambda} \xrightarrow{\rho_{\Lambda}} \underline{\GL}_{\Lambda}(W \otimes \Lambda)$.
	\qed
\end{ex}

\begin{prop}\label{prop:htpy-fixed_Tannakian}
	Let $\Tannakian$ be a neutral Tannakian category over $F$, and let $Q$ be a pro-algebraic group over $F$ acting monoidally on $\Tannakian$ as in Definition~\ref{defn:action-monoidal}.
	Then the (unitary, symmetric) monoidal category of $F$-linear equivariant objects $\Tannakian^Q$ is neutral Tannakian over $F$.
\end{prop}
\begin{proof}
	Every morphism $f \colon (V_1,\underline{\alpha}_1) \rightarrow (V_2,\underline{\alpha}_2)$ in $\Tannakian^Q$ admits a (co)kernel, obtained by equipping the (co)kernel of $f \colon V_1 \rightarrow V_2$ in $\Tannakian$ with the unique structure of equivariant object compatible with $\underline{\alpha}_1$ and $\underline{\alpha}_2$.
	This implies at once that the additive category $\Tannakian^Q$ is abelian and that the forgetful functor $\Tannakian^Q \rightarrow \Tannakian$ is exact.
	It is also clear that the monoidal structure obtained in Lemma~\ref{lem:htpy-fixed_monoidal} makes $\Tannakian^Q$ an abelian tensor category.
	Moreover, any object $(V,\underline{\alpha}) \in \Tannakian^Q$ admits a strong dual $(V^{\lor},\underline{\alpha}^{\lor})$, where $V^{\lor}$ denotes the strong dual of $V$ in $\Tannakian$ and, for every $\Lambda \in \CAlg_F$ and every $q \in Q(\Lambda)$, the isomorphism
	\begin{equation*}
		\alpha^{\lor}_{\Lambda,q} \colon q \cdot (V^{\lor} \otimes \Lambda) = q \cdot (V \otimes \Lambda)^{\lor} = (q \cdot (V \otimes \Lambda))^{\lor} \xrightarrow{\sim} (V \otimes \Lambda)^{\lor} = V^{\lor} \otimes \Lambda
	\end{equation*}
	is defined as the inverse-transpose of $\alpha_{\Lambda,q}$.
	Finally, one obtains a fibre functor for $\Tannakian^Q$ by composing the forgetful functor $\for_Q \colon \Tannakian^Q \to \Tannakian$ with any given fibre functor for $\Tannakian$.
\end{proof}

\begin{ex}\label{ex:Q_acts_vectF}
	Let $Q$ be any pro-algebraic group over $F$, and let it act trivially on the neutral Tannakian category $\vect_F$;
	this is induced by the trivial action of $Q$ on its Tannaka dual $\Spec(F)$ via Example~\ref{ex:Q_acts_K}.
	Then the associated category of $F$-algebraic equivariant objects $\vect_F^Q$ is canonically equivalent to $\Rep_F(Q)$ as an abelian tensor category.
	\qed
\end{ex}

The construction of $F$-algebraic equivariant objects under a given pro-algebraic group $Q$ is clearly compatible with $Q$-equivariant $F$-linear functors.
In the Tannakian setting, this observation translates as follows:
\begin{constr}\label{constr:Rep(Q)_to_Tann^Q}
	Any $Q$-equivariant $F$-linear tensor functor between neutral Tannakian categories $t \colon \Tannakian' \rightarrow \Tannakian$ induces an $F$-linear tensor functor
	\begin{equation*}
		t^Q \colon {\Tannakian'}^Q \rightarrow \Tannakian^Q
	\end{equation*}
	making the diagram
	\begin{equation*}
		\begin{tikzcd}
			{\Tannakian'}^Q \arrow{rr}{t^Q} \arrow{d}{\for_Q} && \Tannakian^Q \arrow{d}{\for_Q} \\
			\Tannakian' \arrow{rr}{t} && \Tannakian
		\end{tikzcd}
	\end{equation*}
	commute.
	Applying this to the case where $\Tannakian' = \vect_F$ (with trivial action by $Q$, as in Example~\ref{ex:Q_acts_vectF}) and $t$ is the unique $F$-linear tensor functor $\vect_F \rightarrow \Tannakian$, we obtain a canonical functor
	\begin{equation*}
		\Rep_F(Q) = \vect_F^Q \rightarrow \Tannakian^Q.
	\end{equation*}
    Note that its essential image is stable under subobjects, since the same holds for the underlying functor $\vect_F \rightarrow \Tannakian$.
    As noted in~\cite[Prop.~2.11]{Jacobsen}, this automatically implies that it is fully faithful.
    \qed
\end{constr}

We move on to the last main result of the appendix.

\begin{constr}\label{constr:1KGQ1}
	Let $\Tannakian$ be a neutral Tannakian category over $F$;
	fix a fibre functor $\omega \colon \Tannakian \rightarrow \vect_F$, and let $K$ denote the associated Tannaka dual group.
	Let $Q$ be a pro-algebraic group over $F$ acting on $\Tannakian$ (as in Definition~\ref{defn:action-monoidal}), and regard $\Tannakian^Q$ as a neutral Tannakian category as in Proposition~\ref{prop:htpy-fixed_Tannakian};
	let $G$ denote its Tannaka dual group with respect to the fibre functor $\omega \circ \for_Q \colon \Tannakian^Q \rightarrow \vect_F$.
	Under Tannaka duality, the forgetful functor $\for_Q \colon \Tannakian^Q \rightarrow \Tannakian$ corresponds to a morphism of pro-algebraic groups
	\begin{equation*}
		K \rightarrow G.
	\end{equation*}
	Similarly, the tensor functor $t^Q \colon \Rep_F(Q) \rightarrow \Tannakian^Q$ described in Construction~\ref{constr:Rep(Q)_to_Tann^Q} corresponds to a morphism of pro-algebraic groups
	\begin{equation*}
		G \rightarrow Q,
	\end{equation*}
    which is faithfully flat by~\cite[Prop.~2.21(b)]{deligne-milne}.
   In fact, $t^Q$ identifies $\Rep_F(Q)$ with the full Tannakian subcategory of $\Tannakian^Q$ consisting of those objects $V = (W,\underline{\alpha})$ whose underlying object $W \in \Tannakian$ is a trivial $K$-representation.
Under Tannaka duality, this means that the composite homomorphism
\begin{equation*}
K \to G \to Q
\end{equation*}
is trivial.
\qed
\end{constr}

The above discussion leads one to wonder about the exactness properties of the sequence
\begin{equation}\label{ses:1KGQ1}
	1 \rightarrow K \rightarrow G \rightarrow Q \rightarrow 1.
\end{equation}
This seems a delicate question in general, especially for what concerns the injectivity of $K \rightarrow G$.
The following particular case suffices for our main applications:
\begin{prop}\label{prop:KrtimesQ}
	Keep the notation and assumptions of Construction~\ref{constr:1KGQ1}.
	Suppose that, under the equivalence $\Tannakian = \Rep_F(K)$ defined by $\omega$, the $F$-algebraic $Q$-action on $\Tannakian$ is induced by an $F$-algebraic right action of $Q$ on $K$ (in the sense of Example~\ref{ex:Q_acts_K}); 
	form the associated semi-direct product $K \rtimes Q$.
	Then there exists a canonical isomorphism of pro-algebraic groups
	\begin{equation*}
		G \xrightarrow{\sim} K \rtimes Q
	\end{equation*}
	under which the sequence \eqref{ses:1KGQ1} gets identified with the split short exact sequence
	\begin{equation*}
		1 \rightarrow K \rightarrow K \rtimes Q \rightarrow Q \rightarrow 1.
	\end{equation*}
\end{prop}
\begin{proof}
	For notational clarity, throughout the proof we identify $\Tannakian$ with $\Rep_F(K)$, and we write $\Rep_F(G)$ as $\Rep_F(K)^Q$.
	By Tannaka duality, in order to prove the thesis it suffices to construct an equivalence of neutral Tannakian categories
	\begin{equation*}
		\Rep_F(K \rtimes Q) = \Rep_F(K)^Q
	\end{equation*}
	making the diagram
	\begin{equation}\label{dia:Rep(K)^Q-G}
		\begin{tikzcd}
			&& \Rep_F(K)^Q \arrow{drr}{\for_Q} \arrow[equal]{dd} \\
			\Rep_F(Q) \arrow{urr}{t^Q} \arrow{drr}{\res_Q^{K \rtimes Q}} &&&& \Rep_F(K) \\
			&& \Rep_F(K \rtimes Q) \arrow{urr}{\res_{K \rtimes Q}^K}
		\end{tikzcd}
	\end{equation}
	commute.
	Using the canonical equivalences provided by Example~\ref{ex:Q_acts_vectF}, we can write the sought-after equivalence in the form
	\begin{equation*}
		\vect_F^{K \rtimes Q} = (\vect_F^K)^Q,
	\end{equation*}
	which makes it easier to construct it explicitly.
	Indeed, consider the canonical action of $K \rtimes Q$ on $\vect_F^K$ (defined as in Construction~\ref{constr:Q_acts_C^Q'}) and the associated functor
	\begin{equation*}
		\sect_K^{K \rtimes Q} \colon \vect_F^K \rightarrow (\vect_F^K)^{K \rtimes Q}.
	\end{equation*}
	By the $F$-algebraic version of Proposition~\ref{prop:C^Q=(C^Q')^Q''-1}, we know that the composite functor
	\begin{equation*}
		\vect_F^{K \rtimes Q} \xrightarrow{\sect_K^{K \rtimes Q}} (\vect_F^K)^{K \rtimes Q} \xrightarrow{\res_{K \rtimes Q}^Q} (\vect_F^K)^Q
	\end{equation*}
	is an equivalence.
    It is easy to check that this functor indeed makes the diagram~\eqref{dia:Rep(K)^Q-G} commute and, in fact, it is characterized by this property;
    we leave the details to the interested reader.
\end{proof}

\begin{rem}\label{rem:Q_acts_K_conjugation}
	In the setting of Proposition~\ref{prop:KrtimesQ}, the original $F$-algebraic right action of $Q$ on $K$ can be recovered as the conjugation action induced by the short exact sequence~\eqref{ses:1KGQ1}.
\end{rem}

\begin{rem}
If one knows that every object of $\Tannakian$
is a subobject of some object in the image of $\Tannakian^Q$,
then one can deduce the exactness of
the entire sequence~\eqref{ses:1KGQ1},
even without the technical assumption
of Proposition~\ref{prop:KrtimesQ},
as an application of~\cite[Prop.~A.13]{DAE22}.
\end{rem}

We conclude by providing a useful criterion in the direction of Proposition~\ref{prop:KrtimesQ}.
In a nutshell, a given monoidal $F$-algebraic action of $Q$ on $\Rep_F(K)$ is induced by an $F$-algebraic action of $Q$ on $K$ precisely when it is an action within the $2$-category of neutralized (as opposed to neutral) Tannakian categories - that is, when it is compatible with the natural fibre functor of $\Rep_F(K)$.
In the language of $F$-algebraic actions employed here, this feature comes down to the following notion:
\begin{defn}\label{defn:concretization}
	Let $\Tannakian$ be a neutral Tannakian category over $F$;
	fix a fibre functor $\omega \colon \Tannakian \rightarrow \vect_F$, and let $K$ denote the associated Tannaka dual group.
	Let $Q$ be a pro-algebraic group over $F$ acting monoidally on $\Tannakian$ (as in Definition~\ref{defn:action-monoidal}).
	By a \textit{concretization} $\gamma$ for the $Q$-action on $\Tannakian$ with respect to $\omega$ we mean the datum of
	\begin{itemize}
		\item for every $\Lambda \in \CAlg_F$ and every $q \in Q(\Lambda)$, a natural isomorphism of functors $\Rep_F(K)_{\Lambda} \rightarrow \modules_{\Lambda}$
		\begin{equation*}
			\gamma_q \colon \omega_{\Lambda} \circ (q \cdot -) \xrightarrow{\sim} \omega_{\Lambda}
		\end{equation*}
	\end{itemize}
	compatibly with composition in $Q$ and functorially with respect to $F$-algebra homomorphism, in the following sense:
	for every $\Lambda \in \CAlg_F$ and every $q_1, q_2 \in Q(\Lambda)$, the diagram of functors $\Rep_F(K)_{\Lambda} \rightarrow \modules_{\Lambda}$
	\begin{equation*}
		\begin{tikzcd}
			\omega_{\Lambda} \circ (q_1 \cdot (q_2 \cdot -)) \arrow{rr}{\gamma_{q_1}} \arrow[equal]{d} && \omega_{\Lambda} \circ (q_2 \cdot -) \arrow{d}{\gamma_{q_2}} \\
			\omega_{\Lambda} \circ (q_1 q_2 \cdot -) \arrow{rr}{\gamma_{q_1 q_2}} && \omega_{\Lambda}
		\end{tikzcd}
	\end{equation*}
	commutes and, for every $F$-algebra homomorphism $\phi \colon \Lambda \rightarrow \Lambda'$ and every $q \in Q(\Lambda)$, the diagram of functors $\Rep_F(K)_{\Lambda} \rightarrow \modules_{\Lambda'}$
	\begin{equation*}
		\begin{tikzcd}
			\omega_{\Lambda'} \circ ((q \cdot -) \otimes_{\Lambda} \Lambda') \arrow[equal]{rr} \arrow[equal]{d} && (\omega_{\Lambda} \circ (q \cdot -)) \otimes_{\Lambda} \Lambda' \arrow{d}{\gamma_q \otimes_{\Lambda} \Lambda'} \\
			\omega_{\Lambda'} \circ (\phi(q) \cdot (- \otimes_{\Lambda} \Lambda')) \arrow{rr}{\gamma_{\phi(q)}} && \omega_{\Lambda'} \circ (- \otimes_{\Lambda} \Lambda')
		\end{tikzcd}
	\end{equation*}
	commutes.
\end{defn}

\begin{rem}
	Suppose that we are given two $F$-algebraic actions of $Q$ on $\Tannakian$ as well as an isomorphism $\gamma$ between them (as in Definition~\ref{defn:iso-actions_F}(1));
	in addition, suppose that we are given a concretization $\gamma_A$ for the first action.
	Then we obtain a concretization $\gamma_B$ for the second action by defining $\gamma_{B,\Lambda,q}$ as the composite
	\begin{equation*}
		\omega_{\Lambda} \circ (q \cdot_B -) \xrightarrow{\omega_{\Lambda}(\gamma_q^{-1})} \omega_{\Lambda} \circ (q \cdot_A -) \xrightarrow{\gamma_{A,\Lambda,q}} \omega_{\Lambda}
	\end{equation*}
	for every $\Lambda \in \CAlg_F$ and every $q \in Q(\Lambda)$:
	indeed, the compatibility of $\gamma_B$ with composition in $Q$ and $F$-algebra homomorphisms follows from the corresponding compatibility conditions for $\gamma_A$ and for $\gamma$.
\end{rem}

Here is the promised criterion:
\begin{lem}\label{lem:concretization}
	Keep the notation and assumptions of Construction~\ref{constr:1KGQ1}.
	The $Q$-action on $\Tannakian$ admits a concretization if and only if, up to isomorphism of actions (in the sense of Definition~\ref{defn:iso-actions_F}(1)), it is induced by an $F$-algebraic right action of $Q$ on $K$ (in the sense of Example~\ref{ex:Q_acts_K}).
\end{lem}
\begin{proof}
	Suppose that the $Q$-action on $\Tannakian$ is induced by an $F$-algebraic right action of $Q$ on $K$.
	Then one can trivially define a concretization $\gamma$ by taking $\gamma_q$ to be the identity natural transformation for every $q \in Q(\Lambda)$ and every $\Lambda \in \CAlg_F$:
	this is legitimate in view of the explicit formula for the $Q$-action on $\Rep_F(K)$ given in Example~\ref{ex:Q_acts_K}.

	Conversely, suppose that the action of $Q$ on $\Tannakian$ admits a concretization $\gamma$.
	By~\cite[Prop.~2.8]{deligne-milne}, for every $\Lambda \in \CAlg_F$, the group $K(\Lambda)$ is naturally identified with the group of tensor-automorphisms of the fibre functor $\omega_{\Lambda} \colon \Tannakian_{\Lambda} \rightarrow \modules_{\Lambda}$, functorially with respect to $\Lambda \in \CAlg_F$.
	For every $\Lambda \in \CAlg_F$, we define a map
	\begin{equation}\label{eq:a_{Lambda,q}-concretization}
		a_{\Lambda} \colon K(\Lambda) \times Q(\Lambda) \rightarrow K(\Lambda), \quad (k,q) \mapsto a_{\Lambda,q}(k)
	\end{equation}
	by the formula
	\begin{equation*}
		a_{\Lambda,q}(k) := \gamma_q \circ (k (q \cdot -)) \circ \gamma_q^{-1},
	\end{equation*}
	where $k (q \cdot -)$ denotes the tensor-automorphism of the composite functor $\omega \circ (q \cdot -)$ induced by $k$.
	We claim that the maps~\eqref{eq:a_{Lambda,q}-concretization} define an $F$-algebraic right action $a \colon K \times Q \rightarrow K$ as in Example~\ref{ex:Q_acts_K}.
	To this end, we need to check that each map~\eqref{eq:a_{Lambda,q}-concretization} defines an action of the abstract group $Q(\Lambda)$ on $K(\Lambda)$ and that they are functorial with respect to $\Lambda \in \CAlg_F$.
	But these two facts follow immediately from the two coherence conditions in Definition~\ref{defn:concretization}.
	
	Now we have two distinct $Q$-actions on $\Tannakian$:
	the original one, and the one induced by the $Q$-action of $K$ via Example~\ref{ex:Q_acts_K};
	for every $\Lambda \in \CAlg_F$ and every $q \in Q(\Lambda)$, let us write $q \cdot_A -$ and $q \cdot_B -$ for the corresponding functors $\Category_{\Lambda} \to \Category_{\Lambda}$ in these two actions, respectively.
	For every $\Lambda \in \CAlg_F$ and every $q \in Q(\Lambda)$, consider the isomorphism of functors $\Tannakian_{\Lambda} \to \modules_{\Lambda}$
	\begin{equation*}
		\gamma_{\Lambda,q} \colon \omega_{\Lambda} \circ (q \cdot_A -) \xrightarrow{\sim} \omega_{\Lambda} = \omega_{\Lambda} \circ (q \cdot_B -),
	\end{equation*}
	where the last equality comes from the explicit description of the second $Q$-action on $\Tannakian$ provided in Example~\ref{ex:Q_acts_K}.
	We claim that, for every $V \in \Tannakian$, the isomorphism in $\mod_{\Lambda}$
	\begin{equation*}
		\gamma_{\Lambda,q,V} \colon \omega_{\Lambda}(q \cdot_A (V \otimes \Lambda)) \xrightarrow{\sim} \omega_{\Lambda}(q \cdot_B (V \otimes \Lambda))
	\end{equation*}  
	comes from an isomorphism in $\Tannakian_{\Lambda}$
	\begin{equation}\label{eq:iso_actions-concretization}
		q \cdot_A (V \otimes \Lambda) \xrightarrow{\sim} q \cdot_B (V \otimes \Lambda).
	\end{equation}
	To see this, it is convenient to identify $\Tannakian$ with $\Rep_F(K)$ via the fibre functor $\omega \colon \Tannakian \rightarrow \vect_F$, which allows us to write $V = (W,\rho)$ as in Notation~\ref{nota:Rep(K)_Lambda};
	for sake of simplicity, let us also write $q \cdot_A (V \otimes \Lambda) = (W' \otimes \Lambda, \rho'_{\Lambda})$ for some $W' \in \vect_F$.
	Then our claim amounts to the fact that, for every $F$-algebra homomorphism $\phi \colon \Lambda \rightarrow \Lambda'$ and every $k \in K(\Lambda')$, the diagram in $\modules_{\Lambda'}$
	\begin{equation*}
		\begin{tikzcd}
			W' \otimes \Lambda' \arrow{d}{\rho'_{\Lambda'}(k)} \arrow{rr}{\gamma_{\Lambda',\phi(q),V}} && W \otimes \Lambda' \arrow{d}{\rho_{\Lambda'}^{\phi(q)}(k)} \\
			W' \otimes \Lambda' \arrow{rr}{\gamma_{\Lambda',\phi(q),V}} && W \otimes \Lambda'
		\end{tikzcd}
	\end{equation*}
	commutes, which follows from the very definition of the homomorphism $\rho_{\Lambda}^q \colon K_{\Lambda} \rightarrow \underline{\GL}_{\Lambda}(W \otimes \Lambda)$.
	By construction, the isomorphisms \eqref{eq:iso_actions-concretization} are coherent with composition in $Q$ and with $F$-algebra homomorphisms, hence they define the sought-after isomorphism of actions. 
\end{proof}

\begin{rem}
	A general $F$-algebraic action of $Q$ on $\Tannakian$ needs not admit a concretization, not even when $F$ is algebraically closed.
    In fact, if $F$ is algebraically closed, then for every $q \in Q(F)$ the two fibre functors $\omega \circ (q \cdot -)$ and $\omega$ are isomorphic, but not canonically so in general.
	In particular, it is not clear that one can choose the various isomorphisms defining a concretization compatibly with composition in $Q(F)$, as required in Definition~\ref{defn:concretization}.
\end{rem}

\printbibliography[heading=bibintoc]

\end{document}